\documentclass{amsart}
\usepackage{blindtext}

\makeatletter
\renewcommand\part{%
   \if@noskipsec \leavevmode \fi
   \par
   \addvspace{4ex}%
   \@afterindentfalse
   \secdef\@part\@spart}

\def\@part[#1]#2{%
    \ifnum \c@secnumdepth >\m@ne
      \refstepcounter{part}%
      \addcontentsline{toc}{part}{\thepart\hspace{1em}#1}%
    \else
      \addcontentsline{toc}{part}{#1}%
    \fi
    {\parindent \z@ \raggedright
     \interlinepenalty \@M
     \normalfont
     \ifnum \c@secnumdepth >\m@ne
       \Large\bfseries \partname\nobreakspace\thepart
       \par\nobreak
     \fi
     \Large \bfseries #2%
     \par}%
    \nobreak
    \vskip 3ex
    \@afterheading}
\def\@spart#1{%
    {\parindent \z@ \raggedright
     \interlinepenalty \@M
     \normalfont
     \normalsize \bfseries #1\par}%
     \nobreak
     \vskip 3ex
     \@afterheading}
\makeatother

\usepackage{dsfont}
\usepackage[usenames,dvipsnames,svgnames,table]{xcolor}
\usepackage[utf8]{inputenc}
\usepackage[T1]{fontenc}
\usepackage{ tipa }
\usepackage{helvet}
\usepackage{color}
\usepackage{mathrsfs}  
\usepackage{stmaryrd}  
\usepackage{amsmath,amsxtra,amsthm,amssymb,xr,enumerate,enumitem,fullpage,graphicx,mathtools}
\usepackage{fdsymbol}
\usepackage{bold-extra}
\usepackage{tikz-cd}
\usepackage{graphicx}
\usepackage{mathtools}

\usepackage[symbol]{footmisc}

\usepackage[all]{xy}
\usepackage[bbgreekl]{mathbbol}
\usepackage{bbm}
\usepackage{enumitem}

\DeclareMathSymbol{\invques}{\mathord}{operators}{`>}
\DeclareUnicodeCharacter{00BF}{\tmquestiondown}
\DeclareRobustCommand{\tmquestiondown}{%
  \ifmmode\invques\else\textquestiondown\fi
}
\usepackage{verbatim}
\numberwithin{equation}{section}

\makeatletter
\newcommand{\mylabel}[2]{#2\def\@currentlabel{#2}\label{#1}}
\makeatother

\newtheorem{theorem}{Theorem}[section]
\newtheorem{lemma}[theorem]{Lemma}

\newtheorem{proposition}[theorem]{Proposition}
\newtheorem{corollary}[theorem]{Corollary}
\newtheorem{defn}[theorem]{Definition}

\newtheorem{remark}[theorem]{Remark}

\newtheorem{assumption}[theorem]{Assumption}

\newcommand{\Gal}{\operatorname{Gal}}

\newcommand{\CC}{\mathbb{C}}

\newcommand{\cB}{\mathcal{B}}

\newcommand{\QQ}{\mathbb{Q}}

\newcommand{\Zp}{\mathbb{Z}_p}
\newcommand{\ZZ}{\mathbb{Z}}

\newcommand{\g}{\mathbf{g}}

\newcommand{\PP}{\mathfrak{P}}
\newcommand{\fp}{\mathfrak{p}}
\newcommand{\fq}{\mathfrak{q}}

\newcommand{\cH}{\mathcal{H}}
\newcommand{\cO}{\mathcal{O}}

\newcommand{\cyc}{\textup{cyc}}

\newcommand{\LL}{\Lambda}

\newcommand{\f}{\textup{\bf f}}

\newcommand{\lra}{\longrightarrow}
\newcommand{\ra}{\lra}

\newcommand{\cP}{\mathcal{P}}

\definecolor{Green}{rgb}{0.0, 0.5, 0.0}

\newcommand{\p}{\mathfrak{p}}
\newcommand{\q}{\mathfrak{q}}

\newcommand{\m}{\mathfrak{m}}

\newcommand{\cD}{\mathcal{D}}

\newcommand{\cW}{\mathcal{W}}

\newcommand{\cR}{\mathcal{R}}

\newcommand{\Cp}{\mathbb{C}_p}

\newcommand{\bbR}{\mathbb{R}}
\newcommand{\boldD}{\textup{\bf D}}
\newcommand{\boldR}{\textup{\bf R}}

\newcommand{\boldV}{\textup{\bf V}}
\newcommand{\boldG}{\textup{\bf G}}
\newcommand{\boldT}{\textup{\bf T}}
\newcommand{\bbLL}{\mathbb{L}}
\newcommand{\frakm}{\mathfrak{m}}
\newcommand{\bfchi}{\bbchi}

\usepackage{hyperref}

 \definecolor{pAlgae}{RGB}{87,115,135}
\definecolor{airforceblue}{rgb}{0.36, 0.54, 0.66}
	\definecolor{bondiblue}{rgb}{0.0, 0.58, 0.71}
\definecolor{britishracinggreen}{rgb}{0.0, 0.26, 0.15}
\definecolor{camouflagegreen}{rgb}{0.47, 0.53, 0.42}
\definecolor{darkcyan}{rgb}{0.0, 0.55, 0.55}

\hypersetup{
    colorlinks = true,
    linkcolor=blue,
     filecolor=blue,
     citecolor = darkcyan,      
     urlcolor=cyan,
    linkbordercolor = {white},
}

\begin{document}

\author{K\^az\i m B\"uy\"ukboduk}
\address{K\^az\i m B\"uy\"ukboduk\newline UCD School of Mathematics and Statistics\\ University College Dublin\\ Belfield\\Dublin 4\\Ireland}
\email{kazim.buyukboduk@ucd.ie}

\author{Manisha Ganguly}
\address{Manisha Ganguly\newline UCD School of Mathematics and Statistics\\ University College Dublin\\ Belfield\\Dublin 4\\Ireland}
\email{manisha.ganguly@ucdconnect.ie}

\title{Functional equations of algebraic Rankin--Selberg $p$-adic $L$-functions}

\begin{abstract}
This article presents an approach to the algebraic functional equation for Selmer complexes, which in turn have applications in the Iwasawa theoretic study of Rankin–Selberg products of the Hida and Coleman families. Our treatment establishes the functional equation for algebraic $p$-adic $L$-functions (that are given in terms of characteristic ideals of Selmer groups, which arise as the cohomology of appropriately defined Selmer complexes in degree $2$). This is achieved by recovering the characteristic ideal as the determinant of the said Selmer complex, once we prove (under suitable but rather mild) hypotheses that the Selmer complex in question is perfect with amplitude $[1,2]$, and its cohomology is concentrated in degree-2. The perfectness of these Selmer complexes turns out to be a delicate problem, and the required properties require a study of Tamagawa factors in families, which may be of independent interest.
\end{abstract}

\maketitle
\tableofcontents

\section{Introduction}
Our primary goal in this paper is to develop a general approach to prove the functional equation for algebraic $p$-adic $L$-functions associated with families of Galois representations. Our starting point is Nekovář's philosophy that these should be recovered as the determinants of Selmer complexes. One of our key contributions is the proof that this is indeed valid in great level of generality, namely that the determinants of appropriately defined Selmer complexes coincide with their classical counterparts (that are often defined as the characteristic ideal of Selmer groups). This requires a careful analysis of the properties of Selmer complexes and occupies the first half of our paper. In the second half\footnote{The reader who is more interested in these arithmetic applications may as well treat the first half as a black box and skip directly to the second half.}, we apply our general results in specific contexts (including the families of symmetric squares of modular forms). 

Although our main results are very general (cf. Theorem~\ref{thm:char_ideals_for_selCom_and_dual_selCom_H_A}), our main motivation stemmed from the papers~\cite{Loeffler_2016,iwasa_sym_nonordinary}, and the forthcoming work of Kundu--Ray--Vigni~\cite{FuncEqSymHida}. In these papers, the results that the authors prove towards the Iwasawa main conjectures are valid only for branches of the cyclotomic Iwasawa algebra with prescribed parity. Thanks to the functional equation we establish in this article (cf. Theorem~\ref{thm:H^1_0_Hida_family_Sym2}), the required parity condition can be eliminated. In this regard, our results, when combined with the functional equations for $p$-adic analytic $L$-functions, yield a strengthening of the said results towards Iwasawa's main conjectures for symmetric squares of modular forms (and the families of such).

We highlight one novel aspect in our approach, which crucially relies on Nekov\'a\v{r}'s interpretation of algebraic $p$-adic $L$-functions: In earlier works, functional equations of algebraic $p$-adic $L$-functions were established using generalized Cassels--Tate pairings, and interpolations of such. As this is not possible in our context, we take an alternative approach and redefine (following Nekov\'a\v{r}) algebraic $p$-adic $L$-functions as the determinants of Selmer complexes. The required functional equation follows from duality theorems for Selmer complexes (cf. \S\ref{subsec_duality_over_H_A}) in a great level of generality. To tie the loose ends, one then needs to check that these two definitions of algebraic $p$-adic $L$-functions agree. We prove that this is the case under suitable hypotheses in\S\ref{chap:Selmer_Complexes}. We verify in \S\ref{chap:5_Rankin-Selberg_product} that these hypotheses hold in the context of Rankin--Selberg products under very mild assumptions. We believe that the verification of these properties is interesting in their own right: Among other things, this step requires a study of Tamagawa numbers away from $p$ in Hida families (cf. \S\ref{subsec_4_7_11_2024_07_12} and \S\ref{subsubsec_2024_07_12_1227}).

\subsection{Detailed outline of the main results and layout}
As the final section of our introduction section (see \S\ref{subsubsec_121_2024_11_03_826} below), we briefly review modular forms, discuss Galois representations attached to them, and record the basic properties of these Galois representations that we shall need in due course. 

The main body of our paper starts with the treatment of the ordinary scenario in Part \ref{Part_ord}, where we systematically rely on Nekov\'a\v{r}'s duality theory for his Selmer complexes in the context of families of Rankin--Selberg products. These duality theorems a priori yield a functional equation up to local error terms (cf. \S\ref{sec:duality_sel_compl}). In \S\ref{subsubsec_2024_11_04_0603} (see also \S\ref{subsubsec_2024_07_12_1227}), we explain that these error terms can be controlled in terms of Tamagawa factors. When the said Tamagawa factors are trivial, we prove that the relevant Selmer complexes are perfect, and combining this fact with the results of \cite{KLZRecLaw} with their Beilinson--Flach element Euler system, we prove that the determinants of our Selmer complexes coincide with the characteristic ideals of Selmer groups (and therefore deserve to be called ``algebraic $p$-adic $L$-functions''). The required functional equation (Theorems~\ref{thm:Func_eq_hida} and \ref{thm_2_21_2024_11_04}) then follow from Nekov\'a\v{r}'s duality theorems.

We begin Part~\ref{part_general_nonord} (which spans Sections~\ref{chap:Characteristic_Ideal_and_Determinant} through \ref{chap:Selmer_Complexes}, and is dedicated to an abstract study of Selmer complexes over in the non-ordinary setting, and their duality theory) with the introduction of the concepts of the characteristic ideals and the determinants of coadmissible torsion modules over Fr\'echet–Stein algebras in \S\ref{chap:Characteristic_Ideal_and_Determinant}. This involves a study of various rigid analytic spaces (e.g. $p$-adic annuli and rigid analytic functions on these). While the definition over Perrin-Riou's ring $\cH(\ZZ_p^\times):=A\widehat{\otimes }\cH(\ZZ_p^\times)$ of tempered distributions is borrowed from Pottharst (cf. \cite[p. 5]{pottharst}), the definition of (cylindrical) characteristic ideal (cf. Definition~\ref{defn:char_Ideal_coadmissible_H_A_module}) over its higher dimensional variants (which we denote by $\cH_A(\ZZ_p^\times)$ in the main text, where $A$ is a Tate algebra) is new.

In \S\ref{chap:Local_cohomology_and_phi_gamma_module}, we summarize basic input from \cite{p-adicHodgetheorybenoisheights}, and provide an introduction to the theory of $(\varphi, \Gamma)$-modules (over affinoids), relating them to $p$-adic Hodge theory and local (Iwasawa) Galois cohomology. This sets the stage for the detailed study of Selmer complexes in the next section.

In \S\ref{chap:Selmer_Complexes}, we define Selmer complexes over various analytic rings and discuss Selmer complexes over an affinoid algebra $A$, the Perrin-Riou ring $\cH(\ZZ_p^\times)$ of tempered distributions on $\ZZ_p^\times$, and finally over the rings $\cH_A(\ZZ_p^\times)$ where $A$ is an affinoid algebra. We establish, relying on the results from \S\ref{chap:Local_cohomology_and_phi_gamma_module}, duality theorems and use these, in turn, to derive an algebraic functional equation for the algebraic $p$-adic $L$-functions (that are defined as the determinants of appropriate Selmer complexes) under specific technical conditions. The main result of this chapter is the functional equation we prove in Theorem~\ref{thm:char_ideals_for_selCom_and_dual_selCom_H_A} for the algebraic $p$-adic $L$-functions over $\cH_A(\ZZ_p)$, where $A$ is a Tate algebra of arbitrary dimension $n$ (including $n=0$).

In \S\ref{chap:5_Rankin-Selberg_product}, where Part~\ref{part_arithmetic_nonord} begins,  we obtain the main results of the present paper, in the context of Rankin--Selberg products of Coleman families. 

This section begins with a review of the basic properties of Coleman families of modular forms and associated Galois representations (cf. \S\ref{sec_4_4_2024_07_12_1610}). Using the results of \cite{LZ0}, one proves that the cohomology of the Selmer complexes associated with Rankin--Selberg products are concentrated in degree $2$ (cf. Theorem~\ref{thm:Vanish_H^1_torsion-ness_of_H^2_for_specialization}, Theorem~\ref{thm:thm:H^1_zero_torsion_H^2_ranking_sel_ColemanFamily_1_vari}, and Theorem~\ref{thm:H^1_zero_torsion_H^2_ranking_sel_ColemanFamily}). Building on these results, combined with the perfectness results for these Selmer complexes (under a certain ``Tamagawa condition''), we prove that the determinants of our Selmer complexes coincide with the characteristic ideals of Selmer groups as in the ordinary case. With these inputs at hand, we prove the functional for algebraic $p$-adic $L$-functions associated with Rankin--Selberg products (cf. Theorems~\ref{thm:FuncEq_SelCom_Rankin_SelbergProducts_0_vari}, \ref{thm:FuncEq_SelCom_Rankin_SelbergProducts_1_vari}, \ref{thm:FuncEq_SelCom_Rankin_SelbergProducts_2_vari} and \ref{thm:H^1_zero_torsion_H^2_ranking_sel_SymColemanFamily_1_vari}), which are the main results of the present paper.

\subsection{General notation} We record here the notation that we will commonly use across all three parts of our paper. 

Let $\CC_p$ be the $p$-adic completion of $\overline{\QQ}_p$. We denote by $v_p: \CC_p \to \bbR \cup \{\infty\}$ the $p$-adic valuation on $\CC_p$ normalized so that $v_p(p)=1$. We put $|x|_p={p}^{-v_p(x)}$. We fix throughout an isomorphism $\jmath_p: \mathbb{C}\simeq \mathbb{C}_p$, as well as an embedding $\iota_\infty: \overline{\QQ}\hookrightarrow \mathbb{C}$, and put $\iota_p:=\jmath_p\circ\iota_\infty$. 

For any finite extension $F$ of $\QQ_\ell$ (where $\ell$ is any prime), we let $F^{\rm ur}$ denote the maximal unramified extension of $F$. We fix a system of primitive $p^n$-th roots of unity $\epsilon=(\zeta_{p^n})_{n\geq0}$ such that $\zeta^p_{p^{n+1}}=\zeta_{p^n}$ for all $n \geq 0$. For any field $L$, we put $L^{\rm cyc}:=\bigcup_{n=0}^{\infty}L(\zeta_{p^n})$, $H_L={\rm Gal}(\overline{L}/L^{\rm cyc})$, $\Gamma_L={\rm Gal}(L^{\rm cyc}/L)$, and finally let $\chi_L:\Gamma_L \to \ZZ^\times_p$ the $p$-adic cyclotomic character. 

We set $L_\infty=(L^{\rm cyc})^{\Delta_L}$, where $\Delta_L={\rm Gal}(L(\zeta_p)/L)$ (which we may and we will identify as a subgroup of $\Gamma_L$). We put 
$ \Gamma^0_L:={\rm Gal}(L_\infty/L)$. In the remainder of this paragraph, let us assume that $L$ is a finite extension of $\QQ$ or $\QQ_p$. In this case, $\Gamma^0_L\cong 1+p\ZZ_p \simeq \ZZ_p $. We denote by $L_n$ the unique subextension of $L_\infty$ of degree $p^n$. 

Let $E$ denote a fixed finite extension of $\QQ_p$ (which will play the role as the coefficients of the Galois modules we shall consider, and we shall enlarge it in due course as necessary) and let $\cO=\cO_E$ be its ring of integers. We denote by $\Lambda_{\cO}(\Gamma_L^0)=\cO[[\Gamma^0_K]]$ the Iwasawa algebra of $\Gamma^0_L$ with coefficients in $\cO$. A choice of a topological generator $\gamma_L$ of $\Gamma^0_L$ fixes an isomorphism $\Lambda_{\cO}(\Gamma_L^0) \cong \cO[[X]]$, which is given by $\gamma_L \mapsto X+1.$ 

\subsection{Preliminaries on modular forms}
\label{subsubsec_121_2024_11_03_826}
Let $f\in S_{k_f+2}(\Gamma_1(N_f),\varepsilon_f)$ be a normalised cuspidal Hecke eigen-newform of some weight $k_f + 2 \geq 2$ and level $N_f > 4$, and let $L$ be a number field containing the Fourier
coefficients of $f$. We fix an extension $E$ of $\QQ_p$ containing $\jmath_p\circ\iota_\infty(L)$, which we shall enlarge as our purposes require. Let $\cO=\cO_E$ denote its ring of integers, $\mathfrak{m}=\frakm_E \subset \cO$ its maximal ideal, and $\mathbb{k}=\mathbb{k}_E:=\cO/\frakm$ its residue field.

\subsubsection{Galois representations attached to modular forms}
\label{sec:Galois_reps_modular_forms}
We let $V_{E}(f)$ (resp. $V_{E}(f)^*$) denote Deligne's Galois representation (resp. its dual) attached to $f$ given as in \cite[\S 2.8]{KLZRecLaw}, which is unramified outside $S$, the finite set of primes dividing $pN_f$. If $g$ is another eigen-newform of weight $k_g+2$ and level $N_g$ with $k_g\geq0$. We enlarge $E$ so that it contains the images of the Fourier coefficients of $f$ and $g$ under $\jmath_g$. We define the Galois representation for $f \otimes g$ by $V_{E}(f \otimes g):=V_{E}(f) \otimes_{E}\otimes V_{E}(g)$, and similarly its dual $V_{E}(f \otimes g)^*:=V_{E}(f)^* \otimes_{E}V_{E}(g)^*.$
\subsubsection{The residual representation}
\label{sec:resi_reps}
Let us choose a Galois-stable lattice $T_{\cO}(f)$ (resp. $T_{\cO}(f)^*$) inside of $V_{E}(f)$ (resp. of $V_{E}(f)^*$). For all modular forms we shall encounter in this work, we assume that the residual Galois representation\footnote{We note that we denote the residual representations by $\widetilde{\rho}$ instead of $\overline{\rho}$. This is because we have reserved $\overline{\star}$ to denote cyclotomic deformation, following Nekov\'a\v{r} and Benois.}
$$\widetilde{\rho}_f\,:\, G_{\QQ} \xrightarrow{\rho_f} {\rm GL}(T_{\cO}(f)) \simeq {\rm GL}_{2}(\cO_{\cO}) \longrightarrow {\rm GL}_{2}(\mathbb{k})$$
is irreducible. In this case, $\widetilde{\rho_f}$ is independent of the choice of the lattice $T_{\cO}(f)$ and it is called the residual Galois representation attached to $f$. 

\subsubsection{Assumptions}
\label{hypo_ranking_sel_1}
Suppose that  $f$ and $g$ are a pair of eigenforms of levels coprime to $p$. We let $\{\alpha_f,\beta_f\}$ denote the roots of the Hecke polynomial of $f$ at $p$, and similarly define  $\{\alpha_g,\beta_g\}$.

The following assumptions will be needed for some of our results.
\begin{enumerate}
\item[(\mylabel{item_p_large}{$\,{\textbf{p}}\,$})] The prime $p \geq 7$ is large enough to guarantee that Theorem~\ref{thm:Irr_THEOREM 3.2.2_Loff_IMA_OF} and Proposition~\ref{prop:Prop_4.2.1_Loff_Ima} hold. 
   \item[(\mylabel{item_reg}{$\,{\textbf{Reg}}\,$})] We assume that $\alpha_f \neq \beta_f$ and $\alpha_g \neq \beta_g$ where $\alpha_f, \beta_f$ are roots of the Hecke polynomial of $f$ at $p$, and similarly for $g$. This is known as $p$-regularity.
    \item[(\mylabel{item_NZ}{$\,{\textbf{nZ}}\,$})]  None of the pairwise products $\{\alpha_f\alpha_g, \alpha_f\beta_g, \beta_f\alpha_g, \beta_f\beta_g\}$ is equal to $p^j$ or $p^{1+j}$, so the Euler factor of $L(f,g, s)$ at $p$ does not vanish at $s=j$ or $s=1+j$.
    \item[(\mylabel{item_nobf}{$\,{\textbf{Nob}}\,$})] If $f$ is ordinary, then either $\alpha_f$ is the unit root of the Hecke polynomial, or $V_E(f)_{|_{G_{\QQ_p}}}$ is not the direct sum of two characters.
    \item[(\mylabel{item_conj}{$\,{\textbf{Conj}}\,$})] $f$ is not a Galois conjugate to a twist of $g$. Namely, for any $\gamma \in G_\QQ$, the Galois conjugate $f^\gamma$ of $f$ is not equal to $g \otimes \psi$ for a Dirichlet character $\psi$. 
   \item[(\mylabel{item_nonCM}{$\,{\textbf{nCM}}\,$})]  $f$ and $g$ are non-CM modular forms. 
   \item[(\mylabel{item_nontrivial}{$\,{\textbf{nTriv}}\,$})]  The character $\varepsilon_f\varepsilon_g$ is nontrivial.
\end{enumerate}
\begin{remark}
    A conjecture of Greenberg predicts the implication $\eqref{item_nonCM}\qquad \implies{}\qquad  \eqref{item_nobf}$.  Note that this conjecture is known to be true for forms of weight $2$ thanks to the works of Serre~\cite{Serre1968}, Emerton~\cite{EmertonGreenbergConjwt2}, Ghate~\cite{GhateGreenbergConjwt2}, and Zhao~\cite{ZhaoGreenbergConjwt2}.
\end{remark}
\begin{remark}
\label{rem:p-regularity}
The following is \cite[Remark 8.1.1]{LZColeman}, which we record as it is important for our purposes. We note that \eqref{item_reg} implies that we have direct sum decompositions 
$$D_{\rm cris}(V_E(f)^*)=D_{\rm cris}(V_E(f)^*)^{\alpha_f} \oplus D_{\rm cris}(V_E(f)^*)^{\beta_f}$$ where $\varphi$ acts on the two direct summands as multiplication by $\alpha^{-1}_f$ and $\beta^{-1}_f$ respectively, and similarly for $g.$ This induces a decomposition of $D_{\rm cris}(V_E(f \otimes g)^*)$ into a direct sum of four $\varphi$-stable submodules: 
$$D_{\rm cris}(V_E(f \otimes g)^*)=\bigoplus_{x,y\in \{\alpha,\beta\}}D_{\rm cris}(V_E(f \otimes g)^*)^{x_fy_g}\,.$$
\end{remark}
\begin{defn}
\label{defn:Definition_8.1.2_LZCol} We write
$D_{\rm cris}(V_E(f \otimes g)^*)^{\alpha_f \,\circ}=D_{\rm cris}(V_E(f \otimes g)^*)^{\alpha_f\alpha_g} \oplus D_{\rm cris}(V_E(f \otimes g)^*)^{\alpha_f\beta_g}$ and we write ${\rm pr}_{\alpha_f}$ for the projection 
$D_{\rm cris}(V_E(f \otimes g)^*) \lra D_{\rm cris}(V_E(f \otimes g)^*)^{\alpha_f \,\circ}$ with $D_{\rm cris}(V_E(f \otimes g)^*)^{\beta_f \,\circ}$ as kernel. 
\end{defn}


\subsection*{Acknowledgements}
This article is derived from the PhD thesis of the second-named author (MG), who was supervised by the first-named author. KB thanks Antonio Lei and Stefano Vigni for many interesting exchanges related to this work.

KB’s research in this publication was conducted with the financial support of Taighde \'{E}ireann -- Research Ireland under Grant number IRCLA/2023/849 (HighCritical). M.G. was funded by SFI-CRT Foundations of Data Science throughout her postgraduate studies.



\part{Rankin--Selberg products of ordinary families}
\label{Part_ord}

As the first basic case, we first study the functional equations of (multivariate) algebraic $p$-adic $L$-functions associated with Rankin--Selberg products of $p$-ordinary families. In this case, Nekov\'a\v{r}'s theory of Greenberg Selmer complexes is available, and the coefficient rings that we work with are Noetherian (and the remaining portions tackle the scenario when this is no longer the case). Even in the $p$-ordinary setting, our approach to proving the functional equations seems novel. Moreover, this functional equation will likely play a role in strengthening the forthcoming work of Kundu--Ray--Vigni~\cite{FuncEqSymHida} on the Iwasawa theory of symmetric squares of Hida families. Finally, we believe that our analysis of Tamagawa factors in families might be of independent interest. 

\section{Algebraic $p$-adic $L$-functions for ordinary families and functional equation}

\subsection{Hida families} 
\label{sec:Selmer complexes_ordinary_scenario}

In this portion, we will closely follow \cite[\S7]{KLZRecLaw} and \cite[\S2.1]{On_Artin_formalism_for_p-adic_Garrett--Rankin_Lfunctions} in our discussion of Hida families. 

Let ${\mathbf f}$ be a Hida family of $p$-adic cusp forms over $R_{\mathbf f}$ with tame level $N_{\mathbf f}$ and character $\varepsilon_{\mathbf f}$ of conductor dividing $pN_{\mathbf f}$. Here, we recall that $R_{\mathbf f}$ is a branch (i.e. an irreducible component) of Hida’s universal ordinary Hecke algebra, and it is a finitely generated module over $\Lambda_{\rm wt}:=\ZZ_p[[\ZZ^\times_p]]$ (where $\ZZ^\times_p$ acts on the space of modular forms via the inverse diamond operators at $p$) by Hida's fundamental work. Then $R_{\mathbf f}$ is a complete local Noetherian domain with maximal ideal $\frakm_{\f}$. For each $k \geq 0$, and $r \geq 1$, we let $I_{k,r}$ denote the ideal of $\Lambda_{\rm wt}$ generated by $(1+p^r)-(1+p^r)^k$. We define an arithmetic prime of $R_\f$ to be a prime ideal of height $1$, denoted by $P_{k,r}$, that lies over the ideal $I_{k,r}$ for some $k\geq 0$ and $r \geq 1$. It follows from Hida's control theorem that each arithmetic prime $\cP$ of $R_\f$ above $I_{k,r}$ corresponds to an eigenform $f_\cP$ of weight $k+2 \geq 2$ and level $Np^r$, and it determines a choice of prime $\PP \mid p$ of the coefficient field $\QQ(f)$. In this setting, we say that $f_\cP$ is a specialization of the family $\f$.

\subsubsection{} Thanks to Hida, we have a big Galois representation
$$\rho_{\mathbf f}:G_{\QQ,S} \longrightarrow {\rm GL}_2({\rm Frac}(R_{\mathbf f}))$$ attached to ${\mathbf f}$, where $S$ is a finite set of primes containing all those dividing $pN_{\mathbf f}\infty$. The Galois representation $\rho_{\mathbf f}$ is characterized by the property that
$${\rm Tr}(\rho_{\mathbf f}({\rm Frob}_\ell)) = a_\ell({\mathbf f}), \quad\forall \ell \not\in S.$$
Let us denote by $T_{\mathbf f} \subset {\rm Frac}(R_{\mathbf f})^{\oplus 2}$ the Ohta lattice (see \cite[Definition 7.2.5]{KLZRecLaw}, where our $T_{\mathbf f}$
corresponds to $M({\mathbf f})^*$ in op. cit.). Let us define the universal weight character $\bfchi$ as the composite map $\bfchi:G_{\QQ} \xrightarrow{\chi_{\rm cyc}} \ZZ^\times_p \hookrightarrow \Lambda_{\rm wt}^\times$
where $\chi_{\rm cyc}$ is the $p$-adic cyclotomic character. The universal weight character $\bfchi$ gives rise to the character 
$\bfchi_{{\mathbf f}}: G_{\QQ} \xrightarrow{\bfchi} \Lambda_{\rm wt}^\times \to R^\times_{\mathbf f}.$

We shall assume the following hypothesis for all Hida families that appear in this work:
\begin{enumerate}
    \item[(\mylabel{item_Irr_res_Hida_F}{${\rm Irr}$})] The residual representation $\widetilde{\rho}_{\mathbf f}: G_{\QQ,S}\to {\rm GL}(T_\f/\m_\f)$ is an absolutely irreducible $G_{\QQ}$-representation.
\end{enumerate}
Then any $G_{\QQ}$-stable lattice in ${\rm Frac}(R_\f)^{\oplus 2}$ is homothetic to $T_\f$.
By a theorem of Wiles, we have 
$${\rho_\f}_{|_{G_{\QQ_p}}} \simeq \begin{pmatrix} 
\delta_\f & * \\
0 & \alpha_\f
\end{pmatrix}$$
where $\alpha_\f: G_{\QQ_p} \to R^\times_\f$ is the unramified character given by $\alpha_\f({\rm Frob}_p)=a_p(\f)$ and $\delta_\f:=\bfchi_\f\chi^{-1}_{\rm cyc}\alpha^{-1}_\f\epsilon_\f$. We assume (for all Hida families that appear in this work) in addition that 
\begin{enumerate}
    \item[(\mylabel{item_Irr_dist_Hida_F}{${\rm Dist}$})] $\delta_\f \not \equiv \alpha_\f \pmod {\frakm_{\f}}$.
\end{enumerate}
Then the lattice $T_\f$ fits in an exact sequence $$0 \longrightarrow F^+T_\f \longrightarrow T_\f \longrightarrow F^-T_\f \longrightarrow 0$$ 
of $R_\f[G_{\QQ_p}]$-modules, where the action on $F^+T_\f $ (resp. $F^-T_\f $) is given by $\delta_\f$ (resp. $\alpha_\f$). If $T$ is a quotient of $T_\f$, we shall denote the image of $F^+T_\f$ by $F^+T$.

\subsubsection{} 
Let $\g$ be another Hida family. We define $T_2:=T_\f \, \widehat{\otimes}_{\ZZ_p}\, T_\g$, put $N:={\rm lcm}(N_\f, N_\g)$, and let $S=S(T_2)$ denote the set of primes dividing $pN\infty$. Then $T_2$ is a $G_{\QQ,S}$-representation of rank $4$ over the complete local Noetherian ring $R_2:=R_\f \widehat{\otimes}_{\ZZ_p} R_\g$ of of Krull-dimension $3$. We define $F^+T_2:=F^+T_\f\, \widehat{\otimes}_{\ZZ_p}\, T_\g$.

For each specialisation $f$ of $\f$, we have a specialisation-at-$f$ isomorphism ${\rm sp}_f: T_\f \otimes_{R_\f} \cO \xrightarrow{\,\sim\,} T_{\cO}(f)$ (where $\cO=\cO_E$ is as in \S\ref{subsubsec_121_2024_11_03_826}, where we enlarge $E$ so that it contains the Hecke fields of both $f$ and $g$). We then have $T_2 \otimes_{{\rm sp}_f\otimes {\rm sp}_g} \cO \simeq T_{\cO}(f) \otimes_{\cO_E} T_{\cO}(g)=:T_{\cO}(f \otimes g)$ for all specializations $f$ (and $g$) of $\f$ (and $\g$). We also put 
$F^+T_{\cO}(f \otimes g):=F^+T_{\cO}(f) \otimes_{\cO_E} T_{\cO}(g).$ 

\subsection{Greenberg Local conditions.}
For $T_2$ as above, we have a short exact sequences of $G_{\QQ_p}$-modules
\begin{equation}
\label{eq:short_exact_seq_Greenberg}
0 \longrightarrow F^+T_2 \longrightarrow T_2 \longrightarrow T_2/F^+T_2 \longrightarrow 0\,.
\end{equation}
Consider the following local conditions on $T_2$ on the level of continuous cochains, in the sense of \cite[\S6]{nekovar06}:
\begin{equation}
\begin{matrix}
\label{eq:greenberg_local}
U^+_v(T_2)=\begin{cases} C^\bullet(G_p, F^+T_2),& \hbox{ if } v=p \\
C^\bullet(G_v/I_v, (T_2)^{I_v}), &  \hbox{ if }  v \in S \setminus \{p\}
\end{cases}
\end{matrix}
\end{equation}
where $I_v$ is the inertia group at $v.$ These come equipped with a morphism of complexes $$i^+_v:U^+_v(T_2) \longrightarrow C^\bullet(G_v, T_2), \qquad \forall v \in S\,.$$ 
For each $v \in S \setminus \{p\}$, we note that $U^+_v(T_2)$ is quasi-isomorphic to the complex 
$$U^+_v(T_2)=[T_2^{I_v} \xrightarrow{{\rm Frob}_v-1} T_2^{I_v}],$$
concentrated in degrees $0$ and $1$, and ${\rm Frob}_v \in G_v/I_v$ is the Frobenius element.
The data 
$$\Delta=\Delta(F^+T_2):=\{U^+_v(T_2) \xrightarrow{i^+_v} C^\bullet(G_v, T_2): v \in S\}$$
is called a Greenberg-local condition on $T_2.$

\subsubsection{Selmer complexes associated to Greenberg local conditions.}
\label{sec:selCom_Greenberg}

Let us put 
$$U^+_S(T_2):=U^+_p(T_2) \oplus \bigoplus_{v \in S\setminus \{p\}} U^+_v(T_2)\,,$$ 
$$C^\bullet_{S}(T_2):=\bigoplus_{v \in S} C^\bullet(G_v, T_2)\,.$$ 

We define the Selmer complex associated to $(T_2, S, \Delta)$ on setting $$S^\bullet(T_2, F^+T_2):={\rm cone}\left( C^\bullet(G_{\QQ,S}, T_2) \oplus U^+_S(T_2) \xrightarrow{{\rm res}_S-i^+_S} C^\bullet_{S}(T_2)\right)[-1].$$
We denote the corresponding object in the derived category by $\boldR\Gamma(T_2, F^+T_2)$ and its cohomology by $\boldR^i\Gamma(T_2, F^+T_2)$. For $T_2^*:={\rm Hom}(T_2,R_2)$, we similarly define the dual Selmer complex
$S^\bullet(T_2^*(1), F^+T_2^*(1))$ given by the dual Greenberg conditions $F^+T_2^*:=F^+T^*_\f \otimes_{\ZZ_p} T^*_\g$, where $F^+T^*_\f:={\rm Hom}_{R_\f}(F^-T_\f, R_\f)$\,.

\subsubsection{Perfectness}
Let $\frakm_{2}$ denote the maximal ideal of $R_2$ and let us put $\widetilde{T}_2:=T_2 \otimes_{R_2} R_2/\frakm_{2}$. From now on, we assume the following two conditions:
\begin{enumerate}
    \item[(\mylabel{item_Irr_res_H0}{$H^0$})] $\widetilde{T}_2^{G_\QQ}=0={\widetilde{T}_2}^\vee(1)^{G_\QQ},$ where  $(-)^\vee:={\rm Hom}(-, \QQ_p/\ZZ_p)$ denotes the Pontryagin duality functor.
    \item[(\mylabel{item_Irr_res_Tama_Hida}{${\rm Tam}$})] $H^1(I_v, T_2)$ is free for all $v \in S \setminus \{p\}$
\end{enumerate}

\begin{proposition}[Proposition 4.46 in \cite{On_Artin_formalism_for_p-adic_Garrett--Rankin_Lfunctions}]
\label{prop:Prop4.46_ON_Artin}   
\item[1)] For any ideal $I$ of $R_2$, we have $\boldR\Gamma(T_2, F^+T_2) \otimes^{\bbLL}_{R_2} R_2/I \simeq \boldR\Gamma(T_2/I, F^+T_2/I).$
\item[2)] $\boldR\Gamma(T_2, F^+T_2) \in \cD^{[1,2]}_{\rm perf}(R_2).$ Similary, $\boldR\Gamma(T_2^*, F^+T_2^*) \in \cD^{[1,2]}_{\rm perf}(R_2).$
\end{proposition}

\begin{corollary}[Corollary 4.49 in \cite{On_Artin_formalism_for_p-adic_Garrett--Rankin_Lfunctions}]
\label{coro:Corollary 4.49_On_Artin}
Suppose that $R_2$ is normal, and that the $R_2$-module $\boldR^2\Gamma(T_2, F^+T_2)$ is torsion. Then, $\boldR^1\Gamma(T_2, F^+T_2)$ vanishes, and the projective dimension of $H^2(T_2, F^+T_2)$ equals to $1$ and $${\rm det}(\boldR\Gamma(T_2, F^+T_2))={\rm char}_{R_2}(\boldR^2\Gamma(T_2, F^+T_2)).$$ 
\end{corollary}
See also \cite[Proposition 4.44]{On_Artin_formalism_for_p-adic_Garrett--Rankin_Lfunctions} and \cite[Corollary 4.45]{On_Artin_formalism_for_p-adic_Garrett--Rankin_Lfunctions} for a general perspective on the equality in Corollary~\ref{coro:Corollary 4.49_On_Artin}.

\subsection{Iwasawa theoretic Selmer complexes}
\subsubsection{Ring theoretic preparations}
Let us consider the ring $R_2[[\Gamma_{\QQ}]]$, which is a complete semi-local complete Noetherian ring. The minimal ideals of $R_2[[\Gamma_{\QQ}]]$ are indexed by the characters $\eta \in \widehat{\Delta}_\QQ$ of the torsion subgroup $\Delta_{\QQ}\simeq \mathbb{F}_p^\times$ of $\Gamma_{\QQ}$. Note also that the ring $R_2[[\Gamma_{\QQ}^0]]$ is complete local Noetherian of dimension $4$, and 
$$R_2[[\Gamma_{\QQ}]]\simeq \bigoplus_{\eta \in \widehat{\Delta}_{\QQ}}e_\eta R_2[[\Gamma_{\QQ}]]\simeq \bigoplus_{\eta \in \widehat{\Delta}_{\QQ}} R_2[[\Gamma_{\QQ}^0]]\,,$$
where $e_\eta=\sum_{g\in \Delta_{\QQ}} \eta^{-1}(g) g$ is the idempotent corresponding to $\eta$. In what follows, our results over the ring $R_2[[\Gamma_{\QQ}]]$ will be deduced from those over the complete local Noetherian ring $R_2[[\Gamma_{\QQ}^0]]$ via the isomorphism above.

We define the characteristic ideal of a torsion module $M$ over $R_2[[\Gamma_{\QQ}]]$ as follows:
\begin{defn}
\label{defn:chara_over_units_Zp}
${\rm char}_{R_2[[\Gamma_{\QQ}]]} M:= \Sigma_{\eta \in \Delta_{\QQ}} {\rm char}_{R_2[[\Gamma_{\QQ}^0]]}\left( e_\eta M \otimes_{R_2[[\Gamma_{\QQ}]]} R_2[[\Gamma_{\QQ}^0]] \right) e_\eta\,.$ 
\end{defn}

\subsubsection{}
We set $\overline{T_2}:=T_2 \otimes_{R_2} R_2[[\Gamma_{\QQ}]]$ and define  $F^+\overline{T_2}:=F^+T_2 \otimes_{R_2} R_2[[\Gamma_{\QQ}]]$. As before, we define the Selmer complex associated to $(\overline{T_2}, S, \Delta)$ on setting 
$$S^\bullet(\overline{T_2}, F^+\overline{T_2}):={\rm cone}\left( C^\bullet(G_{\QQ,S}, \overline{T_2}) \oplus U^+_S(\overline{T_2}) \xrightarrow{{\rm res}_S-i^+_S} C^\bullet_{S}(\overline{T_2})\right)[-1]\,,$$
the corresponding object $\boldR\Gamma(\overline{T_2}, F^+\overline{T_2})$ in the derived category.

\subsubsection{Dualtiy of Selmer Complexes}
\label{sec:duality_sel_compl} 
By \cite[8.9.6.3]{nekovar06}, we have an exact triangle (Grothendieck duality)
$$\boldR\Gamma(\overline{T_2}, F^+\overline{T_2}) \xrightarrow{\,u\,} \boldR{\rm Hom}_{R_2[[\Gamma_{\QQ}]]}(\boldR\Gamma(\overline{T_2^*}, F^+\overline{T_2^*}), R_2[[\Gamma_{\QQ}]])[-3] \longrightarrow \bigoplus_{v \in S} {\rm Err}_v\,,$$ 
where: 
\begin{enumerate}
    \item[i)] ${\rm Err}_v={\rm Err}_v(\mathcal{D},\overline{T_2}) \simeq \cD \circ D({\rm Cone}\left(H^{3}_{\{m\}}(H^1(I_v,\overline{T_2}) ) \xrightarrow{{\rm Frob}_v-1} H^{3}_{\{m\}}(H^1(I_v,\overline{T_2}) ))\right)[-1]$\,, see \cite[Proposition 7.6.7]{nekovar06}. We recall here that $\cD$ denotes Grothendieck duality, and $D$ denotes Matlis duality (cf. \cite{nekovar06}, \S0.4).
\item[ii)] $H^{3}_{\{m\}}(H^1(I_v,\overline{T_2}))$ is as in \cite[\S2.4.1]{nekovar06}.
\end{enumerate}


We conclude this subsection with the following result on the error terms ${\rm Err}_v$.
\begin{proposition}[Proposition 8.9.7.3 and Corollary 8.9.7.4 in \cite{nekovar06}]
\label{prop:Proposition 8.9.7.3_nekovar}
For each non-archimedean prime $v \nmid p$,  we have an isomorphism ${\rm Err}_v \simeq 0$ in $\cD^b(_{R_2[[\Gamma_{\QQ}]]}{\rm Mod}/\hbox{\rm pseudo-null})$. In particular, the cohomology of the complex  ${\rm Err}_v$ is pseudo-null in all degrees.
\end{proposition}

\begin{remark}
Our main strategy in this thesis is to exploit the description of the algebraic $p$-adic $L$-functions as determinants of Selmer complexes. Unfortunately, Proposition~\ref{prop:Proposition 8.9.7.3_nekovar} does not allow any control over the determinant of the error terms, and we refine this general result in the next subsection.
\end{remark}
\subsection{Perfect amplitude of Selmer complexes}
\label{subsec_4_7_11_2024_07_12}
Let us denote by $\rho_{T_2}$ the Galois representation afforded by the Galois module $T_2$. 
\subsubsection{Tamagawa numbers in families} 
\label{subsubsec_2024_11_04_0603}
We recall that $I_v$ denotes a fixed inertia group at $v\nmid p$, and $I_v^{(p)}<I_v$ is the unique subgroup such that $I_v/I_v^{(p)}$ is the maximal pro-$p$ quotient of $I_v$ (in particular,  $I_v/I_v^{(p)}\simeq \ZZ_p$).

Proposition~\ref{prop:Corollary 4.26_vari_on_artin} below is an extension of \cite[Corollary 4.26]{On_Artin_formalism_for_p-adic_Garrett--Rankin_Lfunctions}.

\begin{proposition}
\label{prop:Corollary 4.26_vari_on_artin}  
If both $\rho_\f(I_v)$ and $\rho_\g(I_v)$ have infinite cardinality, then the
following are equivalent.
\item[i)] $H^1(I_v, T_2)$ is free.
\item[ii)] $H^1(I_v, T_\cO(f_{\cP_1}) \otimes_{\cO} T_\cO(g_{\cP_2}))$ is free for some $\cO$-valued arithmetic prime ${\cP_1}$ of $R_\f$ (resp. ${\cP_2}$ of $R_\g$).
\item[iii)] The Tamagawa factor for $T_\cO(f_{\cP_1})$ at $v$ equals to $1$ for some arithmetic prime ${\cP}_1$ of $\cR_\f$, and the same holds for  $T_\cO(f_{\cP_2})$ for some arithmetic prime ${\cP}_2$ of $\cR_\g$.

Moreover, whenever these equivalent conditions are satisfied, then $H^1(I_w, T_\f \otimes T_\g)$ is free for any finite index
subgroup $I_w$ of $I_v$ with $p \nmid  [I_v:I_w].$
\end{proposition}

\begin{proof}
\label{proof:prop:Corollary 4.26_vari_on_artin}
The condition (i) implies the condition (ii) thanks to  \cite[Lemma 4.21(3)]{On_Artin_formalism_for_p-adic_Garrett--Rankin_Lfunctions}. The condition (ii) implies the condition (iii) by the discussion at the start of \cite[\S4.4]{On_Artin_formalism_for_p-adic_Garrett--Rankin_Lfunctions}. We will prove that (iii) implies (i), which will conclude our proof except for the very final assertion.

By \cite[Proposition 4.25]{On_Artin_formalism_for_p-adic_Garrett--Rankin_Lfunctions}, we have an exact sequence of $G_v$-modules
$$0 \longrightarrow R_2 (1) \otimes \mu \longrightarrow T_2 \longrightarrow R_2 \otimes \mu \longrightarrow 0$$ where $\mu:G_v \to \{\pm 1\}$ is a quadratic character. We shall treat the scenario where $\mu=1$ is unramified, as the remaining case is similar. Note that, in this case, ${\rm ord}_v(N_\f)=1={\rm ord}_v(N_\g)$ by \cite[Proposition 4.25]{On_Artin_formalism_for_p-adic_Garrett--Rankin_Lfunctions}.

Let us take a basis $\{e_1, e_2\}$ of $T_\f$ so that ${R_\f}e_1 = R_\f(1)$ 
and 
$\rho_{\f}(t)=\begin{pmatrix}
1 & a_\f  \\
0 & 1
\end{pmatrix}$
for some element $ 0 \neq a_\f \in  R_\f$. Here, $t$ is a fixed topological generator of $I_v/I^{(p)}_v\simeq \ZZ_p$. Similarly, we take  a basis $\{e'_1, e'_2\}$ of $T_\g$ so that ${R_\g}e'_1 = R_\g(1)$
and 
$\rho_{\g}(t)=\begin{pmatrix}
1 & a_\g \\
0 & 1
\end{pmatrix}$
for some element $ 0 \neq a_g \in  R_\g$. We then have 
\begin{equation}
\label{eq:05072024_1425}
\rho_{T_2}(t)=\begin{pmatrix}
1  & 1 \otimes a_\g & a_\f \otimes 1 & a_\f \otimes a_\g\\
0 & 1 & 0 & a_\f \otimes 1\\
0 & 0 & 1  & 1 \otimes a_\g\\
0 & 0 & 0 & 1 
\end{pmatrix}.
\end{equation}

As explained in the proof of \cite[Corollary 4.26]{On_Artin_formalism_for_p-adic_Garrett--Rankin_Lfunctions}, we have $a_\f\in R_\f^\times$ and $a_\g\in R_\g^\times$ (hence also that $a_\f \otimes a_\g \in R^\times_2$) whenever (iii) holds. By Gaussian elimination, it follows from \eqref{eq:05072024_1425} and \cite[Lemma 4.21(1)]{On_Artin_formalism_for_p-adic_Garrett--Rankin_Lfunctions} that the module $H^1(I_v, T_2)$ is free, which is the condition (i).

The very final assertion follows from the fact that $I_w/(I_w \cap I^{(p)}_v)=I_v/I^{(p)}_v$ combined with the fact that $a_\f\in R_\f^\times$ and $a_\g\in R_\g^\times$ (hence also that $a_\f \otimes a_\g \in R^\times_2$) whenever (iii) holds.

\end{proof}

We shall work with the following technical conditions for all the Hida families involved in our discussion.
\begin{assumption}
\label{assump:assum4.35_on_artin}
For any (rational) prime $v$ such that cardinality of $\rho_{\f}(I_v)$ is finite and greater than $0,$ one of the following holds:
\item[1)] $v \neq 2$ and $p \nmid v-1$,
\item[2)] $v=2$ and $p \neq 3,7$,
\item[3)] $\zeta_p + \zeta^{-1}_p \not\in L_f$ where $L_f$ denotes the Hece field of some specialization $f$ of $\f$.
\end{assumption}

The proposition below is a variant of \cite[Proposition 4.37]{On_Artin_formalism_for_p-adic_Garrett--Rankin_Lfunctions}.
\begin{proposition}
\label{prop:prop4.37_on_artin_Iawasawa}
Suppose that Assumption~\ref{assump:assum4.35_on_artin} holds for both $\f$ and $\g.$ Assume also that the $p$-part of the Tamagawa factor (at $v\nmid p$) for some member of the Hida family $\f$ and some member of the family $\g$ equals to $1$. Then $H^1(I_v, \overline{T_2})$ is free. 
\end{proposition}
\begin{proof}
\label{proof:prop:prop4.37_on_artin_Iwasawa}
The proof follows once we show that $H^1(I_v, T_2)$ is a free $\cR_2$-module under our running assumptions, since $H^1(I_v, \overline{T_2})\simeq H^1(I_v, T_2)\widehat{\otimes_{R_2}}R_2[[\Gamma_{\QQ}]]$. The freeness of  $H^1(I_v, T_2)$ is proved in a manner similar to \cite[Proposition 4.37]{On_Artin_formalism_for_p-adic_Garrett--Rankin_Lfunctions}, which we review in what follows.

We retain our notation in the proof of Proposition~\ref{prop:Corollary 4.26_vari_on_artin}. If both $\rho_{\f}(I_v)$ and $\rho_{\g}(I_v)$ have finite cardinality,  then $p$ does not divide the cardinality of $\rho_{\f}(I_v)$ and $\rho_{\g}(I_v)$ by \cite[Corollary 4.34]{On_Artin_formalism_for_p-adic_Garrett--Rankin_Lfunctions}. Since $T_2=T_\f \widehat\otimes_{\ZZ_p} T_\g$ as $I_v$-modules, we conclude that the cardinality of $\rho_{T_2}(I_v)$ is finite and $p$ does not divide the cardinality of $\rho_{T_2}(I_v).$ It then follows from \cite[Corollary 4.23]{On_Artin_formalism_for_p-adic_Garrett--Rankin_Lfunctions} that $H^1(I_v, T_2)$ is free, as required.

Next, we consider the case that cardinality of $\rho_{\f}(I_v)$ is finite but cardinality of $\rho_{\g}(I_v)$ is infinite. Then by \cite[Proposition 4.25]{On_Artin_formalism_for_p-adic_Garrett--Rankin_Lfunctions} (see also \cite[Proposition 12.7.14.1(i)]{nekovar06}), it follows that the subgroup $I^{(p)}_v$ acts trivially on $T_\g$. Hence, 
$$(T_2)^{I^{(p)}_v}=(T_\f)^{I^{(p)}_v} \otimes_{\ZZ_p} T_\g.$$
Moreover, by \cite[Corollary 4.34]{On_Artin_formalism_for_p-adic_Garrett--Rankin_Lfunctions}, we see that $(T_\f)^{I^{(p)}_v}=0$. Thence, $(T_2)^{I^{(p)}_v}=0$, and in particular $H^1(I_v, T_2)=0$ by Remark 4.22 in op. cit. The case that the cardinality of $\rho_{\g}(I_v)$ is finite but the cardinality of $\rho_{\f}(I_v)$ is infinite is similar. The case that both the cardinality of $\rho_{\f}(I_v)$ and cardinality of $\rho_{\g}(I_v)$ are infinite was already treated in Proposition~\ref{prop:Corollary 4.26_vari_on_artin}, and our proof is complete.


\end{proof}

\begin{proposition}
\label{prop:Iwasa_Tamagama_factor_0__10062024_0830}
Suppose that Assumption~\ref{assump:assum4.35_on_artin} holds for both $\f$ and $\g$. Assume also that the $p$-part of the Tamagawa factor (at $v\nmid p$) for some member of the Hida family $\f$, and some member of the family $\g$ equals $1$. Then, 
${\rm Err}_v(\mathcal{D},\overline{T_2}) \simeq 0 \hbox{ in } \cD^b(R_2[[\Gamma_{\QQ}]])$. 
\end{proposition}
\begin{proof}
\label{proof:prop:Iwasa_Tamagama_factor_0__10062024_0830}

By the proof of \cite[Proposition 8.9.7.6(iv)]{nekovar06} (see also \cite[Corollary 7.6.12(iii)]{nekovar06}), we see that ${\rm Err}_v(\mathcal{D},\overline{T_2}) \simeq 0$ in $\cD(R_2[[\Gamma_{\QQ}]])$ if and only if ${\rm Err}_v(\cD,T_2)=0$ in $\cD^b(R_2)$ (where ${\rm Err}_v(\cD,T_2)$ is defined in \cite[Corollary 7.6.12(iii)]{nekovar06}), if and only if 
$$(\forall q=1,2,3)\,\qquad \alpha'_q:{\rm Ext}^q_{R_2}(H^1(I_v, T_2), R_2) \xrightarrow{{\rm Frob}^{-1}_v-1}{\rm Ext}^q_{R_2}(H^1(I_v, T_2), R_2) $$ is an isomorphism. By the proof of Proposition~\ref{prop:prop4.37_on_artin_Iawasawa}, the $R_2$-module $H^1(I_v, T_2)$ is free (under our running assumptions) therefore projective, and 
we ${\rm Ext}^q_{R_2}(H^1(I_v, T_2), R_2)=0$ for all $q=1,2,3$,  as required.
\end{proof}

\subsubsection{Consequences}

\begin{proposition}
\label{prop:perfect_duality_selCOm}
Suppose that we are in the situation of Proposition~\ref{prop:Iwasa_Tamagama_factor_0__10062024_0830}.
Then the morphism $u$ in \S\ref{sec:duality_sel_compl} is an isomorphism. 
\end{proposition}

\begin{proof}
\label{proof:prop:perfect_duality_selCOm}
We proved in Proposition~\ref{prop:Iwasa_Tamagama_factor_0__10062024_0830} that ${\rm Err}_v(\mathcal{D},\overline{T_2}) \simeq 0$. Hence, the map $u$ in \S\ref{sec:duality_sel_compl} is an isomorphism, and the result follows on taking determinants.
\end{proof}

The following result is a special case of \cite[Proposition 4.46]{On_Artin_formalism_for_p-adic_Garrett--Rankin_Lfunctions}, where we take $R=e_\eta R_2[[\Gamma_{\QQ}]]$ and $T=e_\eta\overline{T_2}$, and let $\eta$ run through $\widehat{\Delta}_{\QQ_p}$.

\begin{proposition}
\label{prop:perfect_Iwasawa_complex_cycloto}
\item[i)] For any ideal $I$ of $R_2[[\Gamma_{\QQ}]]$, we have $\boldR\Gamma(\overline{T_2}, F^+\overline{T_2}) \otimes^{\bbLL}_{R_2[[\Gamma_{\QQ}]]} R_2[[\Gamma_{\QQ}]]/I \simeq \boldR\Gamma(\overline{T_2}/I, F^+\overline{T_2}/I).$
\item[ii)] Suppose we are in the situation of Proposition~\ref{prop:Iwasa_Tamagama_factor_0__10062024_0830}. Then 
$$\boldR\Gamma(\overline{T_2}, F^+\overline{T_2}) \in \cD^{[1,2]}_{\rm perf}(R_2[[\Gamma_{\QQ}]])\,.$$
\end{proposition}


\begin{theorem}[Theorem 11.6.4 in \cite{KLZRecLaw}]
\label{thm:Theorem 11.6.4KLZRecLaw}
Suppose that the specializations $f$ and $g$ are crystalline (in the sense that they arise as the unique ordinary $p$-stabilizations of newforms $f^\circ$ and $g^\circ$ of levels $N_f$ and $N_g$, respectively). Let us assume that \eqref{item_p_large} and \eqref{item_NZ} hold true for $f$ and $g$. Suppose also that $\eta\in\widehat{\Delta}_{\QQ_p}$ is such that $e_\eta L_p(f, g) \neq 0$, where $L_p(f,g)$ is Hida's $f$-dominant $p$-adic $L$-function as in \cite[Theorem 3.5.3]{KLZModForm}. Let us put $T:=T_\cO(f)\otimes T_{\cO}(g)$. Then:  
\item[i)] $e_\eta\boldR^1\Gamma(\overline{T},F^+\overline{T})=0.$
\item[ii)] $e_\eta\boldR^2\Gamma(\overline{T},F^+\overline{T})$ is 
torsion over $\cO_E[[\Gamma_{\QQ}]].$
\end{theorem}

\begin{remark}[Remark 11.6.5 in \cite{KLZRecLaw}]
\label{remark_KLZRemark1165}
    Note that if $k_f-k_g \geq 2$, then the assumption that $e_{\eta} L_p(f,g) \neq 0$ is automatically satisfied (for all $\eta \in \widehat{\Delta}_{\QQ_p}$) thanks to \cite[Proposition 2.7.6]{KLZRecLaw}.
\end{remark}

\begin{theorem}
\label{thm:thm:H^1_Selcom_Hida_family_Iwasa_zero}
Suppose that the families $\f$ and $\g$ admit crystalline specializations $f$ and $g$ such that \eqref{item_p_large} and \eqref{item_NZ} hold true for $f$ and $g$, and that $k_f-k_g\geq 2$.
Then we have:
\item[i)] $\boldR^1\Gamma(\overline{T_2}, F^+\overline{T_2})=0.$
\item[ii)] $\boldR^2\Gamma(\overline{T_2}, F^+\overline{T_2})$ is torsion over $R_2[[\Gamma_{\QQ}]]$.
\end{theorem}
\begin{proof}
\label{proof:thm:H^1_Selcom_Specia_Hida_family_Iwasa_zero}
By Proposition~\ref{prop:perfect_Iwasawa_complex_cycloto}(i), we have 
$$\boldR^1\Gamma(\overline{T_2},F^+\overline{T_2}) \Big{/}\ker(R_2\xrightarrow{{{\rm sp}_f\otimes {\rm sp}_g}}\cO_E)\hookrightarrow \boldR^1\Gamma(\overline{T},F^+\overline{T})=0$$
(where the vanishing follows from Theorem~\ref{thm:Theorem 11.6.4KLZRecLaw} and Remark~\ref{remark_KLZRemark1165}), and 
$$\boldR^2\Gamma(\overline{T_2},F^+\overline{T_2}) \otimes_{{\rm sp}_f\otimes {\rm sp}_g}  \cO_E[[\Gamma_{\QQ}]] \simeq \boldR^2(\overline{T},F^+\overline{T})\,.$$
Note that we only use the vanishing 
$$\boldR^0\Gamma(\overline{T_2},F^+\overline{T_2})=0=\boldR^3\Gamma(\overline{T_2},F^+\overline{T_2})$$
for the injection and isomorphism above, rather than the perfectness of the Selmer complex $\boldR\Gamma(\overline{T_2}, F^+\overline{T_2})$.

The proof of the first assertion follows from Nakayama's lemma, whereas the second is clear thanks to Theorem~\ref{thm:Theorem 11.6.4KLZRecLaw}(ii), which tells us that $\boldR^2(\overline{T},F^+\overline{T})$ is a torsion $\cO_E[[\Gamma_{\QQ}]]$-module.
\end{proof}

\subsubsection{Functional equation}
\label{sec:func_eq_ordinary}
We are now ready to state and prove one of the main results of our paper: Functional equation for a 3-variable algebraic $p$-adic $L$-function.
\begin{theorem}
\label{thm:Func_eq_hida}
Suppose that we are in the situation of Proposition~\ref{prop:perfect_duality_selCOm}. Assume also the hypotheses of Theorem~\ref{thm:thm:H^1_Selcom_Hida_family_Iwasa_zero}.
Then,
$${\rm char}_{R_2[[\Gamma_{\QQ}]]}\left(\boldR^2\Gamma(\overline{T_2}, F^+\overline{T_2})\right)={\rm char}_{R_2[[\Gamma_{\QQ}]]}\left(\boldR^2\Gamma(\overline{T_2^*}(1), F^+\overline{T_2^*}(1))\right)^\iota\,.$$
\end{theorem}
\begin{proof}
By Proposition~\ref{prop:perfect_Iwasawa_complex_cycloto}, we have $\boldR\Gamma(\overline{T_2}, F^+\overline{T_2}) \in \cD^{[1,2]}_{\rm perf}(R_2[[\Gamma_{\QQ}]])$. Moreover, by Theorem~\ref{thm:thm:H^1_Selcom_Hida_family_Iwasa_zero}, we have 
$\boldR^1\Gamma(\overline{T_2}, F^+\overline{T_2})=0\,,$ hence $\boldR\Gamma(\overline{T_2}, F^+\overline{T_2})$ can be represented by a complex of the form $[P_1\xrightarrow{\,\,\phi\,\,} P_2]\,,$ where $P_1$ and $P_2$ are free $R_2[[\Gamma_{\QQ}]]$-modules of the same rank with $\phi$ injective, so that 
$${\rm char}_{R_2[[\Gamma_{\QQ}]]}\left(\boldR^2\Gamma(\overline{T_2}, F^+\overline{T_2})\right)=\det \phi= \det \boldR\Gamma(\overline{T_2}, F^+\overline{T_2}) \,.$$
By Proposition~\ref{prop:perfect_duality_selCOm}, the map $u$ in \S\ref{sec:duality_sel_compl} is isomorphism. As a result, we have 
\begin{equation}
\label{eqn_2024_11_03_1714}
\boldR\Gamma(\overline{T_2^*}(1), F^+\overline{T_2^*}(1))\simeq \left[P_2^* \stackrel{\,\,^{t}\phi\,\,}{\lra} P_1^* \right] \quad \hbox{ concentrated in degrees 1 and 2}\,,
\end{equation}
where $P_i^*:={\rm Hom}(P_i,R_2[[\Gamma_\QQ]])$ ($i=1,2$) and $^t\phi$ is the transpose of $\phi$. Moreover, the map $^{t}\phi$ is injective because ${\rm coker}(\phi)\simeq \boldR^2\Gamma(\overline{T_2^*}(1), F^+\overline{T_2^*}(1))$ is torsion by assumption (so that ${\rm Hom}({\rm coker}(\phi), R_2[[\Gamma_\QQ]])=\{0\}$). From \eqref{eqn_2024_11_03_1714} and the discussion that follows, we conclude that 
$$\det(\phi)^\iota=\det(^t\phi)={\rm char}_{R_2[[\Gamma_\QQ]]} \,\boldR^2\Gamma(\overline{T_2^*}(1), F^+\overline{T_2^*}(1))\,, $$
and the proof of our theorem follows.
\end{proof}

\subsection{The adjoint representation}
\label{sec:Adj_reps}
Note that the hypothesis \eqref{item_p_large} that we require in the previous section dictates that the Hida families $\f$ and $\g$ are ``distinct''. In this subsection, where we closely follow closely follow \cite{Loeffler_2016}, we treat the opposite extreme case of the symmetric square of a Hida family $\f$. Our main results in this setting, when combined with the forthcoming work \cite{FuncEqSymHida} of Kundu, Ray and Vigni, have applications towards the Iwasawa main conjectures for the symmetric squares of Hida families.

\subsubsection{} Let $f \in S_{k_f+2}(\Gamma_1(N_f),\varepsilon_f)$ be a ($p$-ordinary) normalized cuspodal eigen-newform such that $(N_f,p)=1$. Let $\alpha_f$ and $\beta_f$ denote the roots of the Hecke polynomial of $f$ at $p$, where $v_p(\alpha_f)=0$. Let $f_{\alpha_f}$ denote the $p$-ordinary $p$-stabilization of $f$.  Let $\psi : (\ZZ/N_\psi\ZZ)^\times \to E^\times$ be a Dirichlet character whose conductor $N_\psi$ is coprime to $pN_f$. We put $V:={\rm Sym}^2V_E(f)^*(1+\psi)$, where we identify $\psi$ with a character of $G_{\QQ}$ in the usual way. Set $T:={\rm Sym}^2 T_{\cO}(f)^*(1+\psi)$.

\subsubsection{} Let $\f$ be the unique Hida family of tame-level $N_f$ and tame-nebentype $\varepsilon_f$ that admits $f_{\alpha_f}$ as a specialization. We define $\boldT:={\rm Sym}^2T_\f^*(1+\psi)$. 

\subsubsection{} As explained in \cite[Note 3.2.2]{Loeffler_2016}, the vector space $T$ is an irreducible representation of $G_\QQ$ of $\cO$-rank $3$, unramified outside $pN_fN_\psi$, and crystalline at $p$. When regarded as a $G_{\QQ_p}$-representation, it admits a $3$-step filtration 
\begin{align}
\begin{aligned}
\label{eqn_2024_07_18_1331}
    T={\mathcal F}^0T \supset {\mathcal F}^1T=&\,\left(F^+T_{\cO}(f)\otimes T_{\cO}(f)+T_{\cO}(f)\otimes F^+T_{\cO}(f)\right)(1+\psi)\,\cap\, T  \\
    &\qquad \supset {\mathcal F}^2T=F^+T_{\cO}(f) \otimes F^+T_{\cO}(f)(1+\psi)\supset {\mathcal F}^3T=\{0\}\,,
\end{aligned}
\end{align}
with all graded pieces $1$-dimensional. Similar discussion applies with $V$ and $\boldT$.

\subsubsection{} Recall that we denote by $(-)^\vee:={\rm Hom}(-, \QQ_p/\ZZ_p)$ the Pontryagin duality functor. Let us consider the discrete Galois module $T^{\vee}(1)$ and equip it with a $3$-step filtration (by a decreasing sequence of $G_{\QQ_p}$-stable $\cO$-submodules) by defining $${\mathcal F}^iT^{\vee}:=(T/{\mathcal F}^{3-i}T)^{\vee}\,.$$

\subsubsection{Tamagawa factors in families (bis)}
\label{subsubsec_2024_07_12_1227}
The proposition below is a variant of \cite[Proposition 4.37] {On_Artin_formalism_for_p-adic_Garrett--Rankin_Lfunctions} for adjoint representation in the Hida family.
\begin{proposition}
\label{prop:prop4.37_on_artin_Hida_Adjoint}
Suppose that Assumption~\ref{assump:assum4.35_on_artin} holds for $\f.$ Assume also that the $p$-part of the Tamagawa factor (at $v\nmid p$) for some member of the Hida family $\f$ equals to $1$. Then $H^1(I_v, \overline{\boldT})$ is free. 
\end{proposition}
\begin{proof}
\label{proof:prop:prop4.37_on_artin_Hida_Adjoint}
The proof follows once we show that $H^1(I_v, \boldT)$ is a free $R_\f$-module under our running assumptions, since $H^1(I_v, \overline{\boldT})\simeq H^1(I_v, \boldT)\otimes_{R_\f}R_\f[[\Gamma_{\QQ}]]$. The freeness of  $H^1(I_v, \boldT)$, which we explain below, is proved in a manner similar to \cite[Proposition 4.37]{On_Artin_formalism_for_p-adic_Garrett--Rankin_Lfunctions}.

If $\rho_{\f}(I_v)$ has finite cardinality, then by \cite[Corollary 4.34]{On_Artin_formalism_for_p-adic_Garrett--Rankin_Lfunctions}, $p$ does not divide the cardinality of $\rho_{\f}(I_v)$. We conclude that the cardinality of $\rho_{T^*_\f \otimes_{R_\f} T^*_\f}(I_v)$ is also finite with order coprime to $p$. By \cite[Corollary 4.23]{On_Artin_formalism_for_p-adic_Garrett--Rankin_Lfunctions},  
$$H^1(I_v, T^*_\f \otimes T^*_\f)\simeq H^1(I_v, T^*_\f \otimes T^*_\f(1+\psi)) \simeq H^1(I_v, \boldT) \oplus H^1(I_v, \wedge^2\,T^*_\f)$$
is free. In the scenario when the cardinality of $\rho_{T^*_\f}(I_v)$ is infinite, we apply \cite[Corollary 4.26]{On_Artin_formalism_for_p-adic_Garrett--Rankin_Lfunctions}, to conclude that
$$H^1(I_v, T^*_\f \otimes T^*_\f(1+\psi)) \simeq H^1(I_v, \boldT) \oplus H^1(I_v, \wedge^2\,T^*_\f)$$ is free. In either case, it follows that from \cite[Theorem 3.5]{rotmanHomoAl} that $H^1(I_v, \boldT)$ is projective (hence free, since the coefficient ring is local), as it is a direct summand of a free $R_\f$-module. Our proof is complete.
\end{proof}

\subsubsection{Greenberg Selmer Complexes} 
\label{subsubsec_2024_07_18_1333}
Let us put $R:=R_\f[[\Gamma_{\QQ}]]$, and as before define the $G_\QQ$-modules $R^\iota$ and $R^\sharp:={\rm Hom}_R (R^\iota,R)$. We let $\overline{\boldT}:=\boldT \otimes_{R_\f} R^\iota\,,$
and ${\mathcal F}^i\overline{\boldT}:={\mathcal F}^i\boldT \otimes_{R_\f} R^\iota.$ Associated to the data $(\overline{\boldT}, {\mathcal F}^i\overline{\boldT})$, we have the Greenberg Selmer complex $\boldR\Gamma(\overline{\boldT}, {\mathcal F}^i\overline{\boldT})$ for $i=1,2$. Likewise, we have the Selmer complexes $\boldR\Gamma(T,{\mathcal F}^iT)$ and $\boldR\Gamma(\overline{T},{\mathcal F}^i\overline{T})$.

\begin{lemma}
\label{lemma:GreenbergSelCom_Symm_ordinary}
Suppose that $i\in \{1,2\}$.
\item[i)] For any ideal $I$ of R, we have $\boldR\Gamma(\overline{\boldT}, {\mathcal F}^i\overline{\boldT}) \otimes^{\bbLL}_{R} R/I \simeq \boldR\Gamma(\overline{\boldT}/I, {\mathcal F}^i\overline{\boldT}/I).$
\item[ii)] Suppose that we are in the situation of Proposition~\ref{prop:prop4.37_on_artin_Hida_Adjoint}. Then, 
$$\boldR\Gamma(\overline{\boldT}, {\mathcal F}^i\overline{\boldT}) \in \cD^{[1,2]}_{\rm perf}(R).$$
\end{lemma}
\begin{proof}
\label{proof:lemma:GreenbergSelCom_Symm_ordinary}
This is \cite[Proposition 4.46]{On_Artin_formalism_for_p-adic_Garrett--Rankin_Lfunctions}, where we set $M=\boldT$ and $R=R_\f$ in \cite[\S4.4.1]{On_Artin_formalism_for_p-adic_Garrett--Rankin_Lfunctions} and we take in $U^+_v(T)={\mathcal F}^i\overline{\boldT}$ for $v=p$.
\end{proof}

\subsubsection{Bounding the Selmer group}
Let us put $T^{\rm alt}:=\bigwedge^2 T_{\cO}(f)^*(1+\psi)\,$, so that $T_\cO(f)^*\otimes T_\cO(f)^* (1+\psi)\,\simeq \, T \oplus T^{\rm alt}\,.$ We will work under the following alternative big image assumption to \eqref{item_p_large} (primarily to employ the Euler system machinery to prove Theorem~\ref{thm:Theorem 5.4.1_THm5.4.2_Loeffler_2016_11072024_1352} below):

\begin{itemize}
\item[(\mylabel{item_BI_Sym}{$\,{\textbf{BI}}_{\rm Sym}\,$})] $p\geq 7$ and $E \simeq \QQ_p$. Moreover, ${\rm im}(G_{\QQ} \to {\rm GL}(V_E(f)))$ contains of conjugate of ${\rm SL}_2(\ZZ_p)$, and there exists $u \in (\ZZ/N_fN_\psi)^\times$ such that $\psi(u) \neq \pm 1\pmod \frakm_E$ and $\varepsilon^{-1}_f\psi(u)$  is a square modulo $\frakm_E$. 
\end{itemize}

\begin{theorem}[Theorem 5.4.1 and 5.4.2 in \cite{Loeffler_2016}]
\label{thm:Theorem 5.4.1_THm5.4.2_Loeffler_2016_11072024_1352}
Let $\eta$ be a character of $\Delta_{\QQ}$ with $\eta(-1)= \psi(-1).$ Then:
\item[i)] The Greenberg Selmer group ${e_\eta}H^1_{{\rm Iw, Gr}, 1}(\QQ(\zeta_{p^\infty}), T)$ is free of rank $1$ over ${e_\eta}\cO_E[[\Gamma_{\QQ}]]$.
\item[ii)] The Greenberg Selmer group  ${e_\eta}H^1_{{\rm Gr},1}(\QQ(\zeta_{p^\infty}),T^\vee(1))^\vee$ is a torsion ${e_\eta}\cO_E[[\Gamma_{\QQ}]]$-module. 
\end{theorem}

We refer the reader to \cite[\S3.4]{Loeffler_2016} for the definition of the classical Greenberg Selmer groups that appear in the statement of Theorem~\ref{thm:Theorem 5.4.1_THm5.4.2_Loeffler_2016_11072024_1352}, that come associated to the data \eqref{eqn_2024_07_18_1331}. In the subsequent sections, we will explain how to recover them (up to negligible error) as the cohomology of our Selmer complexes.

\subsubsection{Comparison with Selmer Complexes}
In this subsection, we will compare the cohomology of the Selmer Complexes we have introduced in \S\ref{subsubsec_2024_07_18_1333} with the Greenberg Selmer groups $H^1_{{\rm Iw, Gr}, i}(\QQ(\zeta_{p^\infty}), T)$. We also have the Greenberg Selmer complex $\boldR\Gamma(\overline{T}^\vee(1), \overline{{\mathcal F}^{i}T^{\vee}}(1))$. 

The following is a direct consequence of \cite[Lemma 9.6.3]{nekovar06}.

\begin{lemma} For $i=1,2$:
\label{lemma:comp_greenberg_SelComp_Greenberg_SelGroup}
\item[i)] We have an isomorphism
\begin{align}
\begin{aligned}
\label{eq:01072024_1356}
 \boldR^1\Gamma(\overline{T},  {\mathcal F}^i\overline{T}) \xrightarrow{\,\,\sim\,\,}  H^1_{{\rm Iw, Gr}, i}(\QQ(\zeta_{p^\infty}), T) \,.
\end{aligned}
\end{align}

\item[ii)] We have a surjective map
\begin{align}
\begin{aligned}
\label{eq:01072024_1356_2}
 \boldR^1\Gamma(\overline{T}^\vee(1),  \overline{{\mathcal F}^{i}T^{\vee}}(1)) \lra  H^1_{{\rm Gr}, i}(\QQ(\zeta_{p^\infty}), T^\vee(1)) 
\end{aligned}
\end{align}
with cotorsion kernel.
\end{lemma}

\begin{proof}
\item[i)] This follows from \cite[Lemma 9.6.3]{nekovar06}, which reads:
\begin{align}
\begin{aligned}
\label{eq:18072024_1356}
0= (\overline{T}/{\mathcal F}^i\overline{T})^{G_{\QQ_p}}\longrightarrow \boldR^1\Gamma(\overline{T},  {\mathcal F}^i\overline{T}) \longrightarrow  H^1_{{\rm Iw, Gr}, i}(\QQ(\zeta_{p^\infty}), T) \longrightarrow 0\,.
\end{aligned}
\end{align}
We remark that $H^1_{\rm Iw}(\QQ(\zeta_{p^\infty}), T) \simeq H^1(\QQ, \overline{T})$,\quad  $H^1_{\rm Iw}(\QQ_{p,\infty}, {\mathcal F}^iT) \simeq H^1(\QQ_p, {\mathcal F}^i\overline{T})$,\quad  and $ H^1_{\rm Iw}(\QQ_{p,\infty}, T/{\mathcal F}^iT) \simeq H^1(\QQ_p, \overline{T/{\mathcal F}^iT})$ by \cite[Propostion 8.3.5]{nekovar06}. We also note that $S^{\rm str}_{X}(\QQ)$ in \cite[Lemma 9.6.3]{nekovar06} coincides with $H^1_{{\rm Iw, Gr}, i}(\QQ(\zeta_{p^\infty}), T)$ when  $X=\overline{T}$  and $X^+_v=\overline{{\mathcal F}^i{T}}$.

\item[ii)] This also follows from \cite[Lemma 9.6.3]{nekovar06}, which gives rise to an exact sequence
\begin{align}
\begin{aligned}
\label{eq:18072024_1355}
H^0(G_{\QQ_p}, \overline{T}^{\vee}(1)/\overline{{\mathcal F}^{3-i}T}^{\vee}(1))\longrightarrow \boldR^1\Gamma(\overline{T}^{\vee}(1),  {\mathcal F}^i\overline{T}^{\vee}(1)) \longrightarrow  H^1_{{\rm Gr}, i}(\QQ(\zeta_{p^\infty}), T^\vee(1)) \longrightarrow 0\,.
\end{aligned}
\end{align}
Our claim follows noting that the $R$-module $H^0(G_{\QQ_p}, \overline{T}^{\vee}(1)/\overline{{\mathcal F}^{3-i}T}^{\vee}(1))$ is cotorsion.
\end{proof}

\subsubsection{Matlis duality for Iwasawa theoretic Selmer Complexes}
\label{sec:dualites_selcompl_Hida_adjoint_reps}

By \cite[\S8.9.6.1]{nekovar06}, we have an isomorphism 
$$\boldR\Gamma(\overline{T},  {\mathcal F}^2\overline{T})^\iota \simeq  {\boldR}{\rm Hom}\left(\boldR\Gamma(\overline{T}^\vee(1),  \overline{{\mathcal F}^1T^\vee}(1)), \QQ_p/\ZZ_p\right)[-3]$$ 
which induces isomorphisms on cohomology 
\begin{align}
\begin{aligned}
\label{eq:02_07_2024_1340}
\boldR^i\Gamma(\overline{T},  {\mathcal F}^2\overline{T})^\iota &\simeq H^{j-3}\left({\bold R}{\rm Hom}(\boldR\Gamma(\overline{T}^\vee(1),  \overline{{\mathcal F}^1T^\vee}(1)), \QQ_p/\ZZ_p)\right) \\
&= \boldR^{3-j}\Gamma(\overline{T}^\vee(1),  \overline{{\mathcal F}^1T^\vee}(1))^\vee 
\end{aligned}
\end{align}
where the final equality is from \cite[\S6.3.5]{nekovar06}.

\begin{theorem}
\label{thm:H^1_0_Hida_family_Sym2}
Suppose we are in the situation of Proposition~\ref{prop:prop4.37_on_artin_Hida_Adjoint}. Let $\eta$ be a character of $\Delta_{\QQ}$ with $\eta(-1)= \psi(-1)$. We assume that the Big Image Hypothesis~\eqref{item_BI_Sym} holds. Then:
\item[i)] $e_\eta\boldR^1\Gamma(\overline{\boldT}, {\mathcal F}^i\overline{\boldT})=0$ where $i=1,2$. 
\item[ii)] $e_\eta\boldR^2\Gamma(\overline{\boldT}, {\mathcal F}^2\overline{\boldT})$ is torsion over $R.$
\end{theorem}

\begin{proof}
 
Part (i) follows from Lemma~\ref{lemma:comp_greenberg_SelComp_Greenberg_SelGroup}(i), combined with Theorem~\ref{thm:Theorem 5.4.1_THm5.4.2_Loeffler_2016_11072024_1352}(i) and Nakayama's lemma as in the proof of Theorem~\ref{thm:thm:H^1_Selcom_Hida_family_Iwasa_zero}.

Part (ii) follows from Lemma~\ref{lemma:comp_greenberg_SelComp_Greenberg_SelGroup}(ii) and \eqref{eq:02_07_2024_1340} combined with  Theorem~\ref{thm:Theorem 5.4.1_THm5.4.2_Loeffler_2016_11072024_1352}(ii) and Nakayama's lemma as in the proof of Theorem~\ref{thm:thm:H^1_Selcom_Hida_family_Iwasa_zero}.
\end{proof}

\subsubsection{Functional equation}
We are now ready to prove another main result of our paper: The functional equation for the algebraic $p$-adic $L$-function associated with the Hida families of symmetric squares of modular forms. 

\begin{theorem}
\label{thm_2_21_2024_11_04}
Suppose we are in the situation of Theorem~\ref{thm:H^1_0_Hida_family_Sym2} . Then, 
$${\rm char}_{R_\f[[\Gamma_\QQ]]}\left(\boldR^2\Gamma(\overline{T},  {\mathcal F}^2\overline{T})\right)={\rm char}_{R_\f[[\Gamma_\QQ]]}\left(\boldR^2\Gamma(\overline{T}^*(1),  \overline{{\mathcal F}^1T^*}(1))\right)^\iota\,,$$
where ${\mathcal F}^1T^*:=(T/{\mathcal F}^2T)^*$.

\end{theorem}
\begin{proof}
The proof of this theorem is identical to the proof of Theorem~\ref{thm:Func_eq_hida}, where we use Lemma~\ref{lemma:GreenbergSelCom_Symm_ordinary} in place of Proposition~\ref{prop:perfect_Iwasawa_complex_cycloto}, and Theorem~\ref{thm:H^1_0_Hida_family_Sym2} in place of Theorem~\ref{thm:thm:H^1_Selcom_Hida_family_Iwasa_zero}.

\end{proof}

\part{General results in the non-ordinary case}
\label{part_general_nonord}
\section{Characteristic Ideals and Determinants} 
\label{chap:Characteristic_Ideal_and_Determinant}
\subsection{Notation and Conventions}
\label{sec:Notation_and_Conventions}
Let $K$ be a finite extension of $\QQ_p$ with ${\texttt k}$ be the residue field of $K$ and $\cO_K$ its ring of integers. We fix an algebraic closure of $\overline{K}$ of $K$ and set $G_K={\rm Gal}(\bar{K}/K)$.

\subsubsection{Coefficient rings}
\label{sec:Coefficient_rings}

The ring of analytic functions, denoted by $\cR^{[s,r]},$ is the completion of $\QQ_p[T^{\pm 1}]$ with respect to the Gauss norm. 
Note that one can have $r=\infty$, but we do not allow $s=\infty$; in which case the inner radius is $0$, and our annulus is a disc with outer radius $p^{-\frac{s}{p-1}}$.  

We set $\cH_i:=\cR^{[p^{-i}, \infty]}$ and $\cH_{E,i}=\cH_i\otimes_{\QQ_p} E$. We note that we have a natural map 
$\cH_{E,i+1}\ra \cH_{E,i}\,,$ 
which is simply the restriction of rigid analytic functions on the disc $A^1[p^{-i-1},\infty]$ (of radius $p^{-\frac{1}{p^{i+1}(p-1)}}$) to the smaller disc $A^1[p^{-i},\infty]$ (of radius $p^{-\frac{1}{p^{i}(p-1)}}$). We put 
$\cH=\varprojlim_i \cH_{i}\,,$ 
which is explicitly given as the algebra of formal power series $f(T) \in \QQ_p[[T]]$ that converges on the open unit disk $B(0,1)=\{x \in \CC_p : |x|_p<1\}=\cup_i\, A^1[p^{-i},\infty]$. We finally put $\cH_E=\cH\otimes_{\QQ_p} E$.
We define $\cR^{[s,r]}(\Gamma^0_K)$ by formally substituting $\gamma_K-1$ in place of $T$ in the ring $\cR^{[s,r]}$. We similarly define $\cH_{i}(\Gamma_K^0)$, $\cH(\Gamma_K^0)$, $\cH_{E,i}(\Gamma^0_K)$, and $\cH_E(\Gamma_K^0)$.  We put 
$$\cH_{i}(\Gamma_K):=\cR^{[p^{-i}[\Gamma_K:\Gamma_K^0], \infty]}(\Gamma^0_K) \otimes_{\ZZ_p} \ZZ_p[\Delta_K]\,,\quad \cH_{E,i}(\Gamma_K)=\cH_{i}(\Gamma_K)\otimes_{\QQ_p}E\,, 
$$ 
and
$\cH(\Gamma_K):={\varprojlim_i}\, \cH_{i}(\Gamma_K)\,, \hbox{and} \quad \cH_E(\Gamma_K):=\cH(\Gamma_K)\otimes_{\QQ_p}E\,.$ 
\begin{remark}
\label{remark:Union_Max(H_i)_quasi-Stein}
Following \cite{pottharst}, we set $X_i={\rm Max}(\cH_i(\Gamma^0_K))$ where ${\rm Max}(\cH_i(\Gamma^0_K))$ denotes the associated rigid analytic space. We view $X_{\rm cyc}=\bigcup_{n \geq 1}X_i$ as a quasi-Stein rigid analytic space, which is a union of open affinoid discs.
\end{remark}

For $?=\{\}, 0$, we let $\Lambda^\iota_{\cO_E}(\Gamma^?_K)$ (resp. $\Lambda^\sharp_{\cO_E}(\Gamma^?_K)$) denote the free $\Lambda_{\cO}(\Gamma^?_K)$-module of rank one that is equipped with the $G_K$-module structure via $G_K\to \Gamma_K^?\xrightarrow{\gamma\,\mapsto\, \gamma^{-1}}  \Gamma_K^? \hookrightarrow \Lambda_{\cO}(\Gamma^?_K)^\times\,$ (resp. via $G_K\to \Gamma_K^? \hookrightarrow \Lambda_{\cO}(\Gamma^?_K)^\times$)\,. We put $\cH_E^\bullet(\Gamma^?_K)):=\Lambda^\bullet_{\cO_E}(\Gamma^?_K)\otimes_{\Lambda_{\cO}(\Gamma^?_K)}\cH_E(\Gamma^?_K))$ for $\bullet=\iota,\sharp$.

\subsection{Determinants of coadmissible modules}
Let $M=\varprojlim M_i$ be a torsion coadmissible $\cH_E(\Gamma^0_K)$-module in the sense of \cite{pottharst}, so that each $M_i$ is a finitely generated $\cH_{E,i}(\Gamma^0_K)$-module.  Then $M$ admits a  projective resolution 
$P^\bullet=(P_1\ra P_0)$ of length $1$ by free $\cH_E(\Gamma^0_K)$-modules of the same rank $n$ by \cite[Lemma 3.28]{Arithmetic_critical_p-adic_L-functions}. 

We put 
${\rm det}_{\cH_{E}(\Gamma^0_K)}(P_m):=\bigwedge^n P_m\simeq \cH_{E}(\Gamma^0_K)$ for $m=0,1$, and define
$${\rm det}_{\cH_{E}(\Gamma^0_K)}(M):={\rm det}(P^\bullet)={\rm det_{\cH_{E}(\Gamma^0_K)}}(P_1)^{-1} \otimes {\rm det_{\cH_{E}(\Gamma^0_K)}}(P_0)\,.$$
 By slight abuse, we denote a generator of ${\rm det_{\cH_{E}(\Gamma^0_K)}}(M)$ also by the same symbol. Observe that such a generator is only defined up to units, so it is prudent to note
 ${\rm det}_{\cH_{E}(\Gamma^0_K)}(M)\in \cH_E(\Gamma_K^0)/\cH_E(\Gamma_K^0)^\times\,.$
By a further abuse of notation, we denote in what follows the ideal generated by $\det_{\cH_{E}(\Gamma^0_K)}(M)$ also by $\det_{\cH_{E}(\Gamma^0_K)}(M)$, which doesn't depend on any choice. 

We similarly define the ideal ${\rm det_{\cH_{E,i}(\Gamma^0_K)}}(M_i)\subset \cH_{E,i}(\Gamma^0_K)$ for a finitely generated torsion $\cH_{E,i}(\Gamma^0_K)$-module $M_i$.

\begin{proposition}
\label{prop:det_torsion_M_Over_H_E_char_Ideal_M}
Suppose that $M$ is a torsion coadmissible $\cH_E(\Gamma^0_K)$-module. Then,
$${\rm det}_{\cH_E(\Gamma^0_K)}(M)={\rm char}_{\cH_E(\Gamma^0_K)}(M)\,,$$
where ${\rm char}_{\cH_E(\Gamma^0_K)}(M)$  is defined as in \cite[p.5]{pottharst}.
\end{proposition}

\begin{proof}
\label{proof:prop:det_torsion_M_Over_H_E_char_Ideal_M}
Since this is well-known to experts and it follows from Proposition~\ref{prop:det_torsion_M_Over_H_A_char_Ideal_M} below as a degenerate case (with $A=E$), we omit its proof.
\end{proof}

\subsection{Algebra of affinoid-valued distributions on $\Gamma_K^0$}
\label{sec:Coefficient_rings_affinoid}

\subsubsection{Affinoid algebras} As before, we let $E$ denote a fixed finite extension of $\QQ_p.$ The Tate algebra in $n$ variables over $E$ is 
$$T_n(E):=\{f(x)= \sum_{I}c_Ix^I : c_I \in E, |c_I| \mapsto 0, |I| \mapsto \infty\}$$ 
where the sum runs over tuples $I = (i_1, . . . , i_n)$ of natural numbers,  $x^I:=x_1^{i_1}...x_n^{i_n},$ and $|I| := i_1 + ... + i_n.$ When the underlying field is understood, we will write $T_n$ instead of $T_n(E).$
We record some of the useful (and well-known) properties of Tate-algebras: the ring $T_n$ is a Noetherian regular unique factorization domain, the Krull dimension of $T_n$ is $n$, for each maximal ideal $\mathfrak{m}$ of $T_n$, the residue field $T_n/\mathfrak{m}$ is a finite extension of $E$. An affinoid algebra is a $E$-algebra $A$ of the form $T_n/\mathfrak{a}$ for some ideal $\mathfrak{a}$. See \cite{Tate_affinoid_RICARDO,affinoid_kedlaya,BGR84} for further details on Tate-Algebras.

\subsubsection{}

Let us put $\cH_{A,i}(\Gamma^0_K):=A \widehat{\otimes}_{\QQ_p}\cH_i(\Gamma^0_K),$ and define 
$$\cH_A(\Gamma^0_K):=A \widehat{\otimes}_{\QQ_p} \cH(\Gamma^0_K)=A \widehat{\otimes}_{\QQ_p} \underset{i}{\varprojlim} \cH_i(\Gamma^0_K)=\underset{i}{\varprojlim}A \widehat{\otimes}_{\QQ_p}\cH_i(\Gamma^0_K)=\cap_{i\geq 0}A \widehat{\otimes}_{\QQ_p}\cH_i(\Gamma^0_K)\,.$$ 
We also set $\cH_A(\Gamma_K):=\cH_A(\Gamma^0_K) \otimes_{\ZZ_p} \ZZ_p[\Delta_K].$

We remark that ${\rm Max}(A) \times X_{\rm cyc}$ is rigid analytic spaces which is a disjoint union of open affinoid discs by Remark~\ref{remark:Union_Max(H_i)_quasi-Stein}. We also note that each ${\rm Max}(A) \times X_i$ coincides with the spaces denoted by $Y_i$ in \cite[\S1C]{jayanalyticfamilies}. For more details on the analytic ring $\cH_A(\Gamma^0_K)$ and on affinoid algebras see \cite[\S2.1]{KPX2014} and \cite{affinoid_kedlaya}. 

\subsection{Cylindrical characteristic ideal of a torsion coadmissible $\cH_A(\Gamma^0_K)$-module}
\label{subsubsec_2023_11_16_1245}
We introduce the notion of the characteristic ideal of a finitely generated torsion module over rings of higher dimension. We first carry this out over the Noetherian domains $\cH_{A,i}(\Gamma^0_K)$, where our main input is \cite[Proposition 2.10.19]{nekovar06}.


\begin{defn}
\label{defn:char_ideal_f_g_M_notherian_H_Ai}
Let $M$ be a finitely generated torsion $\cH_{A,i}(\Gamma^0_K)$-module. We define the characteristic ideal of $M$ on setting
\begin{align}
\begin{aligned}
{\rm char}_{\cH_{A,i}(\Gamma^0_K)}(M):&=
\prod_{{\rm ht}(\fp)=1} \fp^{n(\fp, 1)}\cdots \fp^{t\cdot n(\fp, t)} \\ 
&=\prod_{{\rm ht}(\fp)=1\atop t\geq 1} \fp^{e_p} \qquad \text{ where } e_\fp=n(\fp, 1)+...+t\cdot n(\fp, t)\,,
\end{aligned}
\end{align}
where the integers $n(\fp,i)$ are determined via \cite[Proposition 2.10.19]{nekovar06}.
\end{defn}

Here, we remark that ring $R$ in \cite[Proposition 2.10.18, Proposition 2.10.19]{nekovar06} being local plays no role in the proofs. We also remark that when the affinoid Algebra $A$ is a Tate algebra (as it will be so for our purposes in the latter sections), then $A$ satisfies condition Serre's condition $(R_1)$ (i.e. its localizations at all primes of height at most one are regular). We also note that $\cH_{E,i}$ satisfies condition $(R_1)$, therefore also $\cH_{A,i}.$

\begin{defn}
\label{defn:char_Ideal_coadmissible_H_A_module}
Suppose $M=\varprojlim_iM_i$ is a torsion coadmissible $\cH_A(\Gamma^0_K)$-module. Note then that each $M_i$ is a finitely generated torsion $\cH_{A,i}(\Gamma^0_K)$-module. We define the (cylindrical) characteristic ideal of such $M$ on setting
${\rm char}_{\cH_A(\Gamma^0_K)}(M):=\varprojlim_i {\rm char}_{\cH_{A,i}(\Gamma^0_K)}(M_i)\,,$
where 
${\rm char}_{\cH_{A,i}(\Gamma^0_K)}(M_i)= \prod_{k=1}^{s_i}\fp_{k_i}^{n_{k_i}}$ 
is the characteristic ideal of the torsion $\cH_{A,i}(\Gamma^0_K)$-module $M_i$ given as in Definition~\ref{defn:char_ideal_f_g_M_notherian_H_Ai}.
\end{defn}

We remark that the definition of ${\rm char}_{\cH_A(\Gamma^0_K)}(M)$ makes sense:
Let us denote the natural map $\cH_{A,j}(\Gamma^0_K) \hookrightarrow \cH_{A,i}(\Gamma^0_K)$ by $\psi_j^i$ for $j \geq i.$ Then,
\begin{align}
\begin{aligned}
\label{eq:char_transition_map_H_A}
\psi_j^i(\fp_{k_j}) = \begin{cases}
    \fp_{k_j}\cH_{A,i}(\Gamma_K^0)\qquad\qquad \hbox{ if } {k_j} \leq s_{j}\,,\\
\cH_{A,i}(\Gamma^0_K)\qquad\qquad\qquad\text{otherwise}\,.
\end{cases} 
\end{aligned}
\end{align}
It follows from Definition~\ref{defn:char_Ideal_coadmissible_H_A_module} and \eqref{eq:char_transition_map_H_A} that
\begin{align}
\begin{aligned}
\label{eq:char_M_(i+1)_char_M_i_H_Ai}
{\rm char} M_{i+1} \otimes_{\cH_{A,{i+1}}}\cH_{A,i}
={\rm char}(M_i)\,,
\end{aligned}
\end{align}
so that $\varprojlim_i {\rm char}_{\cH_{A,i}(\Gamma^0_K)}(M_i)$ is a coadmissible $\cH_A(\Gamma^0_K)$-module. 

\subsection{Determinants (bis)}

Suppose $M$ is a torsion coadmissible $\cH_A(\Gamma^0_K)$-module with a  projective resolution $P^\bullet=(P_1 \xrightarrow{} P_0)$ of length $1$ by free $\cH_A(\Gamma^0_K)$-modules (necessarily of the same rank, say $n$).  We put 
${\rm det}_{\cH_{A}(\Gamma^0_K)}(P_m):=\bigwedge^n P_m\simeq \cH_{A}(\Gamma^0_K)$ for $m=0,1$ and define
${\rm det}_{\cH_{A}(\Gamma^0_K)}(M):={\rm det}(P^\bullet)={\rm det_{\cH_{A}(\Gamma^0_K)}}(P_1)^{-1} \otimes {\rm det_{\cH_{A}(\Gamma^0_K)}}(P_0)\,.$
 By slight abuse, we denote a generator of ${\rm det_{\cH_{A}(\Gamma^0_K)}}(M)$ also by the same symbol. 

We similarly define ${\rm det_{\cH_{A,i}(\Gamma^0_K)}}(M_i)\subset \cH_{A,i}(\Gamma^0_K)$ for a finitely generated torsion $\cH_{A,i}(\Gamma^0_K)$-module $M_i$. 

\begin{proposition}
\label{prop:det_torsion_M_Over_H_A_char_Ideal_M}
Suppose that $A$ is Tate algebra and we have a torsion coadimissible $\cH_A(\Gamma^0_K)$-module $M$ with projective resolution $P^\bullet=(P_1 \xrightarrow{} P_0)$ of dimension $1$ by finitely generated free $\cH_A(\Gamma^0_K)$-modules $P_0$ and $P_1$ of (necessarily) the same rank (say $n$).  Then we have 
$${\rm det}_{\cH_{A}(\Gamma^0_K)}(M)={\rm char}_{\cH_A(\Gamma^0_k)}(M).$$ 

\end{proposition}

\begin{proof}
\label{proof:prop:det_torsion_M_Over_H_A_char_Ideal_M}

To say that $M$ is a coadimissible $\cH_A(\Gamma^0_K)$-module amounts to the requirement that $M=\varprojlim_iM_i$ with each $M_i$ is finitely generated torsion $\cH_{A,i}(\Gamma^0_K)$-module, and the map $M_{i+1} \xrightarrow{} M_i$ induces an isomorphism $M_{i+1} \otimes_{\cH_{A,i+1}(\Gamma^0_K)} \cH_{A,i}(\Gamma^0_K) \simeq M_i$ for each $i$.

The structure theorem for torsion $\cH_{A,i}(\Gamma^0_K)$-modules provides an exact sequence (see \S\ref{subsubsec_2023_11_16_1245})
\begin{equation}
\label{eq:torsion_H_Ai_M_i_structureThm}
0\lra A_i\lra M_i\lra\oplus_{k=1}^{s_i} \cH_{A,i}(\Gamma^0_K)/\fp_{k_i}^{n_{k_i}}\cH_{A,i}(\Gamma^0_K)\lra N_i \lra  0
\end{equation}
where each $\fp_{k_i}$ is a  height-$1$ prime ideal of $\cH_{A,i}(\Gamma^0_K)$, $A_i$ and $N_i$ are pseudo-null $\cH_{A,i}(\Gamma^0_K)$-modules and ${\rm char}_{\cH_{A,i}(\Gamma^0_K)}(M_i)= \prod_{k=1}^{s_i}\fp_{k_i}^{n_{k_i}}$. We are set to prove that 
\begin{equation}
\label{eq:det_M_i_equal_char_HAi}{\rm det}_{\cH_{A,i}(\Gamma^0_K)}(M_i)={\rm char}_{\cH_{A,i}(\Gamma^0_K)}(M_i).
\end{equation}

We know, by its definition, that ${\rm det}_{\cH_{A,i}(\Gamma^0_K)}(M_i)$ is a principal proper ideal of $\cH_{A,i}(\Gamma^0_K)$. It is a well-known fact that if a ring $R$ is a Noetherian normal domain, then every associated prime $\fp\in {\rm Spec}(R)$  of a principal ideal has height $1$. In our situation, for a fixed $i$, the Tate algebra $\cH_{A,i}(\Gamma^0_K)$ is a Noetherian normal ring and therefore, every prime ideal $\fp \in \cH_{A,i}(\Gamma^0_K)$ associated to ${\rm det}_{\cH_{A,i}(\Gamma^0_K)}(M_i)$ has height equal to $1$. Since ${\rm char}_{\cH_{A,i}(\Gamma^0_K)_{\fp}}(M_i)$ is, by definition, a product of height-1 primes, it suffices to prove that the $\fp$-primary parts of both sides in \eqref{eq:det_M_i_equal_char_HAi} agree for every height $1$ prime ideal $\fp$ of $\cH_{A,i}(\Gamma^0_K)$. 

Note that the localization $\cH_{A,i}(\Gamma^0_K)_{\fp}$ of the Tate algebra $\cH_{A,i}(\Gamma^0_K)$ at a height-$1$ prime ideal $\fp$ is a discrete valuation ring, as the ring $\cH_{A,i}(\Gamma^0_K)$ is normal. For each height-1 prime $\fp\subset \cH_{A,i}(\Gamma^0_K)$, we have  
$${\rm det}_{\cH_{A,i}(\Gamma^0_K)_{\fp}}(M_i)_{\fp}={\rm char}_{\cH_{A,i}(\Gamma^0_K)_{\fp}}(M_i)_{\fp}$$
 by the proof of \cite[Proposition 4.44]{On_Artin_formalism_for_p-adic_Garrett--Rankin_Lfunctions}. We conclude that 
 $${\rm det}_{\cH_{A,i}(\Gamma^0_K)}(M_i)={\rm char}_{\cH_{A,i}(\Gamma^0_K)}(M_i)$$ 
 by setting $I={\rm char}_{\cH_{A,i}(\Gamma^0_K)}(M_i)$ and $J={\rm det}_{\cH_{A,i}(\Gamma^0_K)}(M_i)$ in \cite[Lemma 4.43]{On_Artin_formalism_for_p-adic_Garrett--Rankin_Lfunctions}.

We denote the natural map $\cH_{A,j}(\Gamma^0_K) \hookrightarrow \cH_{A,i}(\Gamma^0_K)$ by $\psi_j^i$ for $j \geq i.$ Then we have
\begin{align}
\begin{aligned}
\label{eq:det_transition_map_H_A}
\psi_j^i(\fp_{k_j}) = \begin{cases}
    \fp_{k_j}\cH_{A,i}(\Gamma_K^0)\qquad\qquad \hbox{ if } {k_j} \leq s_{j}\,,\\
\cH_{A,i}(\Gamma^0_K)\qquad\qquad\qquad\text{otherwise}\,.
\end{cases} 
\end{aligned}
\end{align}
It follows from \eqref{eq:det_M_i_equal_char_HAi} and \eqref{eq:det_transition_map_H_A} that
\begin{align}
\begin{aligned}
\label{eq:det_M_(i+1)_det_M_i_H_Ai}
\det M_{i+1} \otimes_{\cH_{A,{i+1}}}\cH_{A,i}
={\rm det}(M_i).
\end{aligned}
\end{align}
As before, we observe that $\varprojlim_i {\rm det}(M_i)$ is a coadmissible $\cH_A(\Gamma^0_K)$-module. We claim that 
$${\rm det}_{\cH_A(\Gamma^0_K)}(M)=\varprojlim_i {\rm det}_{\cH_{A,i}(\Gamma^0_K)}(M_i)=\varprojlim_i {\rm char}_{\cH_{A,i}(\Gamma^0_K)}(M_i)={\rm char}_{\cH_A(\Gamma^0_K)}(M).$$ 
We already proved that the second equality in \eqref{eq:det_M_i_equal_char_HAi} and third equality follows from the definition of ${\rm char}_{\cH_A(\Gamma^0_k)}(M).$ So it remains to prove that 
$${\rm det}_{\cH_A(\Gamma^0_K)}(M)=\varprojlim_i {\rm det}_{\cH_{A,i}(\Gamma^0_K)}(M_i).$$

Recall that, we are given a projective resolution of $M$ 
$$0 \lra P_1 \xrightarrow{\,\,\phi\,\,} P_0 \lra M\lra 0.$$
by finite $\cH_A(\Gamma^0_K)$-modules of the same rank $n$. As $\cH_{A,i}(\Gamma^0_K)$ is flat over $\cH_{A}(\Gamma^0_K)$, then we have
\begin{equation}
\label{eq:proj_reso_M_i_H_A,i}
0 \longrightarrow{} P_1 \otimes_{\cH_{A}(\Gamma^0_K)} \cH_{A,i}(\Gamma^0_K) \xrightarrow{\phi_i} P_0 \otimes_{\cH_{A}(\Gamma^0_K)} \cH_{A,i}(\Gamma^0_K) \longrightarrow{} M_i \longrightarrow{} 0
\end{equation}
since $M_i=M \otimes_{\cH_{A}(\Gamma^0_K)} \cH_{A,i}(\Gamma^0_K).$ For simplicity, let us set $P_{m,i}=P_m \otimes_{\cH_{A}(\Gamma^0_K)} \cH_{A,i}(\Gamma^0_K)$ for $m=0,1$. Then each $P_{m,i}$ is a free $\cH_{A,i}(\Gamma^0_K)$-module of rank $n$, as $P_m$ is the same over $\cH_{A}(\Gamma^0_K)$.

From exact sequence \eqref{eq:proj_reso_M_i_H_A,i} and the definition of the determinant functor, it then follows that we have ${\rm det}_{\cH_A(\Gamma^0_K)}(M)=\varprojlim_i{\rm det}_{\cH_{A,i}(\Gamma^0_K)}(M_i)$ using arguments similar to those utilised in the proof of Proposition \ref{prop:det_torsion_M_Over_H_E_char_Ideal_M}. This completes the proof of our proposition.

\end{proof}


\section{\texorpdfstring{$p$}{}-adic Hodge theory and local cohomology}
\label{chap:Local_cohomology_and_phi_gamma_module}
We recall basic input we shall rely on from the theory of $(\varphi,\Gamma)$-modules and its relation to local Galois cohomology, as well as to $p$-adic Hodge theory. We refer readers to \cite{ExceptionalZerosOfnon-ordinaryL-funcs,berger2004equations, p-adicHodgetheorybenoisheights} for more background on these topics.

\subsection{Local cohomology of \texorpdfstring{$(\varphi,\Gamma_K)$-}{}modules}
\label{sec:cohomology_of_varphi_Gamma_K_module}
We retain the notation and conventions from the previous sections. Let $\boldD$ be a $(\varphi, \Gamma_K)$-module over the Robba ring $\cR_A$ (cf. \cite{p-adicHodgetheorybenoisheights, ExceptionalZerosOfnon-ordinaryL-funcs}).
Let us set $\Delta_K={\rm Gal}(K(\zeta_p)/K)$, which we may and will view as a subgroup of $\Gamma_K$, so that 
$\Gamma_K=\Delta_K \times \Gamma^{0}_K\,, \hbox{ and }\Gamma^{0}_K \cong \ZZ_p.$ 
Let $\gamma_K$ be a fixed topological generator of $\Gamma^{0}_K$. We define the complex
$C^{\bullet}_{\gamma_K}(\boldD)\,:\,\quad \boldD^{\Delta_K} \xrightarrow{\gamma_K-1} \boldD^{\Delta_K}$ where the terms are placed in degree $0$ and $1.$ We also consider the complex given by
$$C^{\bullet}_{\varphi,\gamma_K}(\boldD)\,:\, \boldD^{\Delta_K} \xrightarrow{\,d_0\,} \boldD^{\Delta_K} \oplus \boldD^{\Delta_K} \xrightarrow{\,d_1\,} \boldD^{\Delta_K}$$
(in degrees $0,1,2$), where 
$d_0(x)=((\varphi-1)x,(\gamma_K-1)x)\,\hbox{, and } d_1(x,y)=((\gamma_K-1)x-(\varphi-1)y)\,.$ 
The complex $C^{\bullet}_{\varphi,\gamma_K}(\boldD)$ is known as the ``Fontaine--Herr'' complex.  For a given $(\varphi,\Gamma_K)$-module $\boldD$, we denote by 
$\boldR\Gamma(K,D):=[C^{\bullet}_{\varphi,\gamma_K}(\boldD)]$
the corresponding object in the derived category $\cD(A)$ of $A$-modules. The cohomology of $\boldD$ is denoted by $H^i(\boldD):=\boldR^i\Gamma(K,D)=H^i(C^\bullet_{\varphi, \gamma_K}(\boldD))$ (where $i=0,1,2$).

\subsection{The complex \texorpdfstring{$K^\bullet(V)$}{}}
\label{sec:Complex_K_V}
Let $V$ be a $G_K$-representation over $A$. Let us set $V^{\dagger}_{\rm rig}=V \otimes_A \widetilde{B}^{\dagger}_{{\rm rig},A}$, where $\widetilde{B}^{\dagger}_{{\rm rig},A}$ is as in \cite[~p. 43]{p-adicHodgetheorybenoisheights}, and following op. cit., let us put
\begin{equation}
\label{eq:complex_K_bullet_V}
K^\bullet(V)={\rm Tot}\left (C^\bullet(G_K,V^{\dagger}_{\rm rig}) \xrightarrow{\varphi-1} C^\bullet(G_K,V^{\dagger}_{\rm rig})\right )\,.
\end{equation}
We consider the map $C^\bullet_{\gamma_K}(D^{\dagger}_{\rm rig}(V)) \xrightarrow{\alpha_V} C^\bullet(G_K,V^{\dagger}_{\rm rig})$ and $C^\bullet(G_K,V) \xrightarrow{\xi_V} K^\bullet(V)$ given as in op. cit. By \cite[Proposition 2.5.2]{p-adicHodgetheorybenoisheights}, both the map $\alpha_V$ and $\xi_V$ are quasi-isomorphisms.


\subsection{Local Iwasawa cohomology over affinoid algebras}
\label{sec:Local_Iwasawa_Cohomology_over_affinoids}

For a $(\varphi, \Gamma_K)$-module $\boldD$ over $\cR_A$, we define its cyclotomic deformation by 
$\overline{\boldD}:=\boldD \otimes_{A} {\cH_A(\Gamma^0_K)}^\iota={\varprojlim_i}\,\boldD \otimes \cH_{A,i}(\Gamma^0_K)^\iota\,.$
Similarly, we define 
$\overline{\boldD}^*:=\boldD^* \otimes_{A} {\cH_A^\sharp(\Gamma^0_K)}$, where $\boldD^*:={\rm Hom}_{\cR_A}(\boldD, \cR_A)$ is the dual $(\varphi,\Gamma_K)$-module. The Fontaine--Herr complexes $C^{\bullet}_{\varphi,\gamma_K}(\overline{\boldD})$ and $C^{\bullet}_{\psi,\gamma_K}(\overline{\boldD})$ are defined in a manner similar to \S\ref{sec:cohomology_of_varphi_Gamma_K_module}. 
\begin{defn}
\label{defn:Iwasawa_Cohomology_phi-gamma_Over_H_A}
Let $\boldD$ be a $(\varphi,\Gamma_K)$-module over $\cR_A.$ Then we define $C^\bullet_{\psi}(\boldD)$ to be the complex $[\boldD \xrightarrow{\psi-1} \boldD]$ concentrated in degrees $1$ and $2$. Its cohomology groups $H^*_{\psi}(\boldD)$ are called the Iwasawa cohomology of $\boldD.$ 
\end{defn}

Following from \cite[p. 41]{p-adicHodgetheorybenoisheights} and \cite[p. 10]{jayanalyticfamilies}, we have the cup-product pairing: 
$$\bigcup_{\varphi, \gamma_K}: C^\bullet_{\varphi,\gamma_K}(\overline{\boldD})\otimes_{\cR_A}  C^\bullet_{\varphi,\gamma_K}(\overline{\boldD}^*(\chi_K)) \lra  C^\bullet_{\varphi,\gamma_K} (\overline{\boldD} \otimes \overline{\boldD}^*(\chi_K)) \,.$$
Noting that $G_K$ acts trivially on ${\cH_A^\iota(\Gamma^0_K)}\otimes {\cH_A^\sharp(\Gamma^0_K)}$, we can extend to a map
\begin{equation}
\label{eq:cup_product_phi-gamma_module_over_H_A_2}
\bigcup_{\varphi, \gamma_K}: C^\bullet_{\varphi,\gamma_K}(\overline{\boldD})\otimes_{\cR_A}  C^\bullet_{\varphi,\gamma_K}(\overline{\boldD}^*(\chi_K)) \lra A[-2] \otimes \cH_A(\Gamma^0_K)=\cH_A(\Gamma^0_K)[-2].
\end{equation}

We remark that all the discussion in this chapter applies to affinoid domains of arbitrary dimension.

\section{Functional equation of algebraic $p$-adic $L$-functions: abstract setting}
\label{chap:Selmer_Complexes}
In this section, we follow closely \cite[\S3.1]{p-adicHodgetheorybenoisheights} to define Selmer complexes over $\QQ_p$-affinoid algebra $A$, as well as their Iwasawa theoretic versions over the rings $\cH_E(\ZZ_p)$ and $\cH_A(\ZZ_p)$.

\subsection{Selmer complexes over \texorpdfstring{$A$}{}}
\label{sec:Selmer_Complexes_over_A}
\subsubsection{Set-up}
We let $F$ be a number field. We denote by $S_f$ (resp. $S_\infty$) the set of all non-archimedean (resp. archimedean) absolute values on $F$. As before, let us fix a prime number $p$ and a compatible system of $p$-power roots of unity $\epsilon=(\zeta_{p^n})_{n \geq 1}$. 
Let $ S \subset S_f$ be a finite subset containing the subset $S_p$ of $\fq \in S_f$ that lies above $p$. We will write $S^{(p)}$ for the complement of $S_p$ in $S.$ Let $G_{F,S}$  denote the Galois group of the maximal algebraic extension of $F$ unramified outside $S \cup S_\infty.$ For each $\fq \in S$, we fix a decomposition
group at $\fq$, which we identify with $G_{F_\fq}$. 
If $\fq \in S_p,$ we denote by $\Gamma_\fq=\Gamma_{F_\fq}$ the cyclotomic Galois group of $F_\fq$ and we fix a topological generator $\gamma_q \in \Gamma^0_\fq$.

Let us put $F^{\rm cyc}=\bigcup_{n=0}^{\infty}F(\zeta_{p^n})$ and $F_\infty=(F^{\rm cyc})^{\Delta_F}\,,$ where $\Delta_F={\rm Gal}(F(\zeta_p)/F)$.  Then 
$\Gamma_F^0:={\rm Gal}(F_\infty/F) \simeq {\rm Gal}(F^{\rm cyc}/F(\zeta_p))\simeq \ZZ_p$  and  $\Gamma_F:={\rm Gal}(F^{\rm cyc}/F)\simeq \Delta_F \times \Gamma^0_F\simeq \ZZ_p^\times\,.$ If $F$ is unramified at all primes above $p$ (e.g. when $F=\QQ$, which is the typical scenario for the main arithmetic applications of our work), we may and will identify $\Gamma_\fq$ with the Galois group 
$\Gamma_F=\Gal\left(\cup_n \QQ(\zeta_{p^n})/\QQ\right)=\varprojlim_n \Gal\left(\QQ(\zeta_{p^n})/\QQ\right)\,.$
\subsubsection{} Let $V$ be a $p$-adic representation of $G_{F,S}$ with coefficients in a $\QQ_p$-affinoid algebra $A$. We will sometimes denote $V_\fq$ for the restriction of $V$ on the decomposition group $G_{F_\fq}$ (but also often write $V$ in place of $V_\fq$). For each $q \in S_p$, we fix a $(\varphi,\Gamma_\fq)$-submodule $\boldD_\fq$ of $D^{\dagger}_{\rm rig}(V_\fq)$ and we set $\boldD=(\boldD_\fq)_{\fq \in S_p}$. We define the complex
\begin{equation}
\label{eq:def_U_dot_V_D}
\begin{matrix}
U^\bullet_\fq(V,\boldD)=
\begin{cases}
C^\bullet_{\varphi,\gamma_\fq}(\boldD_\fq), & \text{if } \fq \in S_p\\
C^\bullet_{\rm ur}(V_\fq) & \text{if } \fq \in S^{(p)}
\end{cases}
\end{matrix}
\end{equation}
where $$C^\bullet_{\rm ur}(V_\fq)\,:=\, [V^{I_\fq}_\fq \xrightarrow{{\rm Frob}_\fq-1} V^{I_\fq}_\fq], \qquad \fq \in S^{(p)}$$  
concentrated in degrees $0$ and $1$ and ${\rm Frob}_\fq\in G_{F_\fq}/I_\fq$ is the Frobenius element.
If $\fq \in S_p$, the objects $\boldR\Gamma(F_\fq,V)=[C^\bullet(G_{F_\fq},V)]$ and $\boldR\Gamma(F_\fq,\boldD_\fq)=[U^\bullet_\fq(V,\boldD)]$ belong to $\cD^{[0,2]}_{\rm perf}(A)$ by \cite[Theorems 2.3.2 and 2.4.3]{p-adicHodgetheorybenoisheights}.
\subsubsection{} For $\fq \in S^{(p)}$, we have a morphism of complexes
\begin{equation}
\label{eq:g_q_map_for_q_not_dividing_p}
g_\fq: U^\bullet_\fq(V,\boldD) \longrightarrow C^\bullet(G_{F_\fq},V)
\end{equation}
given as in \cite[p. 60]{p-adicHodgetheorybenoisheights}, and the restriction map
\begin{equation}
\label{eq:f_q_map}
f_\fq\,:\,\quad C^\bullet(G_{F,S},V) \xrightarrow{{\rm res}_\fq} C^\bullet(G_{F_\fq},V)\,.
\end{equation}
\subsubsection{} We now assume that $\fq \in S_p.$ The inclusion $\boldD_\fq \subset D^\dagger_{\rm rig}(V_\fq)$ induces a morphism
$$U^\bullet_\fq(V,\boldD)=C^\bullet_{\varphi,\gamma_\fq}(\boldD_\fq) \longrightarrow C^\bullet_{\varphi,\gamma_\fq}(D^\dagger_{\rm rig}(V_\fq)).$$ We denote by 
\begin{equation}
\label{eq:g_q_map_for_q_dividing_p}
g_\fq: U^\bullet_\fq(V,\boldD) \longrightarrow K^\bullet(V_\fq)\,,\qquad \forall\, \fq\mid p
\end{equation}
the composition of this morphism with the quasi-isomorphism 
$$\alpha_{V_\fq}:C^\bullet_{\varphi,\gamma_\fq}(D^\dagger_{\rm rig}(V_\fq)) \longrightarrow K^\bullet(V_\fq)$$ 
and by slight abuse of notation, we let
\begin{equation}
\label{eq:f_q_map_for_q_dividing_p}
f_\fq: C^\bullet(G_{F,S},V) \longrightarrow K^\bullet(V_\fq)\,,\,\qquad \forall\,\fq\mid p
\end{equation}
denote the composition of the restriction map 
$$C^\bullet(G_{F,S},V) \xrightarrow{{\rm res}_\fq} C^\bullet(G_{F_\fq},V)$$ 
with the quasi-isomorphism 
$$\xi_{V_\fq}\,:\, C^\bullet(G_{F_\fq},V) \longrightarrow K^\bullet(V_\fq)$$
given as in \cite[Proposition 2.5.2]{p-adicHodgetheorybenoisheights}.
Let us set 
\begin{equation}
\label{eq:def_K_bullet_q_V}
\begin{matrix}
K^\bullet_{\fq}(V)=
\begin{cases}
K^\bullet(V_\fq)\,, & \qquad \text{if } \fq \in S_p\\
C^\bullet(G_{F_\fq},V)\, & \qquad \text{if } \fq \in S^{(p)}\,,
\end{cases}
\end{matrix}
\end{equation}
and put
$$K^\bullet(V)=\oplus_{\fq \in S}\,K^\bullet_{\fq}(V) \qquad\hbox{ and } \qquad U^\bullet(V,\boldD)=\oplus_{\fq \in S} \, U^\bullet_{\fq}(V,\boldD).$$
Let  
$\boldR\Gamma_S(V):=[C^\bullet(G_{F,S},V)] \in \cD^{[0,3]}_{\rm perf}(A)$ denote object represented by the complex $C^\bullet(G_{F,S},V)$.
\subsubsection{The Selmer complex} With this data at hand, we define the Selmer Complex as follows:
\begin{equation}
\label{eq:SelmerComplex_for_V}
{\mathbf S}(V,\boldD)={\rm cone}\left(C^\bullet(G_{F,S},V) \oplus U^\bullet(V,\boldD) \xrightarrow{f-g} K^\bullet(V)\right)[-1]\,,
\end{equation}
where 
$f=(f_\fq)_{\fq \in S}\,, g=(g_\fq)_{\fq \in S}\,.$
We set $\boldR\Gamma(V,\boldD):=[{\mathbf S}(V,\boldD)]$ for the object in the derived category and write $\boldR^n\Gamma(V,\boldD)$ for the cohomology of $\boldR\Gamma(V,\boldD).$ 
If we assume that $[C^\bullet_{\rm ur}(V_\fq)] \in \cD^{[0,1]}_{\rm perf}(A)$ for all $\fq \in S^{(p)}$ (in what follows, we will give sufficient conditions to ensure this property), then $\boldR\Gamma(V,\boldD) \in \cD^{[0,3]}_{\rm perf}(A).$
\subsubsection{} To slightly generalize the previous constructions, we fix a finite subset $\Sigma \subset S^{(p)}$ and, for each $\fq \in \Sigma$ a locally direct summand $M_\fq$ of the $A$-module $V_\fq$ stable
under the action of $G_{F_\fq}.$ Let us put $M=(M_\fq)_{\fq \in \Sigma}$ and define
$$
\begin{matrix}
U^\bullet_\fq(V,\boldD,M)=
\begin{cases}
C^\bullet_{\varphi,\gamma_\fq}(\boldD_\fq), & \qquad \text{if } \fq \in S_p\\
C^\bullet_{{\rm ur}}(V_\fq) &\qquad \text{if } \fq \in S^{(p)} \setminus \Sigma \\
C^\bullet(G_{F_\fq},M_\fq), & \qquad \text{if } \fq \in \Sigma.
\end{cases}
\end{matrix}
$$
We denote the associated Selmer Complex by $\boldR\Gamma(V,\boldD,M)$ and $\boldR\Gamma(V,\boldD,M):=[\boldR\Gamma(V,\boldD,M)]$. 
\subsection{Duality of Selmer Complexes over $A$}
\label{sec:duality_Sel_Compl_A}
Let us consider the dual Galois representation $V^*:={\rm Hom}_A(V,A)(1)$. We define the \emph{dual local conditions} $(\boldD^{\perp}, M^{\perp})$ on $V^*(1)$ given by
$$D^\perp_\fq={\rm Hom}_{\cR_A}(\boldD^\dagger_{\rm rig}(V)/\boldD_\fq, \cR_A(\chi_\fq))\qquad \hbox{ for all } \fq \in S_p\,,$$ and $M^\perp_\fq={\rm Hom}_{A}(V_\fq/M_\fq, A(\chi_\fq)) \hbox{ for all } \fq \in \Sigma$, where $\chi_\fq$ is the restriction of the cyclotomic character $\chi: \Gamma_F \to \ZZ^\times_p$ to $\Gamma_\fq$. We denote by $f^{\perp}_\fq$ and $g^{\perp}_\fq$ the morphism of complexes corresponding to \eqref{eq:f_q_map} and \eqref{eq:g_q_map_for_q_dividing_p} for the triple $(V^*(1),\boldD^{\perp}, M^{\perp})$.

\begin{remark}
\label{rem:map_g_q_tnsored_g_q_perp_zero_at_unramified}
The composition 
\begin{equation}
\label{eq:rem:map_g_q_tnsored_g_q_perp_zero_at_unramified}
C^\bullet_{\rm ur}(V_\fq) \otimes C^\bullet_{\rm ur}(V^*_\fq(1)) \xrightarrow{g_\fq \otimes g^\perp_\fq} C^\bullet(G_{F_\fq},V) \otimes C^\bullet(G_{F_\fq}, V^*(1)) \xrightarrow{\cup_c} A[-2]
\end{equation}
is the zero map. 
\end{remark}

\subsection{Algebraic functional equation over an affinoid algebra in terms of determinants}
\label{sec:Algebraic_Functional_Equation_SelComplex}
The morphism in ${\bigcup_{V,\boldD,M}}$ (cup-product pairing) in \cite[Theorem 3.1.5(iii)]{p-adicHodgetheorybenoisheights} induces a homomorphism 
$${\bigcup_{{-,V,\boldD,M}}}_*\,:\,\quad\boldR\Gamma(V,\boldD,M) \longrightarrow \boldR{\rm Hom}_A(\boldR\Gamma(V^*(1),\boldD^\perp,M^\perp),A[-3])\,,$$ 
which can be completed to an exact triangle
$$\boldR\Gamma(V,\boldD,M) \xrightarrow{{\bigcup_{-,V,\boldD,M}}_*} \boldR{\rm Hom}_A(\boldR\Gamma(V^*(1),\boldD^\perp,M^\perp),A[-3]) \lra {\rm cone}\left({\bigcup_{-,V,\boldD,M}}_*\right) \lra \boldR\Gamma(V,\boldD,M)[1]$$ 

The following is the first version of the algebraic functional equation, which is given in terms of the determinants of Selmer complexes. We assume throughout that both objects $\boldR\Gamma(V,\boldD,M)$ and $\boldR\Gamma(V^*(1),\boldD^\perp,M^\perp)$ are perfect. In later parts of this paper, we shall provide very general sufficient conditions to ensure that.  

\begin{proposition}[Functional equation, preliminary form]
\label{prop:det_selmerCom_dualSelCom_rel}
We have the following equality of determinants of perfect complexes:
$$\det\,\boldR{\rm Hom}_A(\boldR\Gamma(V^*(1),\boldD^\perp,M^\perp),A[-3])=\det\,\boldR\Gamma(V,\boldD,M)\otimes_A \det\,{\rm cone}\left({\bigcup_{-,V,\boldD,M}}_*\right).$$ 
Moreover, if for all $\fq \in S^{(p)} \setminus \Sigma$ the $A$-module $V^{I^w_\fq}/(t_\fq-1)V^{I^w_\fq}$ is projective, then we have 
$$\det\,\boldR\Gamma(V^*(1),\boldD^\perp,M^\perp)=\det\,\boldR{\rm Hom}_A(\boldR\Gamma(V,\boldD,M), A)[-3]\,.$$  
Here, ${I^w_\fq}\subset I_\fq$ denotes the wild ramification subgroup and $t_\fq$ is a fixed topological generator of $I_\fq/{I^w_\fq}$.
\end{proposition}
\begin{proof}
\label{proof:prop:det_selmerCom_dualSelCom_rel}
The first asserted equality follows from \cite[Lemma 4.40(2)]{On_Artin_formalism_for_p-adic_Garrett--Rankin_Lfunctions}. We remark that we have used the fact that if $f,g\,:\, X \to Y$ is a homotopic pair of morphisms of complexes, then ${\rm cone}(f) \simeq {\rm cone}(g).$ The second assertion follows immediately from \cite[Theorem 3.1.7(iii)]{p-adicHodgetheorybenoisheights}.
\end{proof}

\subsection{Iwasawa theoretic Selmer complexes for families over an affinoid}
\label{sec:Selmer_Complexes_over_A_tnesored_Lambda_Infinity_Iw}
Let $A$ be an $\QQ_p$-affinoid algebra. We let ${\rm Max}(A)$ denote the associated rigid analytic space. In this section, we will review the basic properties of Selmer Complexes over the ring $\cH_A(\Gamma_F^0)$.
We retain the notation and conventions of \S\ref{sec:Selmer_Complexes_over_A}. In particular, we let $V$ denote a $p$-adic representation of $G_{F,S}$ with coefficients in a $\QQ_p$-affinoid algebra $A$.
For each $\fq \in S_p$, we fix a $(\varphi,\Gamma_\fq)$-submodule 
$\boldD_\fq\subset D^{\dagger}_{\rm rig}(V_\fq)$. 

\subsubsection{} We will be interested in the following Iwasawa theoretic objects: $\overline{V}:=V \otimes_A \cH_A^\iota(\Gamma_F^0)$, 
$\overline{\boldD}_\fq:=\boldD_\fq \otimes_A \cH_A^\iota(\Gamma_{F_\fq}^0)$,  $\overline{\boldD}:=\oplus_{\fq\in S_p} \,\overline{\boldD}_\fq$. 
We also consider the corresponding complexes $K^\bullet(\overline{V})$ and $U^\bullet(\overline{V},\overline{\boldD})$, and morphisms $f=(\overline{f}_\fq)$ and $g=(\overline{g}_\fq)$ given as in \S\ref{sec:Selmer_Complexes_over_A}.

We define the Selmer Complex over $\cH_A(\Gamma_F^0)$ associated to the data $(\overline{V},\overline{\boldD})$ in a manner identical to above:
\begin{equation}
\label{eq:SelmerComplex_for_V_Iw}
{\mathbf S}(\overline{V},\overline{\boldD})={\rm cone}\left(C^\bullet(G_{F,S},\overline{V}) \oplus U^\bullet(\overline{V},\overline{\boldD}) \xrightarrow{\,\overline{f}-\overline{g}\,} K^\bullet(\overline{V})\right)[-1]\,.
\end{equation}
We set $\boldR\Gamma(\overline{V},\overline{\boldD}):=[{\mathbf S}(\overline{V},\overline{\boldD})] \in \cD(\cH_A(\Gamma_{F}^0))$ and write $\boldR^n\Gamma(\overline{V},\overline{\boldD})$ for its cohomology. From the definition of $\boldR\Gamma(\overline{V},\overline{\boldD})$, it is easy to see that $\boldR\Gamma(\overline{V},\overline{\boldD})$ is cohomologically bounded. 
\subsubsection{} We define the dual local condition on the dual Galois representation $\overline{V}^*(1):=V^*(1) \otimes_A \cH_A^\sharp(\Gamma_F^0)$ at each $\fq \in S_p$ via the $(\varphi,\Gamma)$-submodule $\overline{\boldD}^\perp_\fq:=\boldD^\perp_\fq \otimes_A \cH_A^\sharp(\Gamma_{F_\fq}^0)$, and put $\overline{\boldD}^\perp:=(\boldD^\perp_\fq)_{\fq\in S_p}$.


\subsection{Duality for punctual Iwasawa theoretic Selmer complexes} We record the key properties of our (punctual) Iwasawa theoretic Selmer complexes when $A=E$. These are proved in \cite[Proposition $6$]{selComp_p-adicHodgetheory} and \cite[Theorem 4.1 and Corollary 4.2]{pottharst}.


\begin{theorem}
\label{thm:Prop_Iw_SelComp}
Assume that our $p$-adic representation $V$ with coefficients in $E$ is potentially semistable at all primes above $p$. Then:
    \item[i)] The complex $\boldR\Gamma(\overline{V},\overline{\boldD})$ has cohomology concentrated in degrees $[1,3]$, and the respective cohomology groups are coadmissible $\cH_E(\Gamma^0_F)$-modules. Moreover, $$\boldR^3\Gamma(\overline{V},\overline{\boldD})\simeq (T^*(1)^{H_{F,S}})^* \otimes_{\Lambda_{\cO}(\Gamma_F^0)} \cH_E(\Gamma_F^0)$$ where $H_{F,S}={\rm Gal}(F_S/F_{\infty})$ and $T^*(1)\subset  V^*(1)$ is a $G_{F,S}$-stable lattice.
    \item[ii)]  If the local conditions are (strict) ordinary in the sense of \cite[p. 15]{pottharst} (see also \cite{jayanalyticfamilies}, \S3A and \S3B),
    $${\rm rank}_{\cH_E(\Gamma^0_F)}\,\boldR^1\Gamma(\overline{V},\overline{\boldD})={\rm rank}_{\cH_E(\Gamma^0_F)}\,\boldR^2\Gamma(\overline{V},\overline{\boldD})\,.$$
    \item[iii)] We have $\boldR\Gamma(\overline{V},\overline{\boldD}) \otimes^{\bbLL}_{\cH_A(\Gamma^0_F)} E \simeq \boldR\Gamma(V,\boldD)\,.$
    In particular, there is a canonical exact sequence
    $$0 \lra  \boldR^i\Gamma(\overline{V},\overline{\boldD})_{\Gamma_F} \lra \boldR^i(V,\boldD) \lra\boldR^{i+1}\Gamma(\overline{V},\overline{\boldD})^{\Gamma_F} \lra 0\,.$$
    \item[iv)] There is canonical duality
    $$\cD{(\boldR\Gamma(\overline{V},\overline{\boldD}))}^\iota \simeq \boldR\Gamma(\overline{V}^*(1),\overline{\boldD}^\perp)[3]$$
    where $\cD$ here denotes Grothendieck duality as in \cite[p.5]{pottharst} (see also \cite{factorization_alge_p_adic}, Proposition 5.14). In particular, we have a canonical exact sequence (for all $i$)
    $$0 \lra  \cD^1 \boldR^{4-i}\Gamma(\overline{V},\overline{\boldD})  \lra \boldR^{i}\Gamma(\overline{V}^*(1),\overline{\boldD}^\perp)  \lra\cD^0 \boldR^{3-i}\Gamma(\overline{V},\overline{\boldD})  \lra 0\,.$$

        
\end{theorem}
\begin{proof}
\label{proof:thm:Prop_Iw_SelComp}
  For proof of Parts $(i)$, $(ii)$, and $(iii)$, see \cite[Proposition 6]{selComp_p-adicHodgetheory}, \cite[Theorem 4.1 and Corollary 4.2]{pottharst}, and \cite[Proposition 3.7]{jayanalyticfamilies}.

We sketch proof of (iv), and refer the reader to  \cite{pottharst} for further details. By \cite[Theorem 4.1]{pottharst}, it suffices to prove that the relevant duality statement holds with $\boldR\Gamma(G_{F,S}, X)$ (where $X=\overline{V}$, $\overline{V}^*(1)$) and $U^\bullet_\fq(\overline{X},\overline{Y})$ for each $\fq \in S$ (where $Y=\overline{\boldD}, \overline{\boldD}^\perp$). The claim for $\boldR\Gamma(G_{F,S}, X)$ follows by applying $\otimes^{\bbLL}_{\Lambda_{\cO}(\Gamma^0_F)} \cH_E(\Gamma^0_F)$ to \cite[Theorem 8.5.6]{nekovar06}. To verify that the required duality holds for $U^\bullet_\fq(\overline{X},\overline{Y})$ for each $\fq \in S$, we argue for $\fq\in S_p$ and $\fq\in S^{(p)}$ separately. If $\fq \in S_p$, then the required duality statement is provided by \cite[Theorem 2.8.2]{p-adicHodgetheorybenoisheights}.  If $\fq \in S^{(p)}$, then the result follows by applying $\otimes^{\bbLL}_{\Lambda_{\cO}(\Gamma^0_F)} \cH_E(\Gamma^0_F)$ to \cite[Theorem 8.9.7.3]{nekovar06}.
\end{proof}

\begin{theorem}
\label{thm:det_selCOm_over_H_E_=_charIdeal_Cohomo_deg2}
Suppose we have $\boldR^i\Gamma(\overline{V}, \overline{\boldD})=0$ for $i \neq 2$ and $\boldR^2\Gamma(\overline{V}, \overline{\boldD})$ is a torsion $\cH_E(\Gamma_F^0)$-module. Then, $\det\,\boldR\Gamma(\overline{V}, \overline{\boldD})\,=\,{\rm char}_{\cH_E(\Gamma_F^0)}\,\boldR^2\Gamma(\overline{V}, \overline{\boldD})\,.$
\end{theorem}

\begin{proof}
\label{proof:thm:det_selCOm_over_H_E_=_charIdeal_Cohomo_deg2}
By Proposition~\ref{prop:det_torsion_M_Over_H_E_char_Ideal_M}, we have that 
$$\det\,\boldR^2\Gamma(\overline{V}, \overline{\boldD}))\,=\,{\rm char}_{\cH_E(\Gamma_F^0)}\,\boldR^2\Gamma(\overline{V}, \overline{\boldD})\,.$$ 
On the other hand, it is clear that we have
$\det\,\boldR\Gamma(\overline{V}, \overline{\boldD})\,=\,\det\,\boldR^2\Gamma(\overline{V}, \overline{\boldD})$ in our setting (cf.  \cite{Determinant_Functors}, Proposition 3.2.3), and the proof of our theorem follows.
\end{proof}

\begin{corollary}[Functional equation for $1$-variable algebraic $p$-adic $L$-functions]
    \label{cor_thm:det_selCOm_over_H_E_=_charIdeal_Cohomo_deg2}
    In the situation of Theorem~\ref{thm:det_selCOm_over_H_E_=_charIdeal_Cohomo_deg2}, we have
    $${\rm char}_{\cH_E(\Gamma_F^0)}(\boldR^2\Gamma(\overline{V}, \overline{\boldD}))={\rm char}_{\cH_E(\Gamma_F^0)}(\boldR^2\Gamma(\overline{V}^*(\chi), \overline{\boldD}^*(\chi)))^\iota\,.$$
\end{corollary}

\begin{proof}
    This is an immediate consequence of Theorem~\ref{thm:Prop_Iw_SelComp}(iv) and Theorem~\ref{thm:det_selCOm_over_H_E_=_charIdeal_Cohomo_deg2}.
\end{proof}

\subsection{Duality for Iwasawa theoretic Selmer complexes over affinoids}
\subsubsection{Inverse Limits}
\begin{proposition}
    \label{Prop_InverseLimit_phi-gamma_Over_H_A}
Let $M$ be a $(\varphi, \Gamma_\fq)$-module over $\cR_A$, where $A$ is an affinoid algebra. Let us put $\overline{M}=M \otimes_A \cH_A^{?}(\Gamma_\fq^0)$. Then,
$$\varprojlim_i\, C^\bullet_{\varphi, \gamma_\fq}(\overline{M}_i) \simeq C^\bullet_{\varphi,\gamma_\fq}(\overline{M})$$ 
where $\overline{M}_i=M \otimes_A \cH_{A,i}^{?}(\Gamma_{F_\fq}^0))$ and $? \in \{\iota, \sharp\}.$ 
\end{proposition}

\begin{proof}
This is clear from the definition of the Fontaine--Herr complex.
\end{proof}

\begin{proposition} 
\label{Prop_H^n_phi-gamma_over_A_tensored_H_i_M_L} 
Suppose we are in the situation of Proposition~\ref{Prop_InverseLimit_phi-gamma_Over_H_A}. Then the isomorphism of Proposition~\ref{Prop_InverseLimit_phi-gamma_Over_H_A} induces an isomorphism $H^n(C^\bullet_{\varphi, \gamma_\fq}(\overline{M})) \simeq \varprojlim_i\,H^n(C^\bullet_{\varphi, \gamma_\fq}(\overline{M}_i)).$ 
\end{proposition}

\begin{proof}
Since this is well-known to the experts, we only sketch a proof. Let us put 
    $H_i^n:=H^n(C^\bullet_{\varphi, \gamma_\fq}(\overline{M}_i))$ to ease our notation. As we have $\varprojlim_i\, C^\bullet_{\varphi, \gamma_\fq}(\overline{M}_i) \simeq C^\bullet_{\varphi,\gamma_\fq}(\overline{M}),$ it follows from \cite[Theorem 3.5.8]{weibel} that 
    $$\underset{i}{{\rm lim}^1}H_i^{n-1} \lra H^n(C^\bullet_{\varphi, \gamma_\fq}(\overline{M})) \lra \varprojlim_i\,H^n(C^\bullet_{\varphi, \gamma_\fq}(\overline{M}_i)) \lra 0\,.$$ 
By \cite[Lemma 2.1.4(2)]{KPX2014} (see also \cite{schneiderteitelbaum2003}, Theorem B), we see that $\underset{i}{{\rm lim}^1}H_i^{n-1}=0$ since $\{H_i^{n}\}_i$ is a Mittag--Leffler system by construction. 
\end{proof}

\begin{proposition}
    \label{Prop_InverseLimit_cochain_complex_F_q_and_F_S_Over_H_A}
For $\overline{V}$ as above, let us put $\overline{V}_i=V \otimes \cH_{A,i}^\iota(\Gamma_{F}^0)$, so that $\varprojlim_i\,\overline{V}_i=\overline{V}$. Then,
$$\varprojlim_i\, H^\bullet(G, \overline{V}_i)\simeq H^\bullet(G, \overline{V})$$ 
where $G=G_{F,S}$ or $G=G_{F_\fq}$. Moreover, we have $[C^\bullet(G,\overline{V})] \in \mathcal{D}^b_{\rm ft}(\cH_A(\Gamma_{F}^0))$, where the subscript ${\rm ft}$ means finite-type in the sense of \cite[\S1C]{jayanalyticfamilies}\footnote{Explicitly,   $[C^\bullet(G,\overline{V})]$ can be represented by a bounded complex of coadmissible $\cH_A(\Gamma_{F}^0)$-modules.}. 
\end{proposition}
\begin{proof}
The first claim is a special case of Proposition~\ref{Prop_H^n_phi-gamma_over_A_tensored_H_i_M_L} when $G=G_{F_\fq}$, and the proof in the scenario when $G=G_{F,S}$ is identical. We explain the second assertion. As we have remarked above, the $\cH_A(\Gamma_F^0)$-module $H^n(G, \overline{V})$ is coadmissable, and 
$H^n(G, \overline{V}) \otimes_{\cH_A(\Gamma_F^0)} \cH_{A,i}(\Gamma_F^0) \simeq H^n(G, \overline{V}_i)\,.$ 
Finally, as we have remarked in \S\ref{sec:Coefficient_rings_affinoid}, ${\rm Max}(A) \times X_{\rm cyc}$ is a quasi-Stein space, and the formalism of \cite[\S1C]{jayanalyticfamilies} applies, and shows that $C^\bullet(G,\overline{V})$ is indeed quasi-isomorphic to a complex of $\cH_A(\Gamma_{F}^0)$-modules of finite type.
\end{proof}

\begin{theorem}
\label{thm:inverselimit_MIttag_SelCOm_Over_H_A}
Suppose we are in the situation of Proposition~\ref{Prop_InverseLimit_cochain_complex_F_q_and_F_S_Over_H_A}.
\item[i)] For any positive integer $i$ and $\fq \in S^{(p)}$, the image of the natural projection maps 
$$\varprojlim_i\, C^\bullet_{\rm ur}(G_{F_\q},\overline{V}_i)\lra C^\bullet_{\rm ur}(G_{F_\q},\overline{V}_i)$$ 
is dense.
    \item[ii)] For $\fq$ as above, we have 
    $C^\bullet_{\rm ur}(G_{F_\q},\overline{V}) \simeq \varprojlim_i\, C^\bullet_{\rm ur}(G_{F_\q},\overline{V}_i)\,.$
    \item[iii)] We have 
   $ H^n(C^\bullet_{\rm ur}(\overline{V})) \simeq \varprojlim_i\, H^n(C^\bullet_{\rm ur}(\overline{V}_i))$ for all natural numbers $n$. Similar statements hold if we replace $\overline{V}$ with $\overline{V}^*(1)$ and  $\overline{V}_i$ with $\overline{V}_i^*(1)=\overline{V}^*(1) \otimes A_i$.
    \item[iv)] Let us put $\overline{\boldD}_i:=\boldD \otimes \cH_{A,i}^\iota(\Gamma_{F}^0)$. Then, 
    $\boldR\Gamma(\overline{V},\overline{D}) \simeq \varprojlim_i\,\boldR\Gamma(\overline{V}_i,\overline{D}_i)\,.$
    \item[v)] $\boldR^n\Gamma(\overline{V},\overline{D}) \simeq \varprojlim_i\,\boldR^n\Gamma(\overline{V}_i,\overline{D}_i)$.

Similar statements hold if we replace $(\overline{V},\overline{\boldD})$ with $(\overline{V}^*(1),\overline{\boldD}^\perp)$.

\end{theorem}
\begin{proof}

Parts (i) and (ii) follow noting that 
$$\overline{V}_i^{I_\fq}\simeq \overline{V}^{I_\fq}\otimes_A A_i=\left(V^{I_\fq}\otimes \cH_{A}^\iota(\Gamma_F^0)\right)\otimes_E \cH_{E,i}(\Gamma_{F}^0)\,,$$
(where the final factor carries no Galois action) since $\cH_{E}(\Gamma_{F}^0)$ is dense in $\cH_{E,i}(\Gamma_{F}^0)$.

Part (iii) follows from the density (which implies the Mittag--Leffler condition); see \cite[p. 1586]{jayanalyticfamilies} and \cite[Theorem A]{schneiderteitelbaum2003}.

Pert (iv) and (v) is proved (by diagram chase) using Proposition~\ref{Prop_InverseLimit_phi-gamma_Over_H_A}, Lemma~\ref{Prop_InverseLimit_cochain_complex_F_q_and_F_S_Over_H_A}, and Part (iii).
\end{proof}

\subsubsection{Inverse limits and duality}

\begin{lemma}
\label{lemma:inverseLimit_Hom_over_A_tensored_H_i_equal_Hom_over_H_A}
Suppose that for each positive integer $i$, we have finitely generated $\cH_{A,i}(\Gamma_F^0)$-modules $M_i$, such that 
$$M_{i+1} \otimes_{\cH_{A,i+1}(\Gamma_F^0)} \cH_{A,i}(\Gamma_F^0) \xrightarrow{\,\,\sim\,\,} M_i\,,\qquad \forall i\in \ZZ^+\,.$$ 
Let us define the $\cH_A(\Gamma_F^0)$-module $M:=\varprojlim_i\,M_i$. Then, we have natural isomorphisms 
$$\varprojlim_i\,{\rm Hom}_{\cH_{A,i}(\Gamma_{F}^0)}(M_i, \cH_{A,i}(\Gamma_{F}^0))\simeq \varprojlim_i\,{\rm Hom}_{\cH_A(\Gamma_{F}^0)}(M, \cH_{A,i}(\Gamma_{F}^0)) \simeq {\rm Hom}_{\cH_A(\Gamma_{F}^0)}(M, \cH_A(\Gamma_{F}^0)).$$
\end{lemma}
\begin{proof}
\label{proof:lemma:inverseLimit_Hom_over_A_tensored_H_i_equal_Hom_over_H_A}
For simplicity, only in this proof, we let $R_i=\cH_{A,i}(\Gamma_F^0)$, $R=\cH_A(\Gamma_F^0)$. The right-most isomorphism in the lemma follows from \cite[Proposition 5.21]{rotmanHomoAl} viewing each $R_i$ as a module over $R.$ It remains to prove the isomorphism  
\begin{equation}
\label{eqn_2023_08_30_1038}
    \varprojlim_i\,{\rm Hom}_{R_i}(M_i, R_i)\simeq {\rm Hom}_{R}(M, R)
\end{equation}
in our lemma.

Note that by \cite[Lemma 2.1.4(3)]{KPX2014} (see also \cite{schneiderteitelbaum2003}, Corollary 3.1), we have a natural identification $M \otimes_{R} R_i \xrightarrow{\sim} M_i$ thanks to our running conditions on the inverse system. 



Since $M=\varprojlim_i\,M_i$ and $R=\varprojlim_i\,R_i$, giving an element $\{\phi_i\}\in \varprojlim_i{\rm Hom}_{R_i}(M_i,R_i)$ is the same as giving an element $\phi\in {\rm Hom}_{R}(M,R)$. 
In more precise wording, given $\phi$, the maps $\phi_i$ is given by (recalling the natural identification $M_i\simeq M\otimes_R R_i$ and the injection $R\hookrightarrow R_i$) 
$$\phi_i\,:\quad M_i\ni m \otimes r_i\stackrel{\phi \otimes {\rm id}} {\longmapsto}\phi(m)\otimes r_i \longmapsto \phi(m)r_i \in R_i$$
This gives the map from the right-hand side of \eqref{eqn_2023_08_30_1038} to the left. 

In the opposite direction, given $\{\phi_i\}$, the map $\phi$ is given by (relying on the fact that $M=\varprojlim M_i$ and $R=\varprojlim R_i$)
$$\phi\,:\quad M=\varprojlim M_i\ni \{m_i\}\stackrel{\{\phi_i\}} {\longmapsto} \{\phi_i(m_i)\}\in \varprojlim R_i =R\,.$$
These maps are mutual inverses, and our lemma is proved.
\end{proof}

\subsubsection{Local duality over $\cH_A$}
\label{subsec_duality_over_H_A}
We prove various local (Iwasawa theoretic) duality results, in preparation for the same for Selmer complexes (that we will use later to prove the functional equation for algebraic $p$-adic $L$-functions). 

\begin{theorem}[Kedlaya--Pottharst--Xiao]
\label{thm:cup-product-phi-gamma-over_H_A}
The cup-product given as in \eqref{eq:cup_product_phi-gamma_module_over_H_A_2} induces an isomorphism 
\begin{align}
\begin{aligned}
\boldR\Gamma(F_\fq, \overline{\boldD}_\fq^*(\chi_\fq)) &\simeq \boldR{\rm Hom}_{\cH_A(\Gamma_{F_\fq}^0)}( \boldR\Gamma(F_\fq, \overline{\boldD}_\fq), \cH_A(\Gamma_{F_\fq}^0)[-2])\,.
\end{aligned}
\end{align}
\end{theorem}
\begin{proof}
\label{proof:thm:cup-product-phi-gamma-over_H_A}
This follows on combining the quasi-isomorphims (4.4.8.1) and  (4.4.8.2) of \cite{KPX2014} with Proposition 2.3.6 of op. cit. and Proposition~\ref{Prop_H^n_phi-gamma_over_A_tensored_H_i_M_L} above.
\end{proof}

As a special case of Theorem~\ref{thm:cup-product-phi-gamma-over_H_A} (which we apply with $\boldD_\fq:=D_{\rm rig}^\dagger(V)$), we have:

\begin{corollary}
    \label{cor_Ind_V_otimes_Ind_dualV_pairing_H_A}
Cup-product pairing induces a natural isomorphism 
$$\boldR\Gamma(F_\fq,\overline{V}^*(1)) \simeq \boldR{\rm Hom}_{\cH_A(\Gamma_F^0)} (\boldR\Gamma(F_\fq,\overline{V}), \cH_A(\Gamma_F^0))[-2]\,.$$

\end{corollary}

\subsubsection{Cohomology with compact support and duality} Let us consider the complex $ C^\bullet_{\rm c}(G_{F,S},\,\cdot\,)$ compactly supported continuous cochains given by 
\[ 
 C^n_{\rm c}(G_{F,S}, \,\cdot\, ):=\ker\left( C^n(G_{F,S}, \,\cdot\,) \lra  \oplus_{\fq \in S} C^n(G_{F_\fq}, \,\cdot\,)  \right)\,,
\]
and denote its cohomology by $H^n_{\rm c}(G_{F,S}, \,\cdot\,)$.

\begin{lemma}
    \label{lemma_InverseLimit_And_ML_C_c_cont_V_Cochain_V^*_H_A}
In the situation of Proposition~\ref{Prop_InverseLimit_cochain_complex_F_q_and_F_S_Over_H_A},
we have $H^n_{\rm c}(G_{F,S},\overline{V})\, \simeq\,\varprojlim_i\,H^n_{\rm c}(G_{F,S},\overline{V}_i)\,.$ 
\end{lemma}
\begin{proof}
       This follows from Proposition~\ref{Prop_InverseLimit_cochain_complex_F_q_and_F_S_Over_H_A} by diagram chase.  
\end{proof}

We now turn our attention to global duality. We begin with recording the following cup-product pairing (very slightly extending the scope of \cite{nekovar06}, \S5.3.3.2):
\begin{align}
\begin{aligned}
\label{eq:cup-product_C_c_cont_V_Cochain_V^*_H_A}   
{_c}\cup\,:\, \boldR\Gamma_{c,{\rm cont}}(G_{F,S}, \overline{V}^*(1)) \otimes \boldR\Gamma(G_{F,S},\overline{V}) &\lra \boldR\Gamma_{c,{\rm cont}}(G_{F,S}, \overline{V}^*(1) \otimes \overline{V})  \\
&\lra\,_{\tau \geq 3} \boldR\Gamma_{c,{\rm cont}}(G_{F,S}, \cH_A(\Gamma_F^0)(1)) \simeq \cH_A(\Gamma_F^0)[-3]\,.
\end{aligned}
\end{align}

\begin{theorem}
\label{thm:cup-product_C_c_cont_V_Cochain_V^*_H_A}
The cup-product pairing \eqref{eq:cup-product_C_c_cont_V_Cochain_V^*_H_A} induces an isomorphism 
$$\boldR\Gamma_{c,{\rm cont}}(G_{F,S}, \overline{V}^*(1)) \simeq \boldR{\rm Hom}_{\cH_A(\Gamma_F^0)}(\boldR\Gamma(G_{F,S},\overline{V}), \cH_A(\Gamma_F^0))[-3].$$
\end{theorem}
\begin{proof}
\label{proof:thm:cup-product_C_c_cont_V_Cochain_V^*_H_A}
We remark that $H^n_{c,{\rm cont}}(G_{F,S}, \overline{V}_i^*(1))$ is a finitely generated $\cH_{A,i}(\Gamma_F^0)$-module for each $i$, because all complexes involved in the definition of $\boldR\Gamma_{c,{\rm cont}}(G_{F,S}, \overline{V}_i^*(1))$ are finitely generated by \cite[Theorem 1.1(4)]{jayanalyticfamilies}. 
    Also, we note that 
    $$H^n_{c,{\rm cont}}(G_{F,S},\overline{V}_{i+1}) \otimes_{\cH_{A,i+1}(\Gamma_F^0)} \cH_{A,i}(\Gamma_F^0) \simeq H^n_{c,{\rm cont}}(G_{F,S},\overline{V}_{i}).$$   
By \cite[Theorem 1.15(2)]{jayanalyticfamilies}, 
$$H^{3-n}_{c,{\rm cont}}(G_{F,S}, \overline{V}_i^*(1)) \simeq \boldR{\rm Hom}_{\cH_{A,i}(\Gamma_F^0)}(H^n(G_{F,S},\overline{V}_i), \cH_{A,i}(\Gamma_F^0))\,.$$ The claimed isomorphism follows on passing to limit, using Proposition~\ref{Prop_InverseLimit_cochain_complex_F_q_and_F_S_Over_H_A}, Lemma~\ref{lemma:inverseLimit_Hom_over_A_tensored_H_i_equal_Hom_over_H_A} and Lemma~\ref{lemma_InverseLimit_And_ML_C_c_cont_V_Cochain_V^*_H_A}. 
\end{proof}

\subsubsection{Duality of Selmer complexes}
\label{sec:Duality_SelmerCom}
In this subsection, we will record Theorem~\ref{thm:duality_selComplex_Iw}, which is a very slight generalization of \cite[Theorem 3.1.5]{p-adicHodgetheorybenoisheights} to Galois representations with coefficients in $\cH_A(\Gamma_F^0)$.

\begin{theorem}
\label{thm:duality_selComplex_Iw}
  There exists a canonical cup-product pairing
    $$\bigcup_{\overline{V},\overline{\boldD}}\,:\, \boldR\Gamma(\overline{V},\overline{\boldD}) \otimes^\bbLL_{\cH_A(\Gamma_F^0)} \boldR\Gamma(\overline{V}^*(1),\overline{\boldD}^{\perp}) \longrightarrow \cH_A(\Gamma_F^0)[-3]\,.$$
\end{theorem}

\begin{proof}
\label{proof:thm:duality_selComplex_Iw_Good_way}
Note that $\cH_A(\ZZ_p)$ is flat over $A$ (cf. \cite{KPX2014}, Corollary 2.1.5). Based on this fact, combined with our preparation above, this is proved applying verbatim the arguments of Benois in his proof of \cite[Theorem 3.1.5]{p-adicHodgetheorybenoisheights}.
\end{proof}

Let us denote by 
\begin{equation}
\label{eqn_2024_11_01_1736}
    {\bigcup_{\overline{V},\overline{\boldD}}}_*\,:\,\quad \boldR\Gamma(\overline{V}^*(1),\overline{\boldD}^\perp) \longrightarrow \boldR{\rm Hom}_{\cH_A(\Gamma_F^0)}(\boldR\Gamma(\overline{V},\overline{\boldD}),\cH_A(\Gamma_F^0)[-3])
\end{equation}
the morphism induced by the cup product $\bigcup_{\overline{V},\overline{\boldD}}$. Before we close this subsection, we record the following easy lemma that will be needed later.

\begin{lemma}
\label{lemma:Hom_Bet_M_and_A_tensored_H_A}
Let $M$ be any finitely generated $A$-module. Then 
$${\rm Hom}_A(M,A)\, {\widehat\otimes} \,\cH_A(\ZZ_p) \simeq {\rm Hom}_{\cH_A(\ZZ_p)}(M {\widehat\otimes} \cH_A(\ZZ_p) , \cH_A(\ZZ_p))$$ as $\cH_A(\ZZ_p)$-modules. 

\end{lemma}

\begin{proof}
Since $\cH_A(\ZZ_p)$ is flat over $A$, we have 
$${\rm Hom}_A(M,A) {\widehat\otimes} \cH_A(\ZZ_p) \simeq {\rm Hom}_A(M,\cH_A(\ZZ_p))={\rm Hom}_{\cH_A(\ZZ_p)}(M{\widehat\otimes} \cH_A(\ZZ_p),\cH_A(\ZZ_p))$$ 
where the last isomorphism is explained in \cite[Theorem 6.31]{TENSOR_PRODUCTS_KEITH_CONRAD}.  
\end{proof}

\subsubsection{Perfectness of the cup-product pairing}  
Our goal in this subsection is to give (cf. Theorem~\ref{thm:duality_selComplex_Iw_iso}(iii) below) a sufficient condition for the perfectness of the cup-product pairing of Theorem~\ref{thm:duality_selComplex_Iw}. 

To that end, suppose that $\fq \in S^{(p)}$. Recall that we denote by $I^w_\fq$ the wild ramification subgroup of $I_\fq$, and we have fixed a topological generator $t_\fq$ of $I_\fq/I^w_\fq$ such that, for any uniformizer $\varpi_\fq$ of $F_\fq$, we have $t_\fq(\varpi^{1/p^n}_\fq)=\zeta_{p^n}\varpi^{1/p^n}_\fq$ for all  $n$. 
As explained in \cite[p. 65]{p-adicHodgetheorybenoisheights}, since $I^w_\fq$ acts on $\overline{V}$ through a finite quotient $H$ and $I_\fq$ acts trivially on $\cH_A(\Gamma_{F_\fq}^0)$, we have a decomposition 
$$\overline{V} \simeq \left(V^{I^w_\fq} \otimes_A \cH_A(\Gamma_{F_\fq}^0)\right) \oplus \left(I_H\cdot V \otimes_A \cH_A(\Gamma_{F_\fq}^0)\right) \,,$$ 
where $I_H={\rm ker}(\ZZ[H] \xrightarrow{} \ZZ)$ is the augmentation ideal. Since $\cH_A(\Gamma_{F_\fq}^0)$ is free and flat over $A$, the submodule $V^{I^w_\fq} \otimes_A \cH_A(\Gamma_{F_\fq}^0)$ is a direct factor of the projective $\cH_A(\Gamma_{F_\fq}^0)$-module $\overline{V}$, it is therefore is projective.
Using Equation (56) of \cite{p-adicHodgetheorybenoisheights}, we also see that 
${V^*(1)}^{I^w_\fq} \otimes_A \cH_A(\Gamma_{F_\fq}^0) ={\rm Hom}_A(V^{I^w_\fq}, A)(1) \otimes_A \cH_A(\Gamma_{F_\fq}^0)\,.$ \newline 

For any $\fq \in S$, let us define
$$\widetilde{U}^\bullet_\fq(\overline{V},\overline{\boldD}):={\rm cone}\left( U^\bullet_q(\overline{V},\overline{\boldD})  \xrightarrow{\overline{g}_\fq} K^\bullet(\overline{V}_\fq)\right)[-1]$$
and put $\boldR\widetilde{\Gamma}(F_\fq, \overline{V},\overline{D}):=[\widetilde{U}^\bullet_\fq(\overline{V},\overline{\boldD})]$  in derived category of $\cH_A(\Gamma_{F_\fq}^0)$-modules. From the orthogonality of $\overline{g}_\fq$ and $\overline{g}^\perp_\fq$ under the cup-product (where $\bigcup_K$ is the natural cup-product pairing as in \cite[\S3.1.4]{p-adicHodgetheorybenoisheights})
\begin{align*}
    K^\bullet(\overline{V}_\fq) \otimes K^\bullet(\overline{V}^*_\fq(1)) \xrightarrow{\bigcup_K} K^\bullet(\overline{V}_\fq \otimes \overline{V}^*_\fq(1)) \longrightarrow K^\bullet(\cH_A(\Gamma_{F_\fq}^0)_\fq(1)) &\lra _{\tau \geq 2}K^\bullet(\cH_A(\Gamma_{F_\fq}^0)_\fq(1)) \\
    &\qquad \lra  \cH_A(\Gamma_{F_\fq}^0)[-2]\,,
\end{align*}
we have an induced pairing
$\widetilde{U}^\bullet_\fq(\overline{V},\overline{\boldD}) \otimes U^\bullet_q(\overline{V}^*(1),\overline{\boldD}^\perp) \xrightarrow{\bigcup_Z} \cH_A(\Gamma_{F_\fq}^0)[-2]\,,$ 
which induces a morphism    
\begin{equation}
\label{eq:cup_Z_induced_morphishm}
\boldR\Gamma(F_\fq, \overline{\boldD}^\perp_\fq) \lra \boldR{\rm Hom}_{\cH_A(\Gamma_{F_\fq}^0)}(\boldR\widetilde{\Gamma}(F_\fq, \overline{V},\overline{D}), \cH_A(\Gamma_{F_\fq}^0))^\iota[-2].
\end{equation}
in the derived category of $\cH_A(\Gamma_{F_\fq}^0)$-modules.

\begin{proposition}
\label{prop:othogonal_complement}
The map \eqref{eq:cup_Z_induced_morphishm} is a quasi-isomorphism if and only if \eqref{eq:cup_Z_induced_morphishm} is.
\end{proposition}
\begin{proof}
\label{proof:prop:othogonal_complement}
Let us set $\overline{\boldD}^\perp_{\fq,i}=\boldD^\perp_\fq \otimes_A \cH_{A,i}$ and $\overline{\boldD}_i=\boldD \otimes_A \cH_{A,i}$. By \cite[p.67]{p-adicHodgetheorybenoisheights} (see also \cite{nekovar06}, Lemma-Definition 6.2.7), we have
\begin{align}
\begin{aligned}
\label{eq:prop:othogonal_complement}
\boldR^{2-j}\Gamma(F_\fq, \overline{\boldD}^\perp_{\fq,i})) &\simeq {\rm Hom}_{\cH_{A,i}(\Gamma_{F_\fq}^0)}\left( (\boldR^{j}\widetilde{\Gamma}(F_\fq, \overline{V}_i,\overline{\boldD}_i)), \cH_{A,i}(\Gamma_{F_\fq}^0)\right) \\
&\Longleftrightarrow \boldR^j\widetilde{\Gamma}(F_\fq, \overline{V}_i,\overline{\boldD}_i) \simeq {\rm Hom}_{\cH_{A,i}(\Gamma_{F_\fq}^0)}\left(\boldR^{2-j}\Gamma(F_\fq, \overline{\boldD}^\perp_{\fq,i}), \cH_{A,i}(\Gamma_{F_\fq}^0)\right)
\end{aligned}
\end{align}
Applying Lemma~\ref{lemma:inverseLimit_Hom_over_A_tensored_H_i_equal_Hom_over_H_A} together with \eqref{eq:prop:othogonal_complement}, we conclude that
\begin{align}
\begin{aligned}
\label{eq:prop:othogonal_complement_2}
\boldR^{2-j}\Gamma(F_\fq, \overline{\boldD}^\perp_\fq) &\simeq \boldR{\rm Hom}_{\cH_{A}(\Gamma_{F_\fq}^0)}\left( \boldR^j\widetilde{\Gamma}(F_\fq, \overline{V},\overline{\boldD}), \cH_A(\Gamma_{F_\fq}^0)\right) \\
&\Longleftrightarrow \boldR^j\widetilde{\Gamma}(F_\fq, \overline{V},\overline{\boldD}) \simeq \boldR{\rm Hom}_{\cH_A(\Gamma_{F_\fq}^0)}\left(\boldR^{2-j}\Gamma(F_\fq, \overline{\boldD}^\perp_\fq), \cH_A(\Gamma_{F_\fq}^0)\right)\,.
\end{aligned}
\end{align}
\end{proof}
 Theorem~\ref{thm:duality_selComplex_Iw_iso} below is key to our proof of Theorem~\ref{thm:char_ideals_for_selCom_and_dual_selCom_H_A}, where we establish the functional equation for algebraic $p$-adic $L$-functions under an explicit ``vanishing of Tamagawa number'' condition. We shall mainly follow the arguments in the proof of \cite[Theorem 3.1.7]{p-adicHodgetheorybenoisheights} in its proof.

\begin{theorem}[Benois]
\label{thm:duality_selComplex_Iw_iso}
    \item[i)] For all $\fq \in S_p$, the map \eqref{eq:cup_Z_induced_morphishm} is an isomorphism in $\cD(\cH_A(\Gamma_{F_\fq}^0)).$ 
    \item[ii)] Suppose that $\fq \in S^{(p)}.$ If the $A$-module 
    $V^{I^w_\fq} /(t_\fq-1)V^{I^w_\fq}\,\simeq H^1(I_\fq, V)(1)$ 
    is projective (see \cite[\S7.5.8]{nekovar06} for an explanation of the isomorphism), then the $\cH_A(\Gamma_{F_\fq}^0)$-modules 
    $$V^{I_\fq} \otimes_A \cH_A^\iota(\Gamma_{F_\fq}^0),\qquad  V^*(1)^{I_\fq} \otimes_A \cH_A^\sharp(\Gamma_{F_\fq}^0),\qquad  (V^*(1)^{I^w_\fq} /(t_\fq-1)V^*(1)^{I^w_\fq}) \otimes_A \cH_A^\sharp(\Gamma_{F_\fq}^0)$$ 
    are projective as well, and \eqref{eq:cup_Z_induced_morphishm} is an isomorphism in $\cD(\cH_A(\Gamma_{F_\fq}^0)).$
    \item[iii)] If the $A$-module $V^{I^w_\fq} /((t_\fq-1)V^{I^w_\fq})$ is projective for all $\fq \in S^{(p)}$, then the duality map \eqref{eqn_2024_11_01_1736}
    is an isomorphism. 

\end{theorem}
\begin{proof}
\label{proof:thm:duality_selComplex_Iw_iso}
We highlight the key ingredients in Benois' argument to prove \cite[Theorem 3.1.7]{p-adicHodgetheorybenoisheights}.
    \item[i)] We use Theorem~\ref{thm:cup-product-phi-gamma-over_H_A} in place of \cite[Theorem 2.4.3]{p-adicHodgetheorybenoisheights}, together with the discussion in \cite[p.66]{p-adicHodgetheorybenoisheights} and the fact that $\cH_A(\Gamma^0_{F_\q})$ is flat over $A$. 
   
\item[ii)] The first claim (on the projectivity) is immediate from \cite[Theorem 3.1.7(ii)]{p-adicHodgetheorybenoisheights}. The second claim that the \eqref{eq:cup_Z_induced_morphishm} is an isomorphism arguing as in the proof of \cite[Theorem 3.1.7(iii)]{p-adicHodgetheorybenoisheights} (the step where Benois proves Eqn. 53 in op. cit. is an isomorphism), and using Corollary~\ref{cor_Ind_V_otimes_Ind_dualV_pairing_H_A} together with Lemma~\ref{lemma:Hom_Bet_M_and_A_tensored_H_A}. 

\item[iii)] 
We have the following commutative diagram (generalizing Proposition 6.3.3 and Proposition 6.3.4 in \cite{nekovar06}):  
\[ \begin{tikzcd}
& \boldR\Gamma_{c,{\rm cont}}(G_{F,S}, \overline{V}^*(1)) \arrow{r}{} \arrow[swap]{d}{\text{Theorem~}\ref{thm:cup-product_C_c_cont_V_Cochain_V^*_H_A}} & \boldR\Gamma(\overline{V}^*(1),\overline{\boldD}^\perp) \arrow{r}{} \arrow{d}{\text{Theorem~} \ref{thm:duality_selComplex_Iw}} & \boldR\Gamma(F_\fq, \overline{\boldD}^\perp_\fq) \arrow{d}{\text{$(i)$ and $(ii)$}} \\%
& \boldR\Gamma(G_{F,S},\overline{V})^*[-3] \arrow{r}{} & \boldR\Gamma(\overline{V},\overline{\boldD})^*[-3] \arrow{r}{} &\boldR\widetilde{\Gamma}(F_\fq, \overline{V},\overline{D})^*[-3]
\end{tikzcd}
\]
The superscript ``$*$'' on the bottom row means application of $\boldR{\rm Hom}_{\cH_A(\ZZ_p)}(-, \cH_A(\ZZ_p))[-3]$ on the respective object. Since we have already proved that left-most and right-most vertical maps are isomorphisms, the middle vertical map in this diagram is an isomorphism as well, as required. 

\end{proof}

\subsection{Perfectness of Selmer complexes over $\cH_A(\Gamma_F^0)$}

\begin{proposition}
    \label{prop_Finiteness_SelCOm_Over_H_A_And_coho_in_deg3_anddeg0_is_0}
Suppose that the $A$-module $V^{I^w_\fq} /(t_\fq-1)V^{I^w_\fq}$ is projective for every $\fq\in S^{(p)}$.
     \item[i)] $\boldR\Gamma(\overline{V},\overline{\boldD}) \in \cD^{[0,3]}_{\rm perf}(\cH_A(\Gamma_F^0)).$
    
    \item[ii)] Suppose that we have  $(V^*_x)^{H_{F,S}}=0$ for all $x \in {\rm Max}(A)$, where $V^*_x=V^*\otimes_A A/ \frakm_x$ with $\frakm_x$ the maximal ideal of $A$ corresponding to $x$, and $H_{F,S}$ is as in Theorem~\ref{thm:Prop_Iw_SelComp}. Then $\boldR^3\Gamma(\overline{V},\overline{\boldD})=0\,.$

    \item[iii)] $\boldR^0\Gamma(\overline{V},\overline{\boldD})=0.$

\end{proposition}
\begin{proof}

    \item[i)] Thanks to our running assumptions, the complex  $C^\bullet_{\rm ur}(\overline{V}_\fq)$ is perfect for every $\fq \in S^{(p)}$. Also, for $\fq \in S_p$, the complex $C^\bullet_{\varphi,\gamma_\fq}(\overline{\boldD}_\fq) \in \cD^{[0,2]}_{\rm perf}(\cH_A(\Gamma_{F_\fq}^0))$ thanks to \cite[Theorem 4.4.6]{KPX2014} combined with Proposition 2.3.6 in op. cit. Recall that we have an exact triangle 
    $\boldR\Gamma(\overline{V},\overline{\boldD}) \lra  \boldR\Gamma(G_{F,S}, \overline{V}) \lra [\widetilde{U}^\bullet_\fq(\overline{V},\overline{\boldD})]$
    from the definition of $\boldR\Gamma(\overline{V}, \overline{\boldD})$. 
    The proof of this portion follows from the discussion of \cite[Theorem 1.11]{jayanalyticfamilies} and the definition of $\widetilde{U}^\bullet_\fq(\overline{V},\overline{\boldD})$.

    \item[ii)] Suppose $x \in {\rm Spm}(A)$ is an $E$-valued point and let $\frakm_x$ denote the corresponding maximal ideal.  Then by \cite[Propositions 5.11 and 5.12]{factorization_alge_p_adic}, we have 
    $\boldR^3\Gamma(\overline{V}, \overline{\boldD}) \otimes_{\cH_A(\Gamma_F^0)} \cH_A(\Gamma_F^0)/\fp_x \simeq \boldR^3\Gamma(\overline{V}_x, \overline{\boldD}_x)$
    where $V_x= V \otimes_A A/ \frakm_x$,\quad  $\boldD_x= \boldD \otimes_A A/\frakm_x$, and $\fp_x=\ker\left(\cH_A(\Gamma_F^0) \xrightarrow{{\rm sp}_{x,{\rm Iw}}} \cH_{E_x}(\Gamma_F^0)\right)$ induced by the specialization at $x \in {\rm Max}(A)$ where $E_x=A/\frakm_x.$ Thanks to our running assumptions, we have $\boldR^3\Gamma(\overline{V}_x, \overline{\boldD}_x)=0$ for all $x \in {\rm Spm}(A)$ by Theorem~\ref{thm:Prop_Iw_SelComp}. The claim that $\boldR^3\Gamma(\overline{V}, \overline{\boldD})=0$ follows by \cite[Lemma 7.7]{factorization_alge_p_adic}.
    \item[iii)]  Note that $C^\bullet_{\varphi,\gamma_\fq}(\overline{\boldD}_\fq)$ naturally injects into $C^\bullet_{\varphi,\gamma_\fq}(D^\dagger_{\rm rig}(\overline{V}_\fq))$. By \cite[Theorem 2.4.3]{p-adicHodgetheorybenoisheights} (see also \cite{jayanalyticfamilies}, Theorem 2.8), we also have $H^0(C^\bullet_{\varphi,\gamma_\fq}(D^\dagger_{\rm rig}(\overline{V}_\fq))) \simeq H^0(F_\fq, \overline{V})\,.$ 
    Hence the result follows from the definition of $\boldR\Gamma(\overline{V},\overline{\boldD})$.
 
\end{proof}

\begin{theorem}[Theorem 5.21 in \cite{factorization_alge_p_adic}]
\label{thm:deg_SelCOm_Over_H_Ais_1_2_facto}
Assume that the following conditions hold:
    \item[i)] $\boldR^3\Gamma(\overline{V},\overline{\boldD})=0.$
    \item[ii)] For each prime ideal $\fp$ of $\Lambda_A[1/p]:=\Lambda_{\cO}(\Gamma_F^\circ)[1/p] \widehat{\otimes}A,$ we have $(V \otimes \Lambda_A^\iota[1/p]/\fp)^{G_{F,S}}=\{0\}$.
    \item[iii)] For each $\fq \in S^{(p)}$, the $A$-module $H^1(I_\fq,V)$ is projective.

Then we have $\boldR\Gamma(\overline{V}, \overline{\boldD}) \in \cD^{[1,2]}_{\rm perf}(\cH_A(\Gamma_F^0)).$ 
\end{theorem}

\begin{remark}
    When the Galois representation $V$ admits an integral structure, condition (ii) in Theorem~\ref{thm:deg_SelCOm_Over_H_Ais_1_2_facto} can be considerably weakened. We discuss this in 
    \S\ref{sec:Triangulations_in_families_1_1} below (see especially Proposition~\ref{prop:Irr_V_0_with_item_Irr_res_F}, where the required input is the residual irreducibility condition \eqref{item_Irr_res_F} below) in the context of Rankin--Selberg convolutions. The reader will notice that our argument to prove Proposition~\ref{prop:Irr_V_0_with_item_Irr_res_F} is formal and can be easily generalized.
\end{remark}

\begin{proposition}[Corollary 5.32 in \cite{factorization_alge_p_adic}]
\label{prop:projective_dim_H^2_SelComp}
Suppose that we are in the situation of Theorem~\ref{thm:deg_SelCOm_Over_H_Ais_1_2_facto}. Assume that $\boldR^1\Gamma(\overline{V}, \overline{\boldD})=0$ and $\boldR^2\Gamma(\overline{V}, \overline{\boldD})$ is a torsion $\cH_A(\Gamma_F^0)$-module. Then $\boldR^2\Gamma(\overline{V}, \overline{\boldD})$ admits a projective resolution $P^\bullet=[P_1 \xrightarrow{} P_2]$ of dimension $1$ by finitely generated  projective $\cH_A(\Gamma_F^0)$-modules of same rank. 
\end{proposition}
\begin{proof}
\label{proof:prop:projective_dim_H^2_SelComp}
By Theorem~\ref{thm:deg_SelCOm_Over_H_Ais_1_2_facto}, we know that $\boldR\Gamma(\overline{V}, \overline{\boldD}) \in \cD^{[1,2]}_{\rm perf}(\cH_A(\Gamma_F^0)).$  Then the result follows arguing as in \cite[Corollary 5.34]{factorization_alge_p_adic}.
\end{proof}

\begin{remark}
    Theorem~\ref{thm:Prop_Iw_SelComp}(ii) tells us that if we assume that the local conditions determined by ${\boldD}_x$ are strict ordinary in the sense of \cite[\S3A-\S3B]{jayanalyticfamilies} for all $x \in {\rm Max}(A)$, then the vanishing $\boldR^1\Gamma(\overline{V}, \overline{\boldD})=0$ is equivalent to the requirement that $\boldR^2\Gamma(\overline{V}, \overline{\boldD})$ be torsion. Note that, for the main arithmetic applications (which we discuss in the next chapter), we will work with strict ordinary local conditions.
\end{remark}

\subsection{Functional equation for the Iwasawa theoretic Selmer complex over $A$}
\label{sec:Algebraic_Functional_Equation_of_Iwasawa_Selmer_Complex}
We are now in a position to state and prove the main result (Theorem~\ref{thm:char_ideals_for_selCom_and_dual_selCom_H_A}) of this chapter.

\begin{theorem}
\label{thm:det_selCOm_over_H_A_=_charIdeal_Cohomo_deg2}
Suppose that $A$ is a Tate algebra, and assume that we are in the situation of Proposition~\ref{prop:projective_dim_H^2_SelComp}. Then,  $\det(\boldR\Gamma(\overline{V}, \overline{\boldD}))={\rm char}_{\cH_A(\Gamma_F^0)}(\boldR^2\Gamma(\overline{V}, \overline{\boldD}))$.
\end{theorem}

In other words, the determinant of the Iwasawa theoretic Selmer complex computes the characteristic ideal of an appropriately defined Iwasawa theoretic Selmer group.
\begin{proof}
\label{proof:thm:det_selCOm_over_H_A_=_charIdeal_Cohomo_deg2}
By Proposition~\ref{prop:projective_dim_H^2_SelComp}, we see that $\boldR^2\Gamma(\overline{V}, \overline{\boldD})$ has a projective resolution $P^\bullet=[P_1 \xrightarrow{} P_2]$ of dimension $1$ by finite $\cH_A(\Gamma_F^0)$-modules $P_i\simeq \cH_A(\Gamma_F^0)^{\oplus r}$ of the same rank. Under the running assumptions, we note that we automatically have
$\det\,\boldR\Gamma(\overline{V}, \overline{\boldD})\,=\,\det\,\boldR^2\Gamma(\overline{V}, \overline{\boldD})\,.$
Moreover, by Proposition~\ref{prop:det_torsion_M_Over_H_A_char_Ideal_M}, we have 
$\det\,\boldR^2\Gamma(\overline{V}, \overline{\boldD})\,=\,{\rm char}_{\cH_A(\Gamma_F^0)}\,\boldR^2\Gamma(\overline{V}, \overline{\boldD})\,,$
as required.
\end{proof}

\begin{theorem}
\label{thm:char_ideals_for_selCom_and_dual_selCom_H_A}
Suppose that the hypotheses of Theorem~\ref{thm:det_selCOm_over_H_A_=_charIdeal_Cohomo_deg2} are satisfied. Then, 
$${\rm char}_{\cH_A(\Gamma_F^0)}\,\boldR^2\Gamma(\overline{V}^*(1), \overline{\boldD}^\perp)\, = \,{\rm char}_{\cH_A(\Gamma_F^0)}\,\boldR^2\Gamma(\overline{V}, \overline{\boldD})^\iota\,.$$ 
\end{theorem}
\begin{proof}
\label{proof:thm:char_ideals_for_selCom_and_dual_selCom_H_A}
Since we work in the setting of Theorem~\ref{thm:det_selCOm_over_H_A_=_charIdeal_Cohomo_deg2} (in particular, the hypothesis of Theorem~\ref{thm:duality_selComplex_Iw_iso}(ii) holds), we have \begin{equation}
\label{eqn_2024_07_18_1121}
\boldR\Gamma(\overline{V}^*(1),\overline{\boldD}^\perp) \xrightarrow{\,\,\sim\,\,} \boldR{\rm Hom}_{\cH_A(\Gamma_{F}^0)}(\boldR\Gamma(\overline{V},\overline{\boldD}),\cH_A(\Gamma_{F}^0))[-3]\,.
\end{equation}
By Proposition~\ref{prop:projective_dim_H^2_SelComp}, the complex $\boldR\Gamma(\overline{V},\overline{\boldD})$ can be represented by a complex of the form 
$$\left[P_1\simeq \cH_A(\Gamma_F^0)^{\oplus r}\xrightarrow{\,\,\,u\,\,\,}\cH_A(\Gamma_F^0)^{\oplus r}\simeq P_2\right]\,,$$ and that 
\begin{equation}
\label{eqn_2024_07_18_1120}
    {\rm char}_{\cH_A(\Gamma_F^0)}\,\boldR^2\Gamma(\overline{V}, \overline{\boldD})\,=\,\det\,\boldR\Gamma(\overline{V},\overline{\boldD})\,=\,\det(u)\,. 
\end{equation}
It follows from \eqref{eqn_2024_07_18_1121} that 
\begin{equation}
\label{eqn_2024_07_18_1127}
    \det\,\boldR\Gamma(\overline{V}^*(1),\overline{\boldD}^\perp)\,=\,\det(^{t}u)\,=\det(u)^{\iota}\,, 
\end{equation}
where $^{t}u$ denotes the transpose of the map $u$. More precisely, we have 
\begin{equation}
\label{eqn_2024_07_18_1154}
\boldR\Gamma(\overline{V}^*(1),\overline{\boldD}^\perp)\simeq \left[P_2^* \stackrel{\,\,^{t}u\,\,}{\lra} P_1^* \right] \quad \hbox{ concentrated in degrees 1 and 2}\,,
\end{equation}
where $P_i^*$ ($i=1,2$) are still finite $\cH_A(\Gamma_F^0)$-modules. Moreover, the map $^{t}u$ is injective because ${\rm coker}(u)\simeq \boldR^2\Gamma(\overline{V},\overline{\boldD})$ is torsion by assumption (so that ${\rm Hom}({\rm coker}(u), \cH_A(\Gamma_F^0))=\{0\}$). 
Combining \eqref{eqn_2024_07_18_1120}, \eqref{eqn_2024_07_18_1127}, \eqref{eqn_2024_07_18_1154}, and Proposition~\ref{prop:det_torsion_M_Over_H_A_char_Ideal_M}, we conclude that 
$$
{\rm char}_{\cH_A(\Gamma_F^0)}\,\boldR^2\Gamma(\overline{V}^*(1),\overline{\boldD}^\perp)\,=\, \det\,\boldR^2\Gamma(\overline{V}^*(1),\overline{\boldD}^\perp)\,=\,  \det\,\boldR\Gamma(\overline{V}^*(1),\overline{\boldD}^\perp)\,=\,{\rm char}_{\cH_A(\Gamma_F^0)}\,\boldR^2\Gamma(\overline{V}, \overline{\boldD})^{\iota}\,,
$$
as required.
\end{proof}

\part{Arithmetic applications in the non-ordinary scenario}
\label{part_arithmetic_nonord}

\section{Rankin--Selberg products of non-ordinary families}
\label{chap:5_Rankin-Selberg_product}
\subsection{Non-ordinary families}
\label{sec_4_4_2024_07_12_1610}

We let $\cW={\rm Hom_{cont}}(\Zp^\times,\Cp^\times)$ denote the weight space. We consider $\cW$ as a rigid analytic space: the space $\cW$ is the disjoint union of $p-1$ open balls of radius $1$. 

We identify $\ZZ$ as a subset of $\cW$ via $\ZZ\ni k \mapsto \nu_k \in \cW$, where $\nu_k(x)=x^k$.  More generally, a homomorphism $\Zp^\times \to \overline{\QQ}_p^\times$ of the form $x \mapsto x^n\alpha(x)$, where $n \in \ZZ$ and $\alpha$ is a Dirichlet character of $p$-power conductor, is called a locally algebraic character of $\Zp^\times$. We sometimes denote this character by $\kappa=n+\alpha$. For such $\kappa$, we put $w(\kappa):= n.$
\begin{defn}
\label{defn:Defn3.7_OchiaiIMC}
Let $k_0$ be an integer and $r$ be a natural number.
    \item[i)] Let $U=B(k_0, p^{-r})$ be a rigid analytic space  (a wide-open disc), whose $\QQ_p$ points are identified with the open disc centered at $\mathbb{N}\supset k_0 \in \cW$ with radius $p^{-r}$.
    \item[ii)]  Let $\Lambda_U$ denote the power bounded analytic functions on $U$, i.e. $$\Lambda_U:= \{f \in E[[T]]: \quad \hbox{ the Gauss norm } ||f|| \hbox{ of } f \hbox{ is }  \leq 1 \} \simeq \cO_E[[T]]\,.$$
 We write $\kappa_U$ for the universal character given as the composition $\ZZ^\times_p \hookrightarrow{} \Lambda(\ZZ^\times_p)^\times \to \Lambda^\times_U$ (where the second arrow is the restriction of functions).
        
    \item[iii)] A character  $\eta\in U$ is called arithmetic if it belongs to $\ZZ \cap U$. When this is the case, the integer $k$ such that $\eta=\nu_k$ is called the weight of $\eta$.
\end{defn}

\subsubsection{Coleman families}
Recall that given a cuspidal eigen-newform $f \in S_{k_f+2}(\Gamma_1(N))$, we denote by $f_{\alpha_f}\in S_{k_f+2}(\Gamma_1(N)\cap \Gamma_0(p))$ its $p$-stabilization. We henceforth assume that $k_f\geq 0$. As in \cite[\S4.5]{LZColeman}, let us set $B_U=\Lambda_U[1/p]$ and let $I_k:=\ker(\nu_k)$ for an integer $k \geq 0$.

\begin{defn}
\label{defn:Coleman_family}
A Coleman family $\f$ over $U$ (of tame level $N$) is a power series 
$$\f=\sum_{n\geq 1}a_n(\f)q^n \in \Lambda_U[[q]],$$ 
with coefficients $a_1(\f)=1$ and $a_p(\f)$ is invertible in $B_U$ such that, for all but finitely many classical weights $k \in U \cap \ZZ_{\geq 0}$, the series $\f_k=\sum_{n \geq 1}a_n(\f)(k) \in \cO_E[[q]]$ is the $q$-expansion of a $p$-stabilized normalised eigenform of weight $k+2$ and level $\Gamma_1(N) \cap \Gamma_0(p)$. Moreover, the slope $val_p(a_p(\f_k))$ is constant.
\end{defn}

Coleman families exist and are (in a certain sense) unique (cf. \cite{bellaiche2012}, \cite[Theorem 4.6.4]{LZColeman}):

\begin{theorem}
\label{thm:Theorem 4.6.4_LZColeman}
For $f$ as above, there exists a disc $U \ni k_{f}$ in $\cW$ and a unique Coleman family $\f$ such that ${\f}_{k_f}=f_{\alpha_f}$.
\end{theorem}

\begin{theorem}[\cite{BB_PR_Volume}, \S6.3]
\label{thm:Theorem4.6.6_LZCol} 
Let $\f$ be the Coleman family passing through $f_{\alpha_f}$ over a sufficiently small wide open disc $U \ni k_{f}$. Then, there exist a free $B_U$-module $V_\f$ of rank $2$ with Galois action $G_\QQ$ such that, for each integer $0<k \in U,$ the modular form $\f_k$ is a classical eigenform, and we have isomorphisms of $E$-linear $G_{\QQ}$-representations $V_{\f}/I_k V_\f=V_E(\f_k)$.  
\end{theorem}

Suppose that ${\rm Spm}(A)=X\subset U$ is an affinoid disc. We will also consider the restriction of $\f$ to $X$ and denote it by the same symbol. Moreover, we will consider the Galois representations $V_{\f}\otimes_{B_U}A$ and $V_{\f}^*\otimes_{B_U}A$ over $X$ (where $V_\f^*={\rm Hom}_{B_U}(V_\f,B_U)$ is the dual representation), and by slight abuse of language, we will continue to denote them by $V_{\f}$ and $V_{\f}^*$, respectively.

\subsubsection{Residual representation}
\label{sec:Resi_reps_Coleman}

As explained in \cite[Remark 1.4 and Remark A.1]{B_y_kboduk_2021} and \cite[Theorem 4.3]{OchiaiRay}, we may choose a $G_{\QQ}$-stable $\Lambda_U$-lattice $T_\f$ (resp. $\Lambda_U$-lattice $T^*_{\f}$) inside of $V_\f$ (resp. inside of $V^*_\f$). Let us denote by $\m_U\subset \LL_U$ the maximal ideal, and consider the residual representation
$$\widetilde{\rho}_{\f}\,:\, G_{\QQ} \to {\rm GL}(T_\f) \to {\rm GL}(T_\f/\m_U T_\f)\simeq {\rm GL}_2(\texttt{k}_E)\,.$$
We note that 
$\widetilde{\rho}_{\f}$ is isomorphic, thanks to Theorem~\ref{thm:Theorem4.6.6_LZCol}(4) to the residual representation $\widetilde{\rho}_f$ of $f$, and it is therefore absolutely irreducible by assumption. As a result, $\widetilde{\rho}_{\f}$ is independent of the choice of the lattice $T_\f$. The same discussion applies to the dual Galois representation.

We will also consider the Galois representations $T_{\f}\otimes_{\Lambda_U} A^\circ$ and $T_{\f}^*\otimes_{\Lambda_U} A^\circ$ where $A^\circ$ denotes the unit ball in the Banach ring $A=A^\circ[1/p]$. By a slight abuse of notation, we will continue to denote them by $T_{\f}$ and $T_{\f}^*$, respectively.

\subsubsection{Triangulations}
\label{sec:Triangulations}

Let $A$ denote a reduced affinoid algebra\footnote{For our eventual goals, it suffices to work with a Tate algebra, which is of course reduced.} and $\alpha \in A^\times.$ 
\begin{defn}
\label{defn:Def6.2.1_LZCol}
We denote by $\cR_A(\alpha^{-1}):=\cR_{A}\cdot e$ the free $(\varphi,\Gamma)$-module of rank one over $\cR_{A}$ with basis vector $e$ such that $\varphi(e)=\alpha^{-1}e$ and $\gamma(e)=e$ for $\gamma \in \Gamma=\Gamma_{\QQ}^0$. 
\end{defn}
\begin{defn}
\label{defn:triangulations}
Let $D$ be a $(\varphi,\Gamma)$-module over $\cR_{A}$, which is free of rank $2.$  A (saturated) triangulation of $D$ is a short exact sequence 
$$0 \lra \mathcal{F}^+D \lra D \lra \mathcal{F}^-D \lra 0\,,$$ 
where $\mathcal{F}^+D$ and $\mathcal{F}^-D$ are free $\cR_{A}$-modules of rank $1$. 
\end{defn}

\begin{theorem}
[Liu~\cite{liu-CMH}, see also Theorem 6.3.2 in \cite{LZColeman}]
\label{thm:Triangulation_Of_Phi_gamma_module}
Suppose that \eqref{item_nobf} holds. Then there exists an affinoid disc ${\rm Spm}(A)=X \subset  \cW$ containing $k_f$ such that the $(\varphi,\Gamma)$-module $D^\dagger_{\rm rig}(V^*_{\f})$ has a canonical triangulation with $\mathcal{F}^-D^\dagger_{\rm rig}(V^*_{\f}) \simeq\cR_A(a_p(\f)^{-1})$ and $\mathcal{F}^+D^\dagger_{\rm rig}(V^*_{\f})(-1-\bbchi_\f) \simeq \cR_A(a_p(\f) \varepsilon_{\f}(p)^{-1})$. Here, $\varepsilon_{\f}$ is the tame nebentype character of $\f$ and $\bbchi_\f$ is the restriction of the universal cyclotomic character to $X$.

\end{theorem}

\subsection{Punctual Selmer complexes}
\label{sec:Triangulation_Appli_BF}
Let $f$, $g$ be cuspidal normalized newforms of weights $k_f+2> k_g+2\geq 0$, levels $N_f$, $N_g$, and nebentype characters $\varepsilon_f$, $\varepsilon_g$, respectively. Put $N:=\gcd(N_f,N_g)$. Until the end of this paper, we assume that \eqref{item_nobf} holds for both $f_{\alpha_f}$ and $g_{\alpha_g}$. We fix a finite extension $E/\QQ_p$ containing the Hecke fields of $f_{\alpha_f}$ and $g_{\alpha_g}$. Let us put $V=V_E(f \otimes g)^*(-j)$ for an integer $k_g+1 \leq j \leq k_f$. 

Let $S$ denote the finite set of places of $\QQ$ containing $\infty$ and all primes dividing $pN$, and recall that $\QQ_S$ is the maximal extension of $\QQ$ unramified outside $S \cup S_\infty$. Recall also that $H_{\QQ,S}:={\rm Gal}(\QQ_S/\QQ_\infty)$\,. Throughout this subsection, we work under the validity of Assumptions~\ref{hypo_ranking_sel_1} (in particular, $p\nmid N$).

\subsubsection{Punctual triangulation}
By the definition of $D^\dagger_{\rm rig}(V)$, we have $D^\dagger_{\rm rig}(V)=D^\dagger_{\rm rig}(V_E(f)^*) \otimes_{\cR_E} D^\dagger_{\rm rig}(V_E(g)^*)(-j)$. Let us consider the rank-one $(\varphi, \Gamma_{\QQ_p})$-module 
$D_f:= \mathcal{F}^+D^\dagger_{\rm rig}(V_E(f)^*)  \subset D^\dagger_{\rm rig}(V_E(f)^*)\,,$ 
and similarly define $D_g[-j]:=\mathcal{F}^+D^\dagger_{\rm rig}(V_E(g)^*(-j)) \subset D^\dagger_{\rm rig}(V_E(g)^*(-j))$. We define the $(\varphi, \Gamma_{\QQ_p})$-submodule $D \subset D^\dagger_{\rm rig}(V)$ of rank $2$ on setting  $D:= D_f \otimes  D^\dagger_{\rm rig}(V_E(g)^*(-j))\,.$

\subsubsection{Amplitude of punctual Selmer complexes}
\label{sec:Amplitude_of_punctual_Selmer_complexes}
The following lemma (due to Benois and Pottharst) explains the relationship between the cohomology of Selmer complexes and Bloch--Kato Selmer groups. 

We recall the Dieudonn\'e modules $D_{\rm st}(D)$ and $D_{\rm st}(D^\dagger_{\rm rig}(V))=D_{\rm st}(V)$ from \cite[\S2.2]{p-adicHodgetheorybenoisheights}. 
\begin{lemma}
\label{lemma_2024_04_02_1704}
With $D$ as above, the following hold under Assumptions~\ref{hypo_ranking_sel_1}:
     \item[i)] ${\rm Fil}^0D_{\rm st}(D)=0$ and ${\rm Fil}^0(D_{\rm st}(D^\dagger_{\rm rig}(V))/D_{\rm st}({D}))=D_{\rm st}(D^\dagger_{\rm rig}(V))/D_{\rm st}(D)$\,. 
     \item[ii)] $(D_{\rm st}(D^\dagger_{\rm rig}(V))/D_{\rm st}(D))^{\phi=1,N=0}=0$ and $D_{\rm st}(D)/(ND_{\rm st}(D)+(p\phi-1)D_{\rm st}(D))=0.$ 
     \item[iii)] $\boldR^1\Gamma(V,D)\simeq H^1_{\rm f}(\QQ,V)\,.$
\end{lemma}

\begin{proof}
    Part (i) is clear since $V$ and $D$ are semistable, and the Hodge--Tate weights of $D$ are positive, whereas the Hodge--Tate weights of $D^\dagger_{\rm rig}(V)/D$ are non-positive. We remark that our sign convention for Hodge--Tate weights is that, there is a filtration jump at an integer $n$ if and only if $-n$ is a Hodge--Tate weight.

We next prove the first assertion in (ii). Note that 
$$(D_{\rm st}(D^\dagger_{\rm rig}(V))/D_{\rm st}(D))^{\phi=1,N=0} \simeq D_{\rm st}(D^\dagger_{\rm rig}(V)/D)^{\phi=1,N=0} \simeq D_{\rm cris}(D^\dagger_{\rm rig}(V)/D)^{\phi=1}.$$ 
By the proof of \cite[Theorem 8.1.4]{LZColeman}, we see that $D_{\rm cris}(D^\dagger_{\rm rig}(V)/D) \simeq D_{\rm cris}(V)^{\alpha_f\,\circ}$, where $D_{\rm cris}(V)^{\alpha_f\,\circ}$ is given as in \ref{defn:Definition_8.1.2_LZCol} (note that we use  \eqref{item_reg} in the definition of this module).
Hence, 
$D_{\rm cris}(D^\dagger_{\rm rig}(V)/D)^{\varphi=1} \simeq (D_{\rm cris}(V)^{\alpha_f\,\circ})^{\varphi=1}\,.$
It follows from \cite[Proposition 3.7]{jayanalyticfamilies} that we have 
$$H^0(\QQ_p, D^\dagger_{\rm rig}(V)/D)=D_{\rm cris}(D^\dagger_{\rm rig}(V)/D)^{\varphi=1}\,,$$
since $D$ is semi-stable, the Hodge--Tate weights of $D$ are positive, and the Hodge-Tate weights of $D^\dagger_{\rm rig}(V)/D$ are non-positive. 

Note that $\varphi$ has eigenvalues $\{p^j\alpha^{-1}_f\alpha^{-1}_g, p^j\alpha^{-1}_f\beta^{-1}_g\}$ on $D_{\rm cris}(D^\dagger_{\rm rig}(V)/D)$. By \eqref{item_NZ}, we see that $\alpha_f\alpha_g \neq p^j$ and $\alpha_f\beta_g \neq p^j$ and thus $p^j\alpha^{-1}_f\alpha^{-1}_g \neq 1\neq p^j\alpha^{-1}_f\beta^{-1}_g$. We have thus proved that 
$$(D_{\rm st}(D^\dagger_{\rm rig}(V))/D_{\rm st}(D))^{\phi=1,N=0} \simeq D_{\rm cris}(D^\dagger_{\rm rig}(V)/D)^{\phi=1} \simeq (D_{\rm cris}(V)^{\alpha_f\,\circ})^{\phi=1}=0,$$
as required.

We now check the second assertion in (ii). Note that $D$ is crystalline, and hence the monodromy operator $N$ annihilates $D_{\rm st}(D)$. As a result, the second claim in (2) can be rephrased as the requirement that $p\varphi -1$ is injective on $D_{\rm cris}(D)$. This property holds thanks to \eqref{item_NZ}.

Part (iii) follows from \cite[Theorem 4.1(4)]{pottharst}, \cite[Proposition 6(v)]{selComp_p-adicHodgetheory} and \cite[Proposition 3.7]{jayanalyticfamilies} combined with the previous claims in this lemma.
\end{proof}

The following is the main application of the Beilinson--Flach element Euler system to the determination of the amplitude of the Selmer complex $\boldR\Gamma(\overline{V}, \overline{D})$.  For more details on Eisenstein classes and Beilinson--Flach Euler system, see \cite[\S3.2-\S3.5]{LZColeman}, as well as \cite{BDR1,BDR2}. 

\begin{theorem} 
\label{thm:Vanish_H^1_torsion-ness_of_H^2_for_specialization}
Under the validity of  Assumptions~\ref{hypo_ranking_sel_1}, we have:
    \item[i)] $\boldR^3\Gamma(\overline{V^*(1)}, \overline{D^\perp})=0=\boldR^3\Gamma(\overline{V}, \overline{D}).$
   
    \item[ii)] Suppose that there exists an element $\tau \in {\rm Gal}(\overline{\QQ}/\QQ(\mu_{p\infty}))$ such that $V\big{/}(\tau-1)V$ is $1$-dimensional (cf. Proposition~\ref{prop:Prop_4.2.1_Loff_Ima} for a sufficient condition). If, in addition, $L(f, g, 1 + j) \neq 0$, then 
    \begin{equation}
    \label{eqn_2024_03_25_2}
        \boldR^1\Gamma({V}, {D})=0=\boldR^1\Gamma(\overline{V}, \overline{D})\,,
    \end{equation}
    and the $\cH_E(\Gamma^0_\QQ)$-module $\boldR^2\Gamma(\overline{V}, \overline{D})$ is torsion.
\end{theorem}

\begin{proof}
    \item[i)] Since $V$ is irreducible as a $G_{\QQ}$-module and ${\rm dim}_E(V) > 1$, then (as explained in the proof of \cite{factorization_alge_p_adic}, Corollary 5.18), we infer that $V^{H_{\QQ,S}}=\{0\}=V^*(1)^{H_{\QQ,S}}$. This portion then follows from Theorem~\ref{thm:Prop_Iw_SelComp}(1).

    \item[ii)]  As the hypotheses of \cite[Corollary 8.3.2]{LZColeman} are satisfied for $V$, it follows\footnote{Here, we note that \cite[Corollary 8.3.2]{LZColeman} employs the Beilinson--Flach Euler system 
    $$_c\cB{\mathcal F}^{[k_f,k_g,j]}_{m,1} \in H^1\left(\QQ(\mu_m), V\right)$$ 
     given as in \cite[Definition 3.5.2]{LZColeman}, for $a=1$ and $k_g+1 \leq j \leq k_f$.} that 
    the Bloch–Kato Selmer group $H^1_{\rm f}(\QQ, V)$ is zero. By Lemma \ref{lemma_2024_04_02_1704}, we have
$\boldR^1\Gamma(V, D)=H^1_{\rm f}(\QQ, V)$ under our running assumptions. Then by Theorem \ref{thm:Prop_Iw_SelComp}(3), we see that \begin{equation}
\label{eqn_2024_03_26}
    \boldR^1\Gamma(\overline{V}, \overline{D})_{\Gamma^0_{\QQ}}=0.
\end{equation}

As explained by \cite[Remark 5.31]{factorization_alge_p_adic}, we have 
$$\boldR\Gamma(\overline{V^*(1)}, \overline{D^\perp})\,, \quad \boldR\Gamma(\overline{V}, \overline{D}) \,\in \,\cD^{[1,2]}_{\rm perf}(\cH_E(\Gamma^0_\QQ))$$
since  $\boldR^3\Gamma(\overline{V^*(1)}, \overline{D^\perp})=0=\boldR^3\Gamma(\overline{V}, \overline{D})\,.$ Then, the exact sequence in Theorem~\ref{thm:Prop_Iw_SelComp}(iv) (applied with $i = 3$) gives an isomorphism between torsion submodule of $\boldR^1\Gamma(\overline{V}, \overline{D})$ and $\boldR^3\Gamma(\overline{V^*(1)}, \overline{D^\perp})=0.$ Thus, by \cite[Proposition 1.1]{pottharst} (the structure theorem
for coadmissible $\cH_E(\Gamma^0_{\QQ})$-modules), we see that $\boldR^1\Gamma(\overline{V}, \overline{D})$ is a finitely generated free $\cH_E(\Gamma^0_{\QQ})$-module, say of rank $n$. This fact, combined with \eqref{eqn_2024_03_26}, yields
$0=\boldR^1\Gamma(\overline{V}, \overline{D})_{\Gamma^0_\QQ} \simeq E^n\,,$
and we conclude that $\boldR^1\Gamma(\overline{V}, \overline{D})=0$, as required. Finally, by Theorem \ref{thm:Prop_Iw_SelComp}(ii) combined with \eqref{eqn_2024_03_25_2}, we infer that $\boldR^2\Gamma(\overline{V}, \overline{D})$ is torsion as an $\cH_E(\Gamma^0_\QQ)$-module.
\end{proof}

\subsubsection{Functional Equation} We are now ready to state and prove the functional equation for the algebraic Rankin--Selberg $p$-adic $L$-function associated with $f_{\alpha_f}\otimes g_{\alpha_g}$.

We recall that ${\rm char}_{\cH_{E}(\Gamma_{\QQ}^0)}$ denotes the characteristic ideal of an ${\cH_{E}(\Gamma_{\QQ}^0)}$-module given as in Proposition~\ref{prop:det_torsion_M_Over_H_E_char_Ideal_M}, and for any $\ZZ_p[[\Gamma_{\QQ}^0)]]$-module $M$, we set $M^\iota:=M\otimes_{\ZZ_p[[\Gamma_{\QQ}^0)]]}\ZZ_p[[\Gamma_{\QQ}^0)]]^\iota$ and $\ZZ_p[[\Gamma_{\QQ}^0)]]^{\iota}$ is the free $\ZZ_p[[\Gamma_{\QQ}^0)]]$-module of rank one on which $g\in \Gamma_{\QQ}^0$ acts via multiplication by $g^{-1}$.

\begin{theorem}[Punctual functional equation]
\label{thm:FuncEq_SelCom_Rankin_SelbergProducts_0_vari}
Under the validity of  Assumptions~\ref{hypo_ranking_sel_1}, we have
$${\rm char}_{\cH_{E}(\Gamma_{\QQ}^0)}\,\left(\boldR^2\Gamma(\overline{V^*}(1), \overline{D}^\perp)\right) = {\rm char}_{\cH_{E}(\Gamma_{\QQ}^0)}\,\left(\boldR^2\Gamma(\overline{V}, \overline{D})\right)^\iota\,.$$ 
\end{theorem}

\begin{proof}
This is a special case of Theorem~\ref{thm:char_ideals_for_selCom_and_dual_selCom_H_A} with $A=E$. Note that we have verified in Theorem~\ref{thm:Vanish_H^1_torsion-ness_of_H^2_for_specialization} that all the assumptions required in Theorem~\ref{thm:char_ideals_for_selCom_and_dual_selCom_H_A} hold in this scenario.
\end{proof}

\subsection{Selmer complexes for families}
\label{subsec_6_3_2024_11_4_1257}
Throughout \S\ref{subsec_6_3_2024_11_4_1257}, we work under the validity of Assumptions~\ref{hypo_ranking_sel_1}.
\subsubsection{1-parameter families}  
\label{sec:Triangulations_in_families_1_1}
Suppose that $k_{f}\in X \subset U$ is a 1-dimensional standard affinoid over $E$ contained in a wide-open disc $U$.  Let $\f$ be a Coleman family over $U$ passing through $f_{\alpha_f}$. Let us consider the corresponding big Galois representation 
${\mathbf W}:=V^*_{\f}\otimes_E V_E(g)^*$ over $A$. Let us set $\boldD_{\mathbf W}:= \mathcal{F}^+D^\dagger_{\rm rig}(V^*_{\f}) \,\otimes_E\,  D^\dagger_{\rm rig}(V_E(g)^*)$, which is a $(\varphi,\Gamma_{\QQ_p})$-submodule of $D^\dagger_{\rm rig}({\mathbf W})$ of rank $2$. As usual, we let $\boldD_{\mathbf W}^\perp$ denote the dual triangulation at $p$ for ${\mathbf W}^*(1)$, which is given by 
$${\rm Hom}_{\cR_{A_1}}\left(D^\dagger_{\rm rig}(V^*_{\f})/\boldD_{\f}\,,\,A_1\right)\,\widehat{\otimes}_{E}\, D^\dagger_{\rm rig}(V_E(g)(1))\subset D^\dagger_{\rm rig}({\mathbf W}^*(1))\,.$$

Before we explain our main result in \S\ref{sec:Triangulations_in_families_1_1} (Theorem~\ref{thm:thm:H^1_zero_torsion_H^2_ranking_sel_ColemanFamily_1_vari}), we provide ring theoretic background. Our discussion is more general than we need in \S\ref{sec:Triangulations_in_families_1_1}, and it will be used also in \S\ref{sec:Triangulations_in_families_2}.
\begin{defn}
    Given a Noetherian ring $R$ and an ideal $I<R$, let us denote by $(R,I)\langle x_1,\cdots,x_n\rangle:=\varprojlim R/I^n[x_1,\cdots,x_n]$, the restricted power series ring over $A$. When $R$ is a local Noetherian ring with maximal ideal $\frakm$, we simply write $R\langle x_1,\cdots,x_n\rangle$ in place of  $(R,I)\langle x_1,\cdots,x_n\rangle$.
\end{defn}
Note that when $R=\cO_E$, we recover the Tate algebra $E\langle x_1,\cdots,x_n\rangle$ in $n$-variables:
$$E\langle x_1,\cdots,x_n\rangle=\cO_E\langle x_1,\cdots,x_n\rangle[1/p]\,.$$
In what follows, we will abbreviate $\LL_{\cO_E}(\ZZ_p)$ (or $\LL_{\cO_E}(\Gamma_?^0)$, where $?=\QQ$, $\QQ_p$) with $\LL_{\cO_E}$ when there is no danger of confusion on the choice of the underlying profinite group.

\begin{proposition}
    \label{prop_2024_05_21_1411}
    \item[i)] The ring $\LL_{\cO_E}\langle x_1,\cdots,x_n\rangle$ is regular of dimension of $n+2$.
    \item[ii)] We have an isomorphism 
    $$\LL_{\cO_E}\langle x_1,\cdots,x_n\rangle \simeq \LL_{\cO_E}(\ZZ_p)\,\widehat{\otimes}_{\cO_E}\,\cO_E\langle x_1,\cdots,x_n\rangle\,.$$
    of topological rings.
    \item[iii)] The ring 
    $$\LL_{\cO_E}\langle x_1,\cdots,x_n\rangle[1/p]\simeq (\LL_{\cO_E}(\ZZ_p)\,\widehat{\otimes}_{\cO_E}\,\cO_E\langle x_1,\cdots,x_n\rangle)[1/p]$$
    is regular of dimension $n+1$.
\end{proposition}

\begin{proof}
    The first assertion is \cite[Th\'eor\`eme 3]{Salmon}, whereas the second follows from \cite[\href{https://stacks.math.columbia.edu/tag/0AL0}{Remark 0AL0}]{stacks-project}. To prove the final claim, we first note that $\LL_{\cO_E}\langle x_1,\cdots,x_n\rangle[1/p]$ is regular, being the localization of a ring. 
    
    It remains to prove that its dimension equals $n+1$. We note that the radical of $\LL_{\cO_E}\langle x_1,\cdots,x_n\rangle$ contains the radical of $\LL_{\cO_E}$ thanks to \cite[Lemme 2]{Salmon}. As a result, all maximal ideals of $\LL_{\cO_E}\langle x_1,\cdots,x_n\rangle$ contain $p$. This shows that the heights of prime ideals in $\LL_{\cO_E}\langle x_1,\cdots,x_n\rangle[1/p]$ can be at most $n+1$, and concludes the proof of our proposition.
\end{proof}

\begin{remark}
    \label{remark_2024_5_21_1457}
    We remark that we have the following tautological ring isomorphism:
   \begin{align}
   \label{align_remark_2024_5_21_1457}
       \left(\LL_{\cO_E}(\ZZ_p)\otimes_{\cO_E}\cO_E\langle x_1,\cdots,x_n\rangle\right)[1/p]\simeq \LL_{E}(\ZZ_p) \otimes_E E\langle x_1,\cdots,x_n\rangle\,.
   \end{align}
   In what follows, as in \cite{jayanalyticfamilies}, the completed tensor product $ \LL_{E} \widehat{\otimes}_E E\langle x_1,\cdots,x_n\rangle$ shall mean $$\left(\LL_{\cO_E}(\ZZ_p)\widehat{\otimes}_{\cO_E}\cO_E\langle x_1,\cdots,x_n\rangle\right)[1/p]=\LL_{\cO_E}(\ZZ_p)\langle x_1,\cdots,x_n\rangle[1/p]\,,$$ 
   where the completed tensor is calculated with respect to the $\frakm_{\LL_{\cO_E}}$-adic topology on the first factor, and $p$-adic topology on the second factor. 
\end{remark}

We will use Proposition~\ref{prop_2024_05_21_1411} in the proof of Lemma~\ref{prop:Irr_V_0_with_item_Irr_res_F} below with the following objects: Let $A^\circ= \cO_E\langle X \rangle$ denote the unit ball inside the Banach space $A=E\langle X\rangle$
(with respect to the Gauss norm), so that $A^\circ[1/p]=A$. As explained in \cite[Theorem 4.3]{OchiaiRay}, there exists a free $A^\circ$-module ${\mathbf W}^\circ$ of rank $2$ which is endowed with a continuous action of $G_\QQ$ such that 
${\mathbf W}^\circ\otimes_{A^\circ}A\simeq {\mathbf W}\,.$
We shall work with the continuous\footnote{Hand-in-hand with Remark~\ref{remark_2024_5_21_1457}, we note that the Galois action on $\LL_{\cO_E}^\iota$ is continuous when the ring $\LL_{\cO_E}$ is endowed with the $\frakm$-adic topology, but not so for the $p$-adic topology.} Galois representation
${\mathbf W}\,\widehat{\otimes}_E\, \LL_E^\iota:=({\mathbf W}^\circ \,\widehat{\otimes}_{\cO_E}\,\LL_{\cO_E}^\iota)[1/p]$
(cf. Remark~\ref{align_remark_2024_5_21_1457}), which is a free $\LL_{\cO_E}\langle X\rangle[1/p]$-module of rank $2$.

Our method to prove Theorem~\ref{thm:thm:H^1_zero_torsion_H^2_ranking_sel_ColemanFamily_1_vari} relies heavily on the approach developed in K\"u\c{c}\"uk's thesis \cite{factorization_alge_p_adic}; see especially Theorem~7.8 in op. cit.  This required a rather stringent hypothesis (recorded as item (c) following Remark 7.2 in  \cite{factorization_alge_p_adic}; see also \eqref{eqn_2024_05_21_1521} below), which will be relaxed in K\"u\c{c}\"uk's forthcoming addendum. We give an overview of this improvement in the present subsection\footnote{We thank F. K\"u\c{c}\"uk for sharing with us the details of this forthcoming addendum.}.

\begin{proposition}
\label{prop:Irr_V_0_with_item_Irr_res_F}
We have 
    \begin{equation}
    \label{eqn_2024_05_21_1521}
        H^0(H_{\QQ,S},({\mathbf W}\,\widehat\otimes_{E}\, \Lambda_{E}^\iota)/\fp)=0= H^0(H_{\QQ,S},({\mathbf W}^*(1)\,\widehat\otimes_{E}\, \Lambda_{E}^\iota)/\fp)
        \end{equation}
    for each prime ideal $\fp$ in $\Lambda_{E}\,\widehat{\otimes}_E\, E\langle X \rangle=\Lambda_{\cO}\langle X\rangle[1/p]$.
\end{proposition}
\begin{proof}
\label{proof:prop:Irr_V_0_with_item_Irr_res_F}

We shall prove only the vanishing on the left since the argument to prove the one on the right is entirely identical. 

Let us first assume that $\mathfrak{p}$ is a maximal ideal. The discussion in the proof of Proposition~\ref{prop_2024_05_21_1411}(iii) shows that $\mathfrak{p}$ corresponds to a unique height-2 prime $\mathfrak{p}^\circ$ of $\LL_{\cO_E}\langle X\rangle$ that does not contain $p$, and $\widetilde{\mathfrak{p}}:=(\varpi_E)+\mathfrak{p}^\circ$ is a height-3 (maximal) ideal of $\LL_{\cO_E}\langle X\rangle$. Moreover,
$$T_\mathfrak{p}:=({\mathbf W}^\circ\,\widehat{\otimes}_{\cO_E}\,\LL_{\cO_E}^\iota)/\mathfrak{p}^\circ \subset ({\mathbf W}\,\widehat{\otimes}_E\,\LL_{E}^\iota)/\mathfrak{p}=:V_\mathfrak{p}$$
is a Galois-stable $\cO_\fp:=\LL_{\cO_E}\langle X\rangle/\mathfrak{p}^\circ$-lattice in the $F_\mathfrak{p}:=\LL_{\cO_E}\langle X\rangle/\mathfrak{p}$-vector space $V_\mathfrak{p}$, and 
$$\overline{T}_\mathfrak{p}:=({\mathbf W}^\circ\,\widehat{\otimes}_{\cO_E}\,\LL_{\cO_E}^\iota)/\widetilde{\mathfrak{p}}=T_\mathfrak{p}/(\varpi_E)\simeq \texttt{k}_E^{\oplus 2}$$
is the residual representation, on which the Galois group acts via $\widetilde{\rho}_{\f}^*\otimes \widetilde{\rho}_g^*$, where the former factor is given as in \S\ref{sec:Resi_reps_Coleman} (and the latter is the residual representation attached to $g$). We are set to prove that $ H^0(H_{\QQ,S},V_\fp)=0$. Since the sequence
$$0\lra H^0(H_{\QQ,S},T_\fp)\lra H^0(H_{\QQ,S},V_\fp)\lra H^0(H_{\QQ,S},\overline{T}_\fp)$$
is exact, it suffices to prove that $H^0(H_{\QQ,S},\overline{T}_\fp)=0$ (since $T_\fp$ has no non-zero $p$-divisible elements, and any submodule of $V_\fp$ is $p$-divisible).

We remark that under our running assumption \eqref{item_p_large}, the following hypothesis is also valid:
\begin{enumerate}
    \item[(\mylabel{item_Irr_res_F}{$\,{\textbf{Irr}}\,$})] The residual representation $\widetilde{\rho}_{V^*_\f}\otimes \widetilde{\rho}_g$ is an absolutely irreducible $G_{\QQ}$-representation.
\end{enumerate}

Relying on \eqref{item_Irr_res_F}, we may argue exactly as in the proof of \cite[Corollary 5.18]{factorization_alge_p_adic} to see conclude the vanishing  $H^0(H_{\QQ,S},\overline{T}_\fp)=0$, and our proposition is proved for maximal ideals $\fp$.

Suppose now that $\fp$ is a height-1 prime ideal of the 2-dimensional regular ring $\LL_{\cO_E}\langle X\rangle[1/p]$. Then $\fp=(\wp)$ is necessarily a principle ideal, and it is contained in a maximal ideal $\frakm=(\wp',\wp)$ where $\{\wp',\wp\}$ is a regular sequence. As before, let us set $V_\mathfrak{p}:=({\mathbf W}\,\widehat{\otimes}_E\,\LL_{E}^\iota)/\mathfrak{p}$ and similarly, $V_\mathfrak{m}:=({\mathbf W}\,\widehat{\otimes}_E\,\LL_{E}^\iota)/\mathfrak{m}$. We would like to prove that $H^0(H_{\QQ,S},V_\mathfrak{p})=0$.

Let us consider 
$$0\lra V_\mathfrak{p} \xrightarrow{\wp'} V_\mathfrak{p}\lra V_\mathfrak{m}\lra 0$$
which is exact thanks to the fact that $\{\wp',\wp\}$ is a regular sequence. On passing to $H_{\QQ,S}$-invariants, we deduce that the sequence 
$$0\lra H^0(H_{\QQ,S},V_\mathfrak{p}) \xrightarrow{\wp'} H^0(H_{\QQ,S},V_\mathfrak{p})\lra H^0(H_{\QQ,S},V_\mathfrak{m})=0\,.$$
is exact, where the right-most vanishing follows from our discussion concerning the case of maximal ideals. This shows that $H^0(H_{\QQ,S},V_\mathfrak{p})$ is $\wp'$-divisible, and in turn also that it vanishes. 

Our proposition is now proved when $\fp$ is either a maximal or a height-1 prime. The only remaining case is when $\fp=(0)$ is the unique height-0 prime, and our proposition can be proved in this case by mimicking the argument employed in the previous part.
\end{proof}
We now prove the first main result of \S\ref{sec:Triangulations_in_families_1_1}, where we closely follow the proof of \cite[Theorem 7.8]{factorization_alge_p_adic}.

\begin{theorem}
\label{thm:thm:H^1_zero_torsion_H^2_ranking_sel_ColemanFamily_1_vari}
Under the validity of Assumptions~\ref{hypo_ranking_sel_1}, we have:
    \item[i)] $\boldR^3\Gamma(\overline {\mathbf W}, \overline \boldD_{\mathbf W})=0=\boldR^3\Gamma(\overline{{\mathbf W}^*(1)}, \overline{\boldD_{\mathbf W}}^\perp).$
    \item[ii)] $\boldR\Gamma(\overline{\mathbf W},\overline{\boldD}_{\mathbf W})$, $\boldR\Gamma(\overline{{\mathbf W}^*}(1), \overline{\boldD_{\mathbf W}}^\perp) \in \cD^{[1,2]}_{\rm perf}(\cH_{A}(\Gamma_{\QQ}^0)).$
    \item[iii)] Suppose in addition that the big image hypothesis in the statement of Theorem~\ref{thm:Vanish_H^1_torsion-ness_of_H^2_for_specialization}(ii) holds for some classical specialization of $\f$. Then, $\boldR^1\Gamma(\overline {\mathbf W}, \overline \boldD_{\mathbf W})=0=\boldR^1\Gamma(\overline{{\mathbf W}^*}(1), \overline{\boldD_{\mathbf W}}^\perp).$ 
\end{theorem}

\begin{proof}
    \item[i)] Let us consider an arbitrary classical specialization $\f_k$ of $\f$, say of weight $k$, and let $\fp_k \in {\rm Spm}(A)$ denote the corresponding maximal ideal. It follows from the control theorem for Selmer complexes (cf. (CTRL-2) in \cite{factorization_alge_p_adic}), we have 
    $$\boldR^3\Gamma(\overline {\mathbf W}, \overline \boldD_{\mathbf W}) \otimes_{A} A/\fp_{k} \simeq \boldR^3\Gamma(\overline{{\mathbf W}/\fp_{k}{\mathbf W}}, \overline{\boldD_{\mathbf W}/\fp_{k}\boldD_{\mathbf W}})=\{0\}\,,$$
   where the vanishing follows from Theorem~\ref{thm:Vanish_H^1_torsion-ness_of_H^2_for_specialization}(i). By \cite[Lemma 7.7]{factorization_alge_p_adic}, we conclude that $\boldR^3\Gamma(\overline{{\mathbf W}}, \overline{\boldD_{\mathbf W}})=0$. The proof that $\boldR^3\Gamma(\overline{\boldV^*_1}(1), \overline{\boldD_{\mathbf W}}^\perp)=0$ is similar.
    
    \item[ii)] We infer from \cite[Lemma 5.25]{factorization_alge_p_adic} that $H^1(I_\fq, {\mathbf W})$ is a projective $A$-module for each $\fq \in \Sigma_p$. Our second assertion now follows from \cite[Theorem 5.21]{factorization_alge_p_adic} combined with Part (i).
    
    \item[iii)]    We have seen in the proof of Theorem~\ref{thm:Vanish_H^1_torsion-ness_of_H^2_for_specialization}(ii) that $\boldR^1\Gamma({\mathbf W}/\fp_{k}{\mathbf W}, \boldD_{\mathbf W}/\fp_{k}\boldD_{\mathbf W})=\{0\}$. It then follows from the control theorem (CTRL-3) in \cite[Proposition 5.11]{factorization_alge_p_adic} that $\boldR^1\Gamma({\mathbf W}, \boldD_{\mathbf W}) \otimes_{A} A/\fp_{k}=\{0\}$ as well. Observe that the complex $\boldR\Gamma({\mathbf W},\boldD_{\mathbf W})$ of $A$-modules is perfect with amplitude $[1,2]$, thanks to  Part (ii) combined with (CTRL-1) in \cite[Proposition 5.11]{factorization_alge_p_adic}  and  \cite[\href{https://stacks.math.columbia.edu/tag/066W}{Lemma 066W}]{stacks-project}. Moreover, since $A$ is a PID, it follows that
     $$\boldR\Gamma({\mathbf W},\boldD_{\mathbf W}) \simeq [A^r \longrightarrow A^r]\,\quad \hbox{ in degrees 1 and 2}\,.$$
   This shows $\boldR^1\Gamma({\mathbf W},\boldD_{\mathbf W})$ is also free of some finite rank $r_0$. On the other hand,  we have seen above that $\boldR^1\Gamma({\mathbf W}, \boldD_{\mathbf W}) \otimes_{A} A/\fp_{k}=0$. This shows that $r_0$ must be zero, so that $\boldR^1\Gamma({\mathbf W},\boldD_{\mathbf W})=0.$

By Theorem~\ref{thm:Prop_Iw_SelComp}(iii), we have 
    \begin{equation}
    \label{eq:gamma_0_co-invari_H^1_overH_E_2}
    \boldR^1\Gamma(\overline {\mathbf W}, \overline \boldD_{\mathbf W})_{\Gamma_{\QQ}^0} \subset \boldR^1\Gamma({\mathbf W},\boldD_{\mathbf W})=0.
    \end{equation}
    Let us take an arbitrary classical point $\p \in X (E_\p)$, where $E_\p$ is a finite extension of $\QQ_p$. By (CTRL-2) in \cite[Proposition 5.11]{factorization_alge_p_adic}, we have
    $$\boldR^1\Gamma(\overline {\mathbf W}, \overline \boldD_{\mathbf W}) \otimes_{A} A/\fp \hookrightarrow \boldR^1\Gamma(\overline{{\mathbf W}/\fp {\mathbf W}}, \overline{\boldD_{\mathbf W}/\fp\boldD_{\mathbf W}}).$$ 
    We note that the $\cH_{E_{k}}(\Gamma_{\QQ}^0)$-module $\boldR^1\Gamma(\overline{{\mathbf W}/\fp {\mathbf W}}, \overline{\boldD_{\mathbf W}/\fp\boldD_{\mathbf W}})$ is free of finite rank thanks to the structure theorem for co-admissible $\cH_{E_{\p}}(\Gamma_{\QQ_p}^0)$-modules (cf. \cite{pottharst}, Proposition 1.1) combined with the fact that 
    $\boldR^3\Gamma(\overline{{\mathbf W}^*(1)/\fp {\mathbf W}^*(1)}, \overline{\boldD_{\mathbf W}^\perp/\fp\boldD_{\mathbf W}^\perp})=0$
    (where this vanishing follows from the control theorem and Part (ii)). 
    Hence, $\boldR^1\Gamma(\overline {\mathbf W}, \overline \boldD_{\mathbf W}) \otimes_{A} A/\fp$ is also a free $\cH_{E_{\p}}(\Gamma_{\QQ}^0)$-module with some finite rank $r_1 \geq 0$. Therefore, the module
    \begin{equation}
    \label{eq:co-inavriant_2}
    \left(\boldR^1\Gamma(\overline {\mathbf W}, \overline \boldD_{\mathbf W}) \otimes_{A} A/\fp\right)_{\Gamma_{\QQ}^0} \simeq \boldR^1\Gamma(\overline {\mathbf W}, \overline \boldD_{\mathbf W})_{\Gamma_{\QQ}^0} \otimes_{A} A_1/\fp
    \end{equation}
is an $E_\p$-vector space of dimension $r_1$. But, by \eqref{eq:gamma_0_co-invari_H^1_overH_E_2}, we have $\boldR^1\Gamma(\overline {\mathbf W}, \overline \boldD_{\mathbf W})_{\Gamma_{\QQ}^0}=0$, hence $r_1=0.$  We conclude that the rank of the free $\cH_{E_{\p}}(\Gamma_{\QQ}^0)$-module $\boldR^1\Gamma(\overline {\mathbf W}, \overline \boldD_{\mathbf W}) \otimes_{A} A/\fp$ equals $0$ for every $\fp$ as above. We apply once again \cite[Lemma 7.7]{factorization_alge_p_adic} to see that $\boldR^1\Gamma(\overline {\mathbf W}, \overline \boldD_{\mathbf W})=0$, as required.

 The proof that $\boldR^1\Gamma(\overline{{\mathbf W}^*(1)}, \overline{\boldD_{\mathbf W}^\perp})=0$ is similar, and the proof of our theorem is now complete.
\end{proof}

\begin{theorem}[Functional equation for $1$-parameter families]
\label{thm:FuncEq_SelCom_Rankin_SelbergProducts_1_vari}
In the setting of Theorem~\ref{thm:thm:H^1_zero_torsion_H^2_ranking_sel_ColemanFamily_1_vari}(iii), we have
$${\rm char}_{\cH_{A}(\Gamma_{\QQ}^0)}\,\left(\boldR^2\Gamma(\overline{{\mathbf W}^*}(1), \overline{\boldD_{\mathbf W}}^\perp)\right) = {\rm char}_{\cH_{A}(\Gamma_{\QQ}^0)}\,\left(\boldR^2\Gamma(\overline {\mathbf W}, \overline \boldD_{\mathbf W})\right)^\iota\,.$$ 
where ${\rm char}_{\cH_{A}(\Gamma_{\QQ}^0)}$ denotes the cylindrical characteristic ideal of an ${\cH_{A}(\Gamma_{\QQ}^0)}$-module given as in Definition~\ref{defn:char_Ideal_coadmissible_H_A_module}.
\end{theorem}
\begin{proof}
\label{proof:thm:FuncEq_SelCom_Rankin_SelbergProducts_1_vari}
This is a special case of Theorem~\ref{thm:char_ideals_for_selCom_and_dual_selCom_H_A}. Note that we have verified in \S\ref{sec:Triangulations_in_families_1_1} (see especially Proposition~\ref{prop:Irr_V_0_with_item_Irr_res_F} and the proof of Theorem~\ref{thm:thm:H^1_zero_torsion_H^2_ranking_sel_ColemanFamily_1_vari}) that all the assumptions required in Theorem~\ref{thm:char_ideals_for_selCom_and_dual_selCom_H_A} hold in this scenario.

\end{proof}

\subsubsection{2-parameter families}
\label{sec:Triangulations_in_families_2}
Let us now consider the case $A=\cO(X_1) \,\widehat{\otimes}_E\,\cO(X_2)\simeq E\langle X_1,X_2\rangle$\,, where $k_{f}\in X_1 \subset U_1$ and $k_{g}\in X_2 \subset U_2$ are standard affinoids over $E$ contained in the wide-open disc $U_1$ and $U_2$ respectively. Let $\f$ and $\g$ be Coleman families over $X_1$ and $X_2$, passing through the eigenforms $f_{\alpha_f}$ and $g_{\alpha_g}$, respectively. Let us consider the corresponding big Galois representation 
$\boldV:=V^*_{\f}\, \widehat{\otimes}_{E}\, V^*_{\g}$ over $A$. Let us put $\boldD:= \mathcal{F}^+D^\dagger_{\rm rig}(V^*_{\f}) \,\widehat{\otimes}\,  D^\dagger_{\rm rig}(V^*_{\g})$, which is a $(\varphi,\Gamma_{\QQ_p})$-submodule of $D^\dagger_{\rm rig}(\boldV)$ of rank $2$. As usual, we define $\boldD^\perp$ to denote the dual triangulation at $p$ for $\boldV^*(1)$, which is given by 
$${\rm Hom}_{\cR_{\cO(X_1)}}\left(D^\dagger_{\rm rig}(V^*_{\f})/\mathcal{F}^+D^\dagger_{\rm rig}(V^*_{\f})\,,\,\cO(X_1)\right)\,\widehat{\otimes}_{E}\, D^\dagger_{\rm rig}(V_{\g}(1))\subset D^\dagger_{\rm rig}(\boldV^*(1))\,.$$

\begin{lemma}
\label{lemma:krulldim_2_variable_Cole}
The Noetherian ring $A\widehat{\otimes}_E \LL_E$ is regular of Krull dimension $3$.
\end{lemma}
\begin{proof}
This is a special case of  Proposition~\ref{prop_2024_05_21_1411}(iii).
\end{proof}

\begin{lemma}
\label{lemma:Irr_V_with_item_Irr_res_F_G}
We have 
    \begin{equation}
    \label{eqn_2024_05_21_1521_bis}
        H^0(H_{\QQ,S},(\boldV\,\widehat\otimes_{E}\, \Lambda_{E}^\iota)/\fp)=0= H^0(H_{\QQ,S},(\boldV^*(1)\,\widehat\otimes_{E}\, \Lambda_{E}^\iota)/\fp)
        \end{equation}
    for each prime ideal $\fp$ of $A\widehat{\otimes}_E \LL_E$.
\end{lemma}

\begin{proof}
\label{proof:lemma:Irr_V_with_item_Irr_res_F_G}

The proof of this lemma is identical to the proof of Proposition~\ref{prop:Irr_V_0_with_item_Irr_res_F} (but requires one more reduction step, since the Krull dimension of $A\widehat{\otimes}_E \LL_E$ is one more than the same of $A_1\widehat{\otimes}_E \LL_E$).

\end{proof}

In Theorem~\ref{thm:H^1_zero_torsion_H^2_ranking_sel_ColemanFamily}, which is one our main result in \S\ref{sec:Triangulations_in_families_2}, we will require the following ``Tamagawa vanishing'' condition on the family $\boldV$ (see Proposition~\ref{prop:prop4.37_on_artin_Iawasawa} for a justification of this terminology):
\begin{enumerate}
 \item[(\mylabel{item_TamBigV}{$\,{\textbf{Tam}_{\boldV}}\,$})] $H^1(I_\fq, \boldV)$ and $H^1(I_\fq, \boldV^*)$ are projective $A$-modules for $\fq \in \Sigma_p.$
\end{enumerate}

\begin{theorem}
\label{thm:H^1_zero_torsion_H^2_ranking_sel_ColemanFamily}
Suppose that \eqref{item_TamBigV} holds. Under the Assumptions~\ref{hypo_ranking_sel_1}, we have:

    \item[i)] $\boldR^3\Gamma(\overline{\boldV}, \overline{\boldD})=0=\boldR^3\Gamma(\overline{\boldV^*}(1), \overline{\boldD}^\perp).$
    \item[ii)] $\boldR\Gamma(\overline{\boldV}$, $\overline{\boldD})$, $\boldR\Gamma(\overline{\boldV^*}(1), \overline{\boldD}^\perp) \in \cD^{[1,2]}_{\rm perf}(\cH_A(\Gamma_{\QQ}^0)).$
    \item[iii)] $\boldR^1\Gamma(\overline{\boldV}, \overline{\boldD})=0=\boldR^1\Gamma(\overline{\boldV^*}(1), \overline{\boldD}^\perp).$
\end{theorem}

\begin{proof}
\label{proof:thm:H^1_zero_torsion_H^2_ranking_sel_ColemanFamily}

    \item [i)] Let us onsider a classical point $\frakm_{\underline{k}} \in {\rm Spm}(A)$ corresponding to a pair $\underline{k}=(k_1,k_2)$ of classical weights. By the control theorem for Selmer complexes (the statement (CTRL-2) in \cite{factorization_alge_p_adic}, Proposition 5.11), we have 
    $$\boldR^3\Gamma(\overline{\boldV}, \overline{\boldD}) \otimes_{A} A/\frakm_{\underline{k}} \simeq \boldR^3\Gamma(\overline{\boldV_{\underline{k}}}, \overline{\boldD_{\underline{k}}})\,.$$ 
   where $Z_{\underline{k}}= Z \otimes_{A} A/ \frakm_{\underline{k}}$ for $Z=\boldV,\boldD$. By Theorem~\ref{thm:Vanish_H^1_torsion-ness_of_H^2_for_specialization}(ii), we have $\boldR^3\Gamma(\overline{\boldV_{\underline{k}}}, \overline{\boldD_{\underline{k}}})=0$ for all such $\underline{k}$, and 
    the proof that $\boldR^3\Gamma(\overline{\boldV}, \overline{\boldD})=0$ follows by \cite[Theorem 7.7]{factorization_alge_p_adic}. Similarly, $\boldR^3\Gamma(\overline{\boldV^*}(1), \overline{\boldD}^\perp)=0.$

    \item[ii)] The proof of this portion follows from \cite[Theorem 5.21]{factorization_alge_p_adic} combined with Lemma~\ref{lemma:Irr_V_with_item_Irr_res_F_G}, our running hypothesis \eqref{item_TamBigV}, and Part (i).
  \item[iii)] We will closely follow the proof of \cite[Theorem 7.7]{factorization_alge_p_adic}. By the control theorem for Selmer complexes (cf. \cite{jayanalyticfamilies}, Theorem 1.12) we have \begin{equation}
                \label{eqn_2024_05_22_1152}
                \boldR\Gamma(\boldV,\boldD) \otimes_{A} \cO(X_1) \simeq \boldR\Gamma({\mathbf W}, \boldD_{\mathbf W})\,,
            \end{equation}
where the tensor product is taken with respect to the map induced from ${\rm id}\otimes k_g: A\to \cO(X_1)$, whose kernel we denote by $\fp$. Since ${\fp}$ is a height $1$-prime, we have ${\rm Tor}^{\cH_A(\Gamma_{\QQ}^0)}_{i}(\,-\,,\cH_{\cO(X_1)}(\Gamma_{\QQ}^0))=0$ for $i \geq 2.$ Then by \cite[Proposition 5.12]{factorization_alge_p_adic}, the spectral sequence induced from \eqref{eqn_2024_05_22_1152} gives rise to an exact sequence
\begin{equation}
\label{eqn_2024_5_22_1416}
0\to  \boldR^n\Gamma(\overline{\boldV}, \overline{\boldD}) \otimes_{A} \cO(X_1) \to \boldR^n\Gamma(\overline{\mathbf W}, \overline{\boldD}_{\mathbf W})  \to {\rm Tor}^{\cH_A(\Gamma_{\QQ}^0)}_1(\boldR^n\Gamma(\overline{\boldV}, \overline{\boldD}), \cH_{A_1}(\Gamma_{\QQ}^0)) \to 0\,\quad (\forall n),
\end{equation}
which boils down to the following data: 
\begin{equation}
\label{eq:p-torsion_Iwas_H1}
\boldR^1\Gamma(\overline{\boldV}, \overline{\boldD})[{\fp}]:=0.
\end{equation}
\begin{equation}
\label{eq:eq_2}
0\longrightarrow \boldR^1\Gamma(\overline{\boldV}, \overline{\boldD}) \otimes_{A} {\cO(X_1)}
\lra 
\boldR^1\Gamma(\overline{\textbf W}, \overline \boldD_{\textbf W}) \lra 
\boldR^2\Gamma(\overline{\boldV}, \overline{\boldD})[{\fp}] \longrightarrow 0.
\end{equation}
\begin{equation}
\label{eq:eq_3}
\boldR^2\Gamma(\overline{\boldV}, \overline{\boldD}) \otimes_{A} {\cO(X_1)} \simeq \boldR^2\Gamma(\overline {\mathbf W}, \overline \boldD_{\mathbf W}).
\end{equation}

Indeed, \eqref{eq:p-torsion_Iwas_H1}  follows from the exact sequence \eqref{eqn_2024_5_22_1416} (with $n=0$) combined with Theorem~\ref{thm:thm:H^1_zero_torsion_H^2_ranking_sel_ColemanFamily_1_vari}(ii). We note also that \eqref{eq:eq_3} follows from the the exact sequence \eqref{eqn_2024_5_22_1416} (with $n=2$) and Part (ii). It follows from \eqref{eq:eq_2} that 
$$\boldR^1\Gamma(\overline{\boldV}, \overline{\boldD}) \otimes_{\cH_A(\Gamma_{\QQ}^0)} \cH_A(\Gamma_{\QQ}^0)/{\fp}=0.$$ 
We then proceed as in the final 2 paragraphs of the proof of \cite[Theorem 7.7]{factorization_alge_p_adic}, to prove that $\boldR^1\Gamma(\overline{\boldV}, \overline{\boldD})=0$. The vanishing of $\boldR^1\Gamma(\overline{\boldV^*}(1), \overline{\boldD}^\perp)$ is verified in an identical way.
\end{proof}

\begin{theorem}[Functional equation for $2$-parameter families]
\label{thm:FuncEq_SelCom_Rankin_SelbergProducts_2_vari}
In the situation of Theorem~\ref{thm:H^1_zero_torsion_H^2_ranking_sel_ColemanFamily}, we have
$${\rm char}_{\cH_{A}(\Gamma_{\QQ}^0)}\left(\boldR\Gamma(\overline{\boldV^*}(1), \overline{\boldD}^\perp)\right) = {\rm char}_{\cH_{A}(\Gamma_{\QQ}^0)}\left(\boldR\Gamma(\overline{\boldV}, \overline{\boldD})\right)^\iota\,.$$ 
\end{theorem}
\begin{proof}
\label{proof:thm:FuncEq_SelCom_Rankin_SelbergProducts_2_vari}
As an immediate consequence of Theorem~\ref{thm:duality_selComplex_Iw_iso}, we have
\begin{align*}
    \det{}_{\cH_{A}(\Gamma_{\QQ}^0)}\left(\boldR\Gamma(\overline{\boldV^*}(1), \overline{\boldD}^\perp)\right) =\det{}_{\cH_{A}(\Gamma_{\QQ}^0)}&\left(\boldR\Gamma(\overline{\boldV^*}(1), \overline{\boldD}^\perp)\right)\\
    &= \det{}_{\cH_{A}(\Gamma_{\QQ}^0)}\left(\boldR\Gamma(\overline{\boldV}, \overline{\boldD})\right)^\iota=\det{}_{\cH_{A}(\Gamma_{\QQ}^0)}\left(\boldR\Gamma(\overline{\boldV}, \overline{\boldD})\right)^\iota\,.
\end{align*}
Note that we have verified in \S\ref{sec:Triangulations_in_families_2} (see especially  Theorem~\ref{thm:H^1_zero_torsion_H^2_ranking_sel_ColemanFamily} and its proof) that all the assumptions required in Theorem~\ref{thm:duality_selComplex_Iw_iso} hold in this scenario, thanks to our running hypotheses.
Moreover, it follows from Theorem~\ref{thm:H^1_zero_torsion_H^2_ranking_sel_ColemanFamily} that 
$$\boldR\Gamma(\overline{\boldV}, \overline{\boldD})\stackrel{qis}{\simeq} [P^1\hookrightarrow P^2]\,,\qquad\hbox{in degrees 1 and 2}$$
where $P^1$ and $P^2$ are projective $\cH_A(\Gamma_{\QQ}^0)$-modules of the same rank; and similarly for $\boldR\Gamma(\overline{\boldV^*}(1), \overline{\boldD}^\perp)$. We therefore have 
$$\det{}_{\cH_{A}(\Gamma_{\QQ}^0)}\left(\boldR\Gamma(\overline{\boldV}, \overline{\boldD})\right)={\rm char}_{\cH_{A}(\Gamma_{\QQ}^0)}\left(\boldR^2\Gamma(\overline{\boldV}, \overline{\boldD})\right)\,,$$
$$\det{}_{\cH_{A}(\Gamma_{\QQ}^0)}\left(\boldR\Gamma(\overline{\boldV^*}(1), \overline{\boldD}^\perp)\right)={\rm char}_{\cH_{A}(\Gamma_{\QQ}^0)}\left(\boldR^2\Gamma(\overline{\boldV^*}(1), \overline{\boldD}^\perp)\right) \,,$$
and the proof of our theorem follows.
\end{proof}

\subsection{Selmer complexes in 1-parameter families (bis)}
\label{sec:Triangulations_in_families_1_bis_bis}
Let $\f$ be a Coleman family passing through $f_{\alpha_f}$ as before, over a standard affinoid $X={\rm Spm}(A)$ and let $V_\f$ be the big Galois representation attached to $\f$. Let us put 
$V^*_\f:=V_{\f}(\bbchi_\f)\otimes \chi_\cyc\varepsilon_\f^{-1}\simeq {\rm Hom}_A(V_\f,A)$, where we recall that $\bbchi_\f: G_\QQ \xrightarrow{\chi_{\rm cyc}} \ZZ^\times_p \hookrightarrow (\ZZ_p[[\ZZ^\times_p]])^\times \to A^\times$ is the universal weight character.

We briefly record here the functional equation for the algebraic $p$-adic $L$-function in the supplementary case when $\boldV=\boldV_{\rm Ad^0}:={\rm Ad}^0 V_{\f}$, following \cite{factorization_alge_p_adic} closely.

\subsubsection{} Let $\psi$ be a finite order Dirichlet character with conductor $N_\psi$ coprime to $pN_f$. Let $V_E(f)$ be the corresponding $p$-adic $G_{\QQ,S}$ representation with coefficients in $E$ (which we enlarge so as to ensure that the $\jmath_p({\rm im}(\psi))\subset E$.

\begin{corollary}[Corollary 6.7 in \cite{factorization_alge_p_adic}]
\label{coro:Corollary6.7_factorization_alge_p_adic}
Suppose that $\psi(-1)=(-1)^j$ and $j \in (-1-k_f,0].$ If \eqref{item_BI_Sym} holds and $$L^{\rm imp}_p({\rm Sym}^2(f), \epsilon^{-1}_f\psi^{-1}, -j+k_f+2)L_p(\psi, 1-j) \neq 0,$$ 
then the Bloch-Kato Selmer groups 
$$H^1_{\rm f}(\QQ, {\rm Ad}^0(V_E(f))(j)\otimes\psi)=0=H^1_{\rm f}(\QQ, {\rm Ad}^0(V_E(f))(1-j)\otimes\psi^{-1})$$
vanish.
\end{corollary}

We remark that 
$${\rm Sym}^2(V_E(f)^*)(-k_f-1) \otimes \varepsilon_f^{-1}\simeq {\rm Ad}^0(V_E(f)) \simeq {\rm Sym}^2(V_E(f))(k_f+1) \otimes \varepsilon_f\,,$$
since we have an isomorphism $V_E(f)^*\simeq V_E(f)(k_f+1)\otimes\varepsilon_f$.

\subsubsection{} We let $\Delta: X \longrightarrow X \times X$ denote the diagonal map, which corresponds to the map of affinoid algebras 
$\Delta^*: A \times A \longrightarrow A$ given by $a_1 \otimes a_2 \mapsto a_1a_2$.  We then have the following decomposition:
$$\Delta^*(V_\f \widehat{\otimes}_E V_\f^*)\otimes\psi \simeq {\rm Ad}\,V_\f\,\otimes\psi \simeq {\rm Ad}^0\,V_\f\,\otimes \psi\, \oplus\, A(\psi).$$

\subsubsection{Local Conditions }
For a Galois representation $V$, let us denote by $V(\psi+j)$ the twist $V\otimes\psi\otimes\chi_\cyc^j$. Also for notational simplicity, we set
$$\boldV_{\rm Ad}={\rm Ad}(V_\f)(\psi+j)\,,\quad \boldV_{{\rm Ad}^0}={\rm Ad}^0(V_\f)(\psi+j)\,,\quad \boldV_{\rm Tr}=A(\psi+j)\,.$$ 
We note that these specialize (at the point $x_0\in X$ corresponding to the eigenform $f_{\alpha_f}$)  
$$V_{\rm Ad}={\rm Ad}(V_E(f))(\psi+j)\,,\quad V_{{\rm Ad}^0}={\rm Ad}^0(V_E(f)F)(\psi+j)\,,\qquad \hbox{and } V_{\rm Tr}=E(\psi+j).$$
Let us define
$$D_{\rm Ad}:={\rm Hom}_E(D^\dagger_{\rm rig}(V_E(f), \mathcal{F}^+D^\dagger_{\rm rig}(V_E(f)))(\psi+j)\,,$$ 
which is a is a rank-$2$ $(\varphi, \Gamma_{\QQ_p})$-submodule of $D^\dagger_{\rm rig}(V_{\rm Ad})$ over $\cR_E$. Let us also set 
$$D_{{\rm Ad}^0}:={\rm Hom}_E(D^\dagger_{\rm rig}(V_E(f))/\mathcal{F}^+D^\dagger_{\rm rig}(V_E(f)), \mathcal{F}^+D^\dagger_{\rm rig}(V_E(f)))(\psi+j)$$
and $D_{\rm tr}:={\rm Hom}(\mathcal{F}^+D^\dagger_{\rm rig}(V_E(f)),\mathcal{F}^+D^\dagger_{\rm rig}(V_E(f)))(\psi+j) \simeq \cR_E(\psi+j)\,.$ 

Similarly, we define local conditions at $p$ for $\boldV_{\rm Ad}$, $\boldV_{{\rm Ad}^0}$ and $\boldV_{\rm Tr}$. Namely, we put
$$\boldD_{\rm Ad}:={\rm Hom}_A(D^\dagger_{\rm rig}(V_\f), \boldD)(\psi+j),\quad\boldD_{{\rm Ad}^0}:={\rm Hom}_{A}(D^\dagger_{\rm rig}(V_\f)/\boldD, \boldD)(\psi+j)\,,$$
$\hbox{ and} \quad 
\boldD_{\rm Tr}:={\rm Hom}_{A}(\boldD,\boldD)(\psi+j) \simeq \cR_A(\psi+j)$.

The following is a restatement of \cite[Theorem 7.6]{factorization_alge_p_adic}. 

\begin{theorem}[K\"u\c{c}\"uk]
\label{thm:H^1_zero_torsion_H^2_ranking_sel_SymColemanFamily_1_vari}
Suppose $(\boldV, \boldD_V)\in \{(\boldV_{\rm Ad}, \boldD_{\rm Ad}), (\boldV_{{\rm Ad}^0},\boldD_{{\rm Ad}^0}), (\boldV_{\rm Tr}, \boldD_{\rm Tr})\}.$ 
We assume that $\psi(-1)=(-1)^j$ and $j \in (-1-k_f,0],$ that \eqref{item_BI_Sym} holds, and that $$L^{\rm imp}_p({\rm Sym}^2(f), \epsilon^{-1}_f\psi^{-1}, -j+k_f+2)L_p(\psi, 1-j) \neq 0.$$
Furthermore, we assume the following strengthening of \eqref{item_reg}:
$$\psi(p) \neq p^j, \quad\psi(p)\frac{\beta_f}{\alpha_f} \neq p^j, \quad \frac{\alpha_f}{\beta_f}\psi(p) \neq p^{j-1},\quad\alpha_f\neq \beta_f.$$
Then:
    \item[i)] $\boldR^3\Gamma(\overline \boldV, \overline \boldD_V)=0=\boldR^3\Gamma(\overline{\boldV^*}(1), \overline{\boldD}_V^\perp).$
    \item[ii)] $\boldR\Gamma(\overline\boldV,\overline{\boldD_V})$, $\boldR\Gamma(\overline{\boldV^*}(1), \overline{\boldD}_V^\perp) \in \cD^{[1,2]}_{\rm perf}(\cH_{A}(\Gamma_{\QQ}^0)).$
    \item[iii)] $\boldR^1\Gamma(\overline \boldV, \overline \boldD_V)=0=\boldR^1\Gamma(\overline{\boldV^*}(1), \overline{\boldD}_V^\perp).$ 
\end{theorem}

\begin{theorem}
\label{thm:FuncEq_SelCom_Rankin_SelbergProducts_Sym2}
Suppose we are in the situation of Theorem~\ref{thm:H^1_zero_torsion_H^2_ranking_sel_SymColemanFamily_1_vari}. Then: 
$${\rm char}_{\cH_A(\Gamma_{\QQ}^0)}(\boldR^2\Gamma(\overline \boldV, \overline \boldD_V))={\rm char}_{\cH_A(\Gamma_{\QQ}^0)}(\boldR^2\Gamma(\overline{\boldV^*}(1), \overline{\boldD}_V^\perp))^\iota\,.$$ 
\end{theorem}
\begin{proof}
\label{proof:thm:FuncEq_SelCom_Rankin_SelbergProducts_Sym2}
This is a special case of Theorem~\ref{thm:char_ideals_for_selCom_and_dual_selCom_H_A}, where we note that all the assumptions required in Theorem~\ref{thm:char_ideals_for_selCom_and_dual_selCom_H_A} hold in this scenario thanks to Theorem~\ref{thm:FuncEq_SelCom_Rankin_SelbergProducts_Sym2}.
\end{proof}

\appendix

\section{Irreducibility of tensor products of Galois representations}
\label{sec:Irreducibility}
In this Appendix, we briefly review \cite[\S1-3]{Loff_IMA_OF_ADELIC_GALOIS}, and deduce an irreducibility result that is useful for our purposes. Our notation in this portion is identical to that in \cite[\S1-3]{Loff_IMA_OF_ADELIC_GALOIS}. In particular, we let $L_f$ and $L_g$ denote the Hecke fields of $f$ and $g$, respectively. We also recall from op. cit. the quaternion algebras $B_f$, $B_g$ over $L_f$ and $L_g$ respectively, and the algebraic groups $G_f$, $G_g$ (see the statement of Theorem~\ref{thm:Irr_THEOREM 3.2.2_Loff_IMA_OF} for a description of $\ZZ_p$-points of these algebraic groups). We let $L=L_f \times L_g$ and $B=B_f \times B_g$. There is a natural norm map ${\rm norm}_{B/L}: B^\times \to L^\times.$ We let $\eta: \boldG_m \to {\rm Res}_{L/\QQ}\boldG_m={\rm Res}_{L_f/\QQ}\boldG_m \times {\rm Res}_{L_g/\QQ}\boldG_m$ be the character sending $\lambda$ to $(\lambda^k_f ,\lambda^k_g)$. Then, $G=\{(x,\lambda) \in B^\times \times \boldG_m: {\rm norm}_{B/F}(x)=\lambda^{1-k}\}$ is an algebraic group which is the product of $G_f$ and $G_g$ over $\boldG_m.$ 

The following is \cite[Definition 2.2.1]{Loff_IMA_OF_ADELIC_GALOIS}.

\begin{defn}
\label{defn:DEFINITION 2.2.1_Inner_twist}
An inner twist of $f$ is a pair $(\gamma,\psi)$, where $\gamma:L_f \hookrightarrow \CC$ and $\psi$ is a Dirichlet character, such that the conjugate newform $f^\gamma$ coincides with the twist $f \otimes \psi.$
\end{defn}
It follows from \cite[Lemma 1.5]{Momose_l_adic_reps} that if $(\gamma,\psi)$ is an inner twist of $f$ , then $\psi$ takes values in $L^\times$ and  $\gamma(L) = L$. Thus, the inner twists $(\gamma,\psi)$ of $f$ form a group $\Gamma_f$ with the group law 
$$(\gamma,\psi) (\sigma,\tau)=(\gamma\cdot\sigma, \psi^{\sigma}\cdot\tau).$$ 
Assume that $f$ is not $CM$. Then for any inner twist $(\gamma,\psi) \in \Gamma_f$, the Dirichlet character $\psi$ is uniquely determined by $f$ and $\gamma$, and we denote it as $\psi_{f,\gamma}$. The map $(\gamma,\psi) \mapsto \gamma$ identifies $\Gamma_f$ with an abelian subgroup of ${\rm Aut}(L_f/\QQ)$. We write $F_f\subset L_f$ for the fixed-field of $\Gamma_f$. The extension $L_f/F_f$ is Galois, with Galois group $\Gamma_f$ by \cite[Proposition 1.7]{Momose_l_adic_reps}.

Let $H_f$ be the open subgroup of $G_\QQ$ which is the intersection of the kernels of the Dirichlet characters $\psi_{f,\gamma}$ for $\gamma \in \Gamma_f$, interpreted as characters of $G_\QQ$ in the usual way. We similarly define $H_g$ and set $H=H_f \cap H_g$. 

Fir any prime $p$, we have a Galois representation $\rho_{f,g,p}=\rho_f\otimes\rho_g\,:\,G_\QQ \to G(\ZZ_p)$. The following is a restatement of \cite[Theorem 3.2.2]{Loff_IMA_OF_ADELIC_GALOIS}.
\begin{theorem}
\label{thm:Irr_THEOREM 3.2.2_Loff_IMA_OF}
If \eqref{item_conj} and \eqref{item_nonCM} hold, then for all but finitely many primes $p$ we have $\rho_{f,g,p}(H) = G(\ZZ_p)$, where $G(\ZZ_p)=G_f(\ZZ_p) \times G_g(\ZZ_p)$, and 
$$G_h(\ZZ_p)=\{(x,\lambda) \in {\rm GL}_2(\cO_{L_h}) \times \ZZ^\times_p\} : {\rm det}(x)=\lambda^{1-k}\}\quad  \hbox{ for } h \in \{f,g\}\,.$$
\end{theorem}

\begin{lemma}
\label{lemma:irredu_V}
In the situation of Theorem~\ref{thm:Irr_THEOREM 3.2.2_Loff_IMA_OF}, the Galois representation $V_E(f \otimes g)$ is irreducible as a $G_{\QQ}$-representation. 
\end{lemma}

\begin{proof}
\label{proof:lemma:irredu_V}
By Theorem~\ref{thm:Irr_THEOREM 3.2.2_Loff_IMA_OF}, $\rho_{f,g,p}(H)$ contains ${\rm SL}_2(\ZZ_p)\times {\rm SL}_2(\ZZ_p)$. This concludes the proof of our lemma. 
\end{proof}

Proposition~\ref{prop:Prop_4.2.1_Loff_Ima} below is a restatement of \cite[Proposition 4.2.1]{Loff_IMA_OF_ADELIC_GALOIS}.  We assume $f$ and $g$ have weights $\geq 2$, both $f$ and $g$ are non-CM, and $f$ is not Galois-conjugate to any twist of $g$. We say $\p$ is a good prime if the prime $p$ of $\QQ$ below $\p$ is $ \geq 5$, $p$ is unramified in the quaternion algebra $B$ over $L_f \times L_g$ described above, $p \nmid N_fN_g$, and the conclusion
of Theorem \ref{thm:Irr_THEOREM 3.2.2_Loff_IMA_OF} holds for $p$.

\begin{proposition}
\label{prop:Prop_4.2.1_Loff_Ima}
Let $u \in (\ZZ/N\ZZ)^\times$ be such that $\varepsilon_f(u)\varepsilon_g(u) \neq 1$. Let $\p$ be a good prime, and suppose that $\psi_{f,\gamma} (u) = 1$ for all $\gamma$ in the decomposition group of $\p$ in $\Gamma_f$, and similarly for $\Gamma_g.$ Then there is an element $\tau \in {\rm Gal}(\overline{\QQ}/\QQ(\mu_{p^\infty}))$ such that $V_E(f \otimes g)^*/(\tau-1)V_E(f \otimes g)^*$ has dimension $1$ over $E$.
\end{proposition}

\bibliographystyle{alpha}
\bibliography{references}
\end{document}